\documentclass[10pt]{article} 
\usepackage[preprint]{tmlr}

\usepackage{amsmath,amsfonts,bm}

\def\eqref#1{equation~\ref{#1}}

\def\plaineqref#1{\ref{#1}}

\def\1{\bm{1}}

\DeclareMathAlphabet{\mathsfit}{\encodingdefault}{\sfdefault}{m}{sl}
\SetMathAlphabet{\mathsfit}{bold}{\encodingdefault}{\sfdefault}{bx}{n}

\usepackage{hyperref}
\usepackage{url}
\usepackage{subfigure}

\usepackage[utf8]{inputenc} 
\usepackage[T1]{fontenc}    
\usepackage{hyperref}       
\usepackage{url}            
\usepackage{booktabs}       
\usepackage{amsfonts}       
\usepackage{nicefrac}       
\usepackage{microtype}      
\usepackage{xcolor}         

\usepackage{amsmath}
\usepackage{amssymb}
\usepackage{mathtools}
\usepackage{amsthm}

\theoremstyle{plain}
\newtheorem{theorem}{Theorem}[section]
\newtheorem{proposition}[theorem]{Proposition}
\newtheorem{lemma}[theorem]{Lemma}

\theoremstyle{definition}

\newtheorem{assumption}[theorem]{Assumption}
\theoremstyle{remark}

\usepackage{color}

\def\1{\mathbf{1}}

\newcommand{\dv}[1]{{\textcolor{orange}{}}}
\newcommand{\ls}[1]{{\textcolor{violet}{}}}
\def\a{\alpha}
\def\g{\gamma}
\def\dist{{\rm dist}}
\def\bd{m}

\renewcommand{\plaineqref}[1]{(\ref{#1})}
\renewcommand{\eqref}[1]{(\ref{#1})}

\title{Generalized Smooth Stochastic Variational Inequalities:  Almost Sure Convergence and Convergence Rates}

\author{\name Daniil Vankov \email dvankov@asu.edu
      \AND
      \name Angelia Nedi\'{c} \email angelia.nedich@asu.edu
      \AND
      \name Lalitha Sankar \email lsankar@asu.edu\\
      \addr Arizona State University}

\def\month{09}  
\def\year{2025} 
\def\openreview{\url{https://openreview.net/forum?id=EjqSpbUBWU}} 

\begin{document}

\maketitle

\begin{abstract}
    This paper focuses on solving a stochastic variational inequality (SVI) problem under relaxed smoothness assumption for a class of structured non-monotone operators. The SVI problem has attracted significant interest in the machine learning community due to its immediate application to adversarial training and multi-agent reinforcement learning. In many such applications, the resulting operators do not satisfy the smoothness assumption. To address this issue, we focus on a weaker generalized smoothness assumption called $\alpha$-symmetric. Under $p$-quasi sharpness and $\alpha$-symmetric assumptions on the operator, we study clipped projection (gradient descent-ascent) and clipped Korpelevich (extragradient) methods. For these clipped methods, we provide the first almost-sure convergence results without making any assumptions on the boundedness of either the stochastic operator or the stochastic samples.  We also provide the first in-expectation unbiased convergence rate results for these methods under a relaxed smoothness assumption for $\alpha \leq \frac{1}{2}$.
\end{abstract}

\section{Introduction}
This paper focuses on the stochastic variational inequality (SVI) problem, which consists of finding a point $u^* \in U$, such that 
\begin{equation*}
    \langle F(u^*), u - u^* \rangle \geq 0 \qquad \text{for all } u \in U,
\end{equation*}
where the operator $F(\cdot)$ is specified as the expected value of a stochastic operator $\Phi(\cdot,\cdot): U \times \Xi \rightarrow \mathbb{R}^\bd$, i.e.,
\[F(u) = \mathbb{E}[\Phi (u,\xi)]\qquad\hbox{for all }u\in U,\]
where $\xi \in \Xi$ is a random vector. 
Variational Inequality (VI) problems encompass many practical applications, such as optimization, min-max problems, and multi-agent games. In particular, they play a vital role in modeling equilibrium problems where it's important to capture an interaction between many agents. In machine learning literature, the increasing focus on VIs is due to their relevance to generative adversarial networks (GANs) \cite{gemp2018global, DBLP:conf/iclr/GidelBVVL19}, actor-critic methods \cite{pfau2016connecting}, adversarial training, and multi-agent reinforcement learning \cite{sokota2022unified, kotsalis2022simple}. 
In many such applications, the corresponding operator is defined as an expected value of stochastic or finite sum of operators, which motivates us to study SVIs.
One of the pivotal works~\cite{nemirovski2004prox, juditsky2011solving} on SVIs proposed and studied the celebrated Mirror-Prox method under assumptions on monotonicity and Lipschitz continuity of an operator.\footnote{\noindent The well-studied  Mirror-Prox method~\cite{nemirovski2004prox} has been proven to be optimal for solving VIs under strong monotonicity and Lipschitz continuity assumptions. In fact, this method is the stochastic version of the classical extragradient method.} These assumptions become classical for the analysis of first-order methods for solving SVIs~\cite{beznosikov2022smooth, hsieh2019convergence, hsieh2020explore, loizou2021stochastic}. 

In adversarial and multi-agent training, where the corresponding operator is a gradient of a highly non-linear neural network model, these classical assumptions might not be satisfied. It is well-known that one possible remedy for such non-convergent behavior is in clipping, normalization, or adaptive stepsizes, such as ADAM~\citep{DBLP:journals/corr/KingmaB14}. This effect might be explained by the experiment conducted in~\cite{DBLP:conf/iclr/ZhangHSJ20}. In this work, authors observed that when training deep neural network, the norm of the hessian of the loss function correlates with a norm of a gradient along the optimization trajectory. 

This observation motivated \cite{DBLP:conf/iclr/ZhangHSJ20} to introduce a new and more realistic assumption on the linear growth of the hessian. This led to a great number of works in optimization investigating new assumptions on generalized smoothness and convergence behavior of classical gradient~\citep{li2023convergence}, normalized~\citep{DBLP:conf/icml/00020LL23}, clipped~\citep{koloskova2023revisiting}, and adaptive methods~\citep{wang2023convergence, zhang2024convergence}. Despite this progress in optimization, there are only a few works on generalized smooth min-max~\citep{xiandelving} and VI~\citep{vankov2024icml} problems. This motivates us to delve into investigation of the first-order methods for generalized smooth SVIs.

\subsection{Related work}

\paragraph{Weaker Assumption on SVIs.}

More work has focused on stochastic methods for SVI under more relaxed assumptions to develop and analyze the methods applicable to broader problem classes. In particular, some studies have explored SVIs under pseudo-monotonicity~\cite{kannan2019optimal}, quasi-monotonicity~\cite{loizou2021stochastic}, co-coercivity~\citet{pmlr-v206-beznosikov23a} and quasi-sharpness~\cite{vankov2023last}. \cite{diakonikolas2021efficient} showed that such conditions may not be satisfied even in two played Markov games and introduced the weakest known structured non-monotone assumption. Later, weak Minty SVIs were studied~\cite{DBLP:conf/iclr/PethickFLPC23, choudhury2023single, alacaoglurevisiting} under Lipschitz continuity assumption on the operator.
In our work, we consider the generalized smooth assumption that goes beyond the existing settings.

\paragraph{Normalized and Clipped Methods for SVIs.}
\cite{jelassi2022dissecting} studied the performance of normalized stochastic gradient descent-ascent and ADAM and suggested the crucial role of normalization for training GANs. It is worth noting that, with the right clipping parameters, clipped and normalized step sizes are equivalent up to a constant. Another line of works~\cite{gorbunov2022clipped} focuses on smooth SVIs under heavy-tail noise. Using Lipschitz's continuity of operator and the right choice of clipping parameters, the authors showed a high probability convergence rate for the clipped stochastic Korpelevich method. Recent work~\cite{xiandelving} considered generalized smooth stochastic nonconvex strongly-concave min-max problems and provided $\mathcal{O}(\frac{1}{\sqrt{K}})$ convergence rates for variants of stochastic gradient descent-ascent (SGDA) with normalized stepsizes. Due to the specific structure of the minmax problem and the fact that the gradient of the corresponding function is nonmonotone in one variable and strongly monotone in another, it is difficult to compare this work with ours. Moreover, in this work, the crucial part of the analysis is in the fact that the norm of a gradient can be upper bounded by a function residual. One can not use such bounds in SVIs due to the absence of function values. In our analysis, we develop a new technique to bound the operator norm in almost sure (\emph{a.s.}) sense and in expectation.

\paragraph{Stochastic Analysis of Clipped Methods for Generalized Smooth Optimization.}
In~\cite{DBLP:conf/iclr/ZhangHSJ20}, authors analyzed clipped gradient method under \emph{a.s.} bounded error assumption. \cite{koloskova2023revisiting} showed that the gradient method with standard clipping may not converge to a solution even with small stepsizes. The authors analyzed clipped gradient descent as a biased method and provided a convergence rate for non-convex functions. Later, \cite{li2024convex} developed a new technique allowing to bound stochastic gradients by the function value residual along the optimization trajectory, which helps to find the convergence rate for the gradient with the right choice of stepsizes. In our work, we do not make an assumption on \emph{a.s.} bounded noise and bounded stochastic gradients. Furthermore, we provide not only in-expectations but also \emph{a.s.} convergence of the considered clipped methods.

\paragraph{Contributions.} In light of the existing literature, we consider stochastic VIs with $p$-quasi sharp \emph{generalized smooth} $\alpha$-symmetric operators. We assume a bounded variance of the noise and do not use a restrictive assumption of bounded stochastic operators or bounded samples.
Our key contributions are summarized below (see also Table~\ref{tab:table1}):
\vspace{-0.075in}
\begin{itemize}
    \item We provide the first known analysis of the clipped stochastic projection method (clipped SGDA) for solving stochastic generalized smooth VIs 
    with $p$-quasi sharp and $\a$-symmetric operators. 
    The key feature of our analysis is the use of cleverly chosen clipped {\it stochastic stepsizes} $\gamma_k$. We use two different samples of stochastic the operator, one for clipping stepsizes $\gamma_k$ and another for the direction of the method update. This choice allows us to separate the clipping part from the stochastic error and analyze the method in an unbiased manner. To show \emph{a.s.} convergence, we prove that the series of clipped \emph{stochastic} stepsizes is not summable \emph{a.s.}, i.e. $\mathbb{P}(\sum_{k=0}^{\infty} \gamma_k  = \infty) = 1$. 

    \item We also provide convergence rate for stochastic clipped projection method for $\a \leq 1/2$. For $p=2$ we achieve $\mathcal{O}(k^{-1})$ last iterate convergence. For $p>2$, we show the best iterate convergence rate of $\mathcal{O}(k^{-2(1-q)/p})$, where $1 > q>1/2$ is a parameter of the stepsize choice.

    \item We provide the first known analysis of the stochastic clipped Korpelevich method for solving stochastic generalized smooth VIs
    with $p$-quasi sharp and $\a$-symmetric operators. 
    By reusing clipping stepsizes $\gamma_k$ for both iterates updates $h_k$ and $u_k$, we separate stochastic stepsize from the stochastic error, similar to the projection method analysis. To show \emph{a.s.} convergence, we prove that the series of clipped \emph{stochastic} stepsizes is not summable \emph{a.s.}, i.e. $\mathbb{P}(\sum_{k=0}^{\infty} \gamma_k = \infty) = 1$. 
    
    \item  Moreover, we prove in-expectation convergence rates for the stochastic clipped Korpelevich methods for $\alpha \leq 1/2$.  For $p=2$, we show the last iterate sublinear convergence rate $\mathcal{O}(k^{-1})$. For $p>2$, we show the best iterate convergence rate of $\mathcal{O}(k^{-2(1-q)/p})$, where $1 > q>1/2$ is a parameter of the stepsize choice.
    \item Finally, we present numerical experiments where we compare the performance of the methods with proposed stochastic clipping for different stepsize parameter $q > 1/2$ and quasi-sharpness parameter~$p$.

\end{itemize}
\begin{table*}[h!]
    \begin{center}
    \begin{tabular}{|l|l|l|}
        \hline
          & Stochastic Projection& Stochastic Korpelevich  \\
      \hline
      $p > 0 , \alpha \in (0,1]$ & Asym (Thm \ref{thm-proj-as-converge}) &  Asym (Thm \ref{thm-korpelevich-as-converge}) \\
      
      $p=2 , \alpha \in (0,\frac{1}{2}]$ & $\mathcal{O}\left(\dfrac{D_0}{k^2} + \dfrac{\sigma^2 (C_F + \sigma)^2}{\mu^2 k}\right)$ &
      $\mathcal{O}\left(\dfrac{(C_F + \sigma)^2 K^2 D_0}{\mu^2 k^2} + \dfrac{(\sigma^2+ K_1 \sigma^{2 \alpha}) (C_F + \sigma)^2}{\mu^2 k}\right)$ \\
      $p > 2, \alpha \in (0,\frac{1}{2}]$ &  $\mathcal{O}\left(\dfrac{ (D_0 + \sigma^2)^{2/p} (C_F +\sigma)^{2/p}}{\mu^{2/p} k^{2(1-q)/p}}\right)$ &
      $\mathcal{O}\left(\dfrac{\sigma^{4/p} (D_0 + \sigma^2 + K_1\sigma^{2\alpha})^{2/p} (C_F +\sigma)^{2/p}}{\mu^{2/p} k^{2(1-q)/p}}\right)$ \\
     \hline

    \end{tabular}
    \caption{ 
     Summary of convergence rate results showing the decrease of certain performance measures with the number $k$ of iterations. We use ``Asym" as an abbreviation for asymptotic almost sure convergence results.  For $p$-quasi sharp operators, with $p=2$, and for stochastic projection and Korpelevich methods, the performance measures are $D_k=\mathbb{E}[\dist^2(u_k, U^*)]$ and $D_k=\mathbb{E}[\dist^2(h_k, U^*)]$, respectively.  
     For $p> 2$, the performance measure for both methods is $\bar{D}_k = \mathbb{E}[\dist^2 (\bar{u}_k, U^*)]$, $\bar{u}_k = (\sum_{t=0}^{k} \beta_t)^{-1} \sum_{t=0}^k \beta_t u_t$. The constant $C_F$ denotes the upper bound on $\mathbb{E}[\|F(u_k)\|]$ and $\mathbb{E}[\|F(h_k)\|]$ for stochastic projection and Korpelevich methods, respectively.}
    \label{tab:table1}
    \end{center}
\end{table*}

The rest of the paper is organized as follows. In Section~\ref{sec-prelim}, we provide the assumption on the operator class we consider and  the first-order methods we focus on. In Section~\ref{section-proj} we show the almost sure convergence result of clipped stochastic projection method. In Section~\ref{sec-korp}, we provide \emph{a.s.} convergence results and in-expectation convergence rates for the clipped stochastic Korpelevich method. In Section~\ref{Numerical}, we conduct experiments on solving generalized smooth SVIs and compare the performance of the stochastic clipped projection and Korpelevich method for different problem and stepsize parameters. Section~\ref{Conclusion} concludes our work and presents some further research directions.

\section{Preliminaries}\label{sec-prelim}
In this section, we provide the necessary concepts and assumptions for the considered SVI problem.
We start with a standard definition; the operator $F$ is said to be Lipschitz continuous on a set $U$ if there exists $L > 0$ such that
\begin{equation}
\label{def-Lipschitz}
\|F(u) - F(v)\| \leq L \|u - v\| \quad \hbox{for all } u,v \in U.
\end{equation}
So far, the Lipschitz continuity of the operator was the most common assumption to study SVIs~\cite{nemirovski2004prox, 
DBLP:conf/cdc/YousefianNS14,
YousefianNS2017,
hsieh2019convergence, loizou2021stochastic, alacaoglurevisiting}. 

However, this assumption does not hold in modern deep-learning applications. Based on the experiments on neural network training provided in \cite{DBLP:conf/iclr/ZhangHSJ20}, the norm of Jacobian of the operator correlates with the norm of the operator. Motivated by this observation, \cite{DBLP:conf/iclr/ZhangHSJ20} proposed a new, more realistic, and weaker assumption named $(L_0, L_1)$-smooth:
a differentiable operator $F$ is $(L_0, L_1)$-smooth operator on a set $U$ when the following relation holds
\begin{equation}
\label{def-general-l0l1}
\|\nabla F(u)\| \leq 
L_0 + L_1\|F(u))\|
\quad \text{ for all } u \in U.
\end{equation}

When the operator $F(\cdot)$ is $L$-Lipschitz continuous, it satisfies~\eqref{def-general-l0l1} with $L_0 = L$ and $L_1 = 0$.
Recent work~\cite{DBLP:conf/icml/00020LL23} generalized a class of $(L_0, L_1)$-smooth operators and introduced a new class termed $\alpha$-symmetric. The class of $\alpha$-symmetric operators includes the class of $(L_0, L_1)$-smooth operators and coincides with it when the operator is differentiable for $\a=1$.
Given that the class of $\alpha$-symmetric operators includes $(L_0, L_1)$-smooth and Lipschitz continuous operators,
we focus on this class in our work.

\begin{assumption}\label{asum-alpha} 
Given a convex set $U\subseteq\mathbb{R}^m$, the operator $F(\cdot):U\to\mathbb{R}^\bd$ is $\alpha$-symmetric over $U$, i.e.,
for some $\alpha\in(0,1]$ and $L_0, L_1 \geq 0$, we have for all $u, v \in U$,
\begin{equation}
\label{assum-general-alpha}
\|F(u) - F(v)\| \leq 
\left (L_0 + L_1 \max_{\theta \in (0,1)}\|F(w_{\theta})\|^{\alpha}\right) \|u - v\|,
\end{equation}
where $w_{\theta} = \theta u + (1 - \theta) v$.
\end{assumption}

An alternative characterization of $\a$-symmetric operators has been proved in~\cite{DBLP:conf/icml/00020LL23}, as given in the following proposition.

\begin{proposition}[\cite{DBLP:conf/icml/00020LL23}, Proposition 1]
\label{prop-a} 
Let $U\subseteq \mathbb{R}^m$ be a nonempty convex set and let
$F: U\to\mathbb{R}^m$ be an operator. Then, the following statements hold:
\begin{itemize}
    \item [(a)] $F(\cdot)$ is $\a$-symmetric with $\a \in (0,1)$ and constants $L_0, L_1
\ge 0$ if and only if the following relation holds for all $y, y' \in U$,
\begin{equation}
\begin{aligned}
\label{alpha-property}
\|F(y) - F(y')\| &\leq \|y - y'\| (K_0 + K_1 \|F(y')\|^{\alpha} + K_2 \|y - y'\|^{\a / (1 - \a)}),
\end{aligned}
\end{equation}
where $K_0 = L_0(2^{\a^2 / (1 - \a)} + 1)$, $K_1 = L_1 2^{\a^2 / (1 - \a)} 3^\a$, and $K_2 = L_1^{1 / (1 - \a)} 2^{\a^2 / (1 - \a)} 3^{\a} (1 - \a)^{\a / (1 - \a)} $.

\item[(b)] 
$F(\cdot)$ is $\a$-symmetric with $\a =1 $ and constants $L_0, L_1\ge0$ if and only if the following relation holds for all $y, y' \in U$,
\begin{equation}
\begin{aligned}
\label{alpha-property-1}
&\|F(y) - F(y')\| \leq \|y - y'\| (L_0+ L_1 \|F(y')\|) \exp (L_1 \|y - y'\|).
\end{aligned}
\end{equation}
\end{itemize}
\end{proposition}

Proposition~\ref{prop-a} is useful for our analysis, since it  describes an $\alpha$-symmetric operator by using two points $y, y' \in U$, and bypasses the evaluation of $\max_{\theta\in(0,1)}\|F(w_\theta)\|^\alpha$.
The solution set for the variational inequality problem defined by the set $U$ and operator $F$, denoted $U^*$, is given by
\[U^*=\{u^*\in U\mid \langle F(u^*), u - u^* \rangle \geq  0 \ \ \hbox{for all } u \in U\}.\]

Throughout this paper, we make the following assumption on the set $U$ and the solution set.

\begin{assumption}\label{asum-set}
The set $U\subseteq\mathbb{R}^m$ is a nonempty closed convex set, and the solution set $U^*$ is nonempty and closed. 
\end{assumption}
In our analysis we assume that operator $F$ is $p$-quasi sharp~\cite{vankov2023last}.

\begin{assumption}\label{asum-sharp} The operator $F:U\to\mathbb{R}^\bd$ has a $p$-quasi sharpness property over $U$ relative to the solution set $U^*$, i.e.,
for some $p>0$, $\mu >0$, and for all $u \in U$ and $u^*\in U^*$,
\begin{equation}\label{eq-psharp}
\langle F(u), u - u^* \rangle \geq \mu \,\dist^p(u, U^*).
\end{equation}
\end{assumption}

In the optimization community, such assumptions are called error bounds and are widely studied~\cite{zhou2017unified}. This class of operators encompasses strongly monotone, $p$-monotone~\cite{facchinei2003finite, lin2025perseus}, strongly quasi-monotone~\cite{loizou2021stochastic} and strongly coherent~\cite{song2020optimistic} operators and aligns with the class of operators that satisfy the saddle-point metric subregularity~\cite{DBLP:conf/iclr/WeiLZL21} for $p > 2$. The are many applications where operator satisfy \eqref{eq-psharp}, for example robust learning \cite{wang2023convergence, zarifis2024robustly}, network congestion games~\cite{xiao2025computing}.

Assumption~\ref{asum-sharp} is one of the most general assumptions with the positive inner product between an operator value $F(u)$ and $u - u^*$, which is crucial in our analysis. 

For solving the SVI problem,
we consider stochastic variants of projection and~\cite{korpelevich1976extragradient} methods, where stochastic approximations 
$\Phi(u_k, \xi_k)$ and $\Phi(h_k, \xi_k^1)$ are used, respectively, instead of the directions $F(u_k)$ and $F(h_k)$. The iterates of each of the stochastic methods are defined as follows:

Stochastic projection method:
\begin{equation}
\begin{aligned}
\label{eq-projection-stoch}
u_{k+1} &=  P_{U}(u_k - \g_k \Phi(u_k, \xi_k)),
\end{aligned}
\end{equation}
Stochastic Korpelevich method:
\begin{equation}
\begin{aligned}
\label{eq-korpelevich-stoch}
u_{k} &=  P_{U}(h_k - \g_k \Phi(h_k, \xi_k^1)), \cr 
h_{k+1} &= P_{U}(h_{k} - \g_{k}  \Phi(u_{k}, \xi_k^2)),
\end{aligned}
\end{equation}

where $\{\g_k\}$ is a sequence of  stochastic positive stepsizes, and $u_0, h_0 \in U$ are arbitrary deterministic initial points\footnote{The results easily extend to the case when the initial points are random as long as $\mathbb{E}[\|u_0\|^2]$ and $\mathbb{E}[\|h_0\|^2]$ are finite.}. 
At each operator evaluation of these stochastic methods, a random sample $\xi_k$ is drawn according to the distribution of the random variable $\xi$. 
We assume that the stochastic approximation error $\Phi(u, \xi)-F(u)$ is unbiased and has finite variance, leading to the following formal assumption.
\begin{assumption}\label{asum-samples}
The random sample $\xi$ is such that
for all $u \in U$,
\[\mathbb{E}[\Phi(u, \xi)-F(u)]= 0, \qquad \mathbb{E}[\|\Phi(u, \xi)-F(u)\|^{2}]\le \sigma^{2}.\]
\end{assumption}

Our proof techniques in the following sections can be applied to analyze the \emph{a.s.} convergence and convergence rate of the stochastic Popov~\citep{popov1980modification} method with an appropriate selection of stochastic clipping. However, due to space constraints, we leave this exploration for future research.

\section{Stochastic Clipped Projection Method}
\label{section-proj}
Common approaches to developing convergent methods for generalized smooth optimization and VI problems are normalized or clipping stepsizes. We focus on the latter one and present stepsizes for the stochastic projection method~\eqref{eq-projection-stoch} applied to $\alpha$-symmetric operators:
\begin{equation}
    \label{eq-stepsizes-projection}
    \gamma_k = \beta_k \min \left\{ 1, \frac{1}{\|\Phi(u_k, \xi_k^2)\|}\right\},
\end{equation}
where $\beta_k > 0$ for all $k \geq 0$ and $\xi_k^2$ is a random variable, such that $\xi_k^2$ and $\xi_k$ are independent conditionally on $u_k$.
In other words, at every iteration of the projection method, having $u_k$, two independent samples of the stochastic operator are drawn: (1) $\Phi(u_k, \xi_k)$ for the direction of update and (2) $\Phi(u_k, \xi_k^2)$ for clipping stepsize $\gamma_k$.  We define the sigma-algebra $\mathcal{F}_{k}$ for the method:
\begin{equation}
    \mathcal{F}_k = \{\xi_{0}, \xi_{0}^2, \ldots, \xi_{k}, \xi_{k}^2\}\qquad\hbox{for all }k\ge0,
\end{equation}
with $\mathcal{F}_{-1} = \emptyset$. In the sequel, we provide important results on the behavior of the iterates of the clipped stochastic projection method.

\subsection{Almost sure convergence}

The following lemma establishes a key relation for the iterate sequence $\{u_k\}$ generated by the stochastic projection method with stochastic clipping stepsizes.
Its proof is in Appendix~\ref{sec-lem31}
\begin{lemma}
\label{lemma-proj-as-bound}
    Let Assumptions~\ref{asum-alpha},~\ref{asum-set},~\ref{asum-sharp},~\ref{asum-samples} hold, and $\{u_k\}$ be the iterate sequence generated by stochastic projection method~\plaineqref{eq-projection-stoch} with stepsizes $\gamma_k$ defined in~\plaineqref{eq-stepsizes-projection}. Let parameter $\beta_k$ be such that $\sum_{k=0}^{\infty} \beta_k = \infty$ and $\sum_{k=0}^{\infty} \beta_k^2 < \infty$. Then, the following relation holds almost surely for all $k \geq 0$,
    \begin{equation}
    \label{eq-lemma-projection}
         \mathbb{E}[\|u_{k+1}-u^*\|^2 \mid \mathcal{F}_{k-1} \cup \xi_k^2] \leq \|u_k -u^*\|^2 - 2 \mu  \g_k  \dist^p(u_k, U^*)  + 3 \beta_k^2 (2 \sigma^2 +1).
    \end{equation}
Furthermore, almost surely, we have
\begin{equation}
\label{eq-lemma-projection-2}
\sum_{k=0}^{\infty} \gamma_k \, \dist^p(u_k, U^*) < \infty,
\end{equation}
and the sequence $\{ \|u_k - u^*\|\}$ is bounded almost surely for all $u^* \in U^*$. 
\end{lemma}

In the conventional analysis of the methods for SVIs with Lipschitz continuous operators, the sequence $\{\gamma_k \}$ of stepsizes is deterministic and such that $\sum_{k=0}^\infty \gamma_k=\infty$. In our case, $\gamma_k$ is a random variable, and to show \emph{a.s.} convergence we have to show that the series $\sum_{k=0}^{\infty}\{ \gamma_k \}$ is not summable. We do so, providing a sequence of lower bounds for the series and by showing that random variable $\|F(u_k)\|$ is  \emph{a.s.} upper bounded for all $k\geq 0$ and constructing. Moreover, we separate the series into \[\sum_{k=0}^{\infty} \gamma_k = \sum_{k=0}^{\infty}  (\gamma_k - S_k) + \sum_{k=0}^{\infty}  S_k, \]
where $\{S_k\}$ is a convergent martingale. In the next theorem, we present the first results on \emph{a.s.} convergence of the stochastic projection method. 

\begin{theorem}
\label{thm-proj-as-converge}
Let Assumptions~\ref{asum-alpha},~\ref{asum-set},~\ref{asum-sharp}, and~\ref{asum-samples} hold, 
and $\{u_k\}$ be the iterate sequence generated by stochastic projection method~\plaineqref{eq-projection-stoch} with stepsizes $\gamma_k$ defined in~\plaineqref{eq-stepsizes-projection}. Let parameter $\beta_k$ be such that  $\sum_{k=0}^{\infty} \beta_k = \infty$ and $\sum_{k=0}^{\infty} \beta_k^2 < \infty$.
Then, the iterates $u_k$ converge almost surely to a point $\bar u$ such that $\bar u \in U^*$ almost surely.
\end{theorem}
The proof of Theorem~\ref{thm-proj-as-converge} can be found in Appendix~\ref{sec-thm32}. Notice that in an unconstrained setting ($U = \mathbb{R}^m$) according to Theorems 3.1 and 3.2 in~\cite{koloskova2023revisiting}, for any clipping parameters $ \beta >0, c > 0$, there exist a stochastic gradient operator $\nabla f (\cdot, \xi)$ which satisfies Assumptions~\ref{asum-alpha},~\ref{asum-sharp} (with $p=2$), ~\ref{asum-samples} for which there exists a fixed point $\hat{v}$ of a standard clipping with one-sample which there exists  a solution 
\[\mathbb{E}_{\xi}[\beta_k \min \{1, \frac{c}{\|\nabla f(\hat{v}, \xi )\|} \} \hat{v} ] = 0 \quad \text{and} \quad \|\mathbb{E}_{\xi} [\nabla f(\hat{v}, \xi)]\| \geq \sigma^2 / 12,\]
where $c > 0$ is a constant independent from a step sizes parameter $\beta_k$. This observation leads to an unavoidable bias in one-sample clipped SGD~\citep{koloskova2023revisiting}. In contrast, by using two samples in clipped projection method, we overcome this problem and provide \emph{a.s.} convergence to a solution.

\subsection{Convergence rate}
The difficulty of the convergence rate analysis is in the randomness of stepsizes $\gamma_k$. To show in-expectation convergence, we can take a total expectation on both sides of equation~\plaineqref{eq-lemma-projection} of Lemma~\ref{lemma-proj-as-bound}. However, since $\gamma_k$ is a random variable, we have to provide a lower bound on $\mathbb{E}[\gamma_k \dist^p(u_k, U^*)]$. With this goal in mind, in the next lemma we show that the sequence $\{\mathbb{E}[\|F(u_k)\|]\}$ of expected norms  is bounded. The proof of the lemma is in Appendix~\ref{sec-lemma-EF}.
\begin{lemma}
\label{lemma-EF-bound}
    Let Assumption~\ref{asum-alpha}  hold, with $\a \in (0, 1/2]$, Assumptions~\ref{asum-set},~\ref{asum-sharp}, and~\ref{asum-samples}  hold, and let $\{u_k\}$ be the iterate sequence generated by stochastic projection method~\plaineqref{eq-projection-stoch} with stepsizes $\gamma_k$ defined in~\plaineqref{eq-stepsizes-projection}. Let parameter $\beta_k$ be such that $\sum_{k=0}^{\infty} \beta_k = \infty$, and $\sum_{k=0}^{\infty} \beta_k^2 < \infty$. Then, the sequence $\{\mathbb{E}[\|F(u_k)\|]\}$ is bounded by some constant $C_F > 0$.
\end{lemma}
To prove the preceding lemma, we show that the expected norms of the operator are bounded by some constant $C_F$ on the trajectory of the method. Unfortunately, even though the sequence $\{\|F(u_k)\|\}$ is bounded almost surely, it does not imply that $\{\mathbb{E}[\|F(u_k)\|]\}$ is bounded. To show this, we rely on the properties of the method and the generalized smoothness of the operator in Proposition~\ref{prop-a} to obtain that for all $k \geq 0$, and arbitrary solution $v^*$, 
\begin{align}
\label{section-3-F-bound}
\|F(u_k)\| &\leq \|F(u_k) - F(v^*)\| + \|F(v^*)\| \cr
&\leq \|u_k - v^*\| (K_0 + K_1 \|F(v^*)\|^{\alpha} + K_2 \|u_k - v^* \|^{\a / (1 - \a)}) + \|F(v^*)\| .
\end{align}
Notice that by taking an expectation in~\eqref{section-3-F-bound}, the RHS  is undefined for  $\a > 1/2$. For $\alpha \in (0, 1/2]$, using ~\eqref{section-3-F-bound} and boundedness of $\mathbb{E}[\|u_k - v^*\|]$, that follows from taking an expectation in \eqref{eq-lemma-projection},  we achieve the desired bound on $\mathbb{E}[\|F(u_k)\|]$. Using this result, in the next theorem, we provide a convergence rate for the projection method with clipping.
\begin{theorem}
\label{thm-proj-rates}
Let Assumption~\ref{asum-alpha}, with $\a \in (0, 1/2]$, and Assumptions~\ref{asum-set},~\ref{asum-sharp}, and~\ref{asum-samples} hold.
Let
$\{u_k\}$ be the sequence generated by stochastic projection method~\plaineqref{eq-projection-stoch} with stepsizes $\gamma_k$ defined in~\plaineqref{eq-stepsizes-projection}. Let $D_k = \mathbb{E}[\dist^2(u_k, U^*)]$ and $C_F$ be an upperbound on $\mathbb{E}[\|F(u_k)\|]$. Then, we have:

\textbf{Case $p=2$}. Let $\beta_k= \frac{2}{a(2 + k)}$ with $a=\mu \min \left\{1, \frac{1}{2 (C_F+ \sigma)}\right\}$. Then, the following inequality holds
\begin{align}
 D_{k+1} \leq  \dfrac{8  D_0}{  k^2} + \frac{6 (2 \sigma^2 + 1)}{ a^2 k}\qquad\hbox{for all $k \geq 1$}.
\end{align}

\textbf{Case $p\geq 2$}. Let $\beta_k = \frac{b}{(k+1)^q}$, where $1/2<q<1$ and $b > 0$.
Then, the following inequality holds
\begin{align}\label{eq-thm-proj-rates}
\bar{D}_k 
&\le \dfrac{(1 - q)^{2/p} \left(D_0  + 3b^2 (2 \sigma^2 +1 ) / (2 q -1)\right)^{2/p}}{ (ab)^{2/p}\, \left((k+1)^{1 - q} - 2^{1 - q}\right)^{2/p}}
\qquad\hbox{for all $k \geq 1$},
\end{align}
where $\bar{D}_k = \mathbb{E}[\dist^2 (\bar{u}_k, U^*)]$, $\bar{u}_k = (\sum_{t=0}^{k} \beta_t)^{-1} \sum_{t=0}^k \beta_t u_t$,
and $a=\mu \min \left\{1, \frac{1}{2 (C_F+ \sigma)}\right\}$.
\end{theorem}

To derive the convergence rate in terms of $\mathbb{E}[\dist^2(u_k, U^*)]$ we need to relate the progress at each iteration, measured by $\mathbb{E}[\gamma_k \dist^p(u_k, U^*)]$ to $\mathbb{E}[\dist^2(u_k, U^*)]$. Using the boundedness of $\mathbb{E}[\|F(u_k)\|]$, we first bound $\mathbb{E}[\gamma_k \dist^p(u_k, U^*)]$ in terms of $\mathbb{E}[\dist^p(u_k, U^*)]$. Finally, we estimate  $\mathbb{E}[\dist^p(u_k, U^*)]$ through $\mathbb{E}[\dist^2(u_k, U^*)]$ by applying Jensen's inequality, which holds for $p \geq 2$. The proof of Theorem~\ref{thm-proj-rates} is in Appendix~\ref{sec-thm-proj-rates}.

For the simplicity of convergence rate comparison, assume $2 (C_F +\sigma) \geq 1$. Then, from Theorem~\ref{thm-proj-rates} we obtain $\mathcal{O}(\frac{D_0}{k^2} + \frac{\sigma^2 (C_F + \sigma)^2}{\mu^2 k})$ last iterate convergence rate for $p=2$, and $\mathcal{O}(\frac{ (D_0 + \sigma^2)^{2/p} (C_F +\sigma)^{2/p}}{\mu^{2/p} k^{2(1-q)/p}})$ average (or best) iterate convergence rate for  $p>2$ with $q\in(1/2,1)$. It is worth mentioning that obtained rates are unbiased, unlike the analysis in~\cite{koloskova2023revisiting}. However, it comes with the price of two oracle calls per iteration. For $p = 2$, the rate from Theorem~\ref{thm-proj-rates} matches the rate $\mathcal{O}(\frac{1}{k})$ obtained in Theorem 4.3~\citep{loizou2021stochastic} for SGDA under stronger assumption on quasi-strong monotonicity and Lipschitz continuity of the operator. Interestingly, for $p=2$, the parameter $\mu$ appears in the rate as $\frac{1}{\mu^2}$ in both~\citep{loizou2021stochastic} and in our Theorem~\ref{thm-proj-rates}, which is known to be the optimal dependence on $\mu$~\citep{beznosikov2022smooth}.
The rate for $p > 2$ is new in the stochastic case and generalizes the convergence results in deterministic setting~\citep{vankov2024icml}.
From a theoretical perspective, the {clipped projection method and its two-sample} variant differ in several key aspects. In terms of asymptotic convergence, the standard {clipped projection method} suffers from an unavoidable bias, whereas our {clipped projection method} with two samples per iteration enjoys almost sure convergence. For convergence rates, our theorem shows a sublinear rate for the {two-sample} variant, while there are no known results for the standard {clipped projection method} in the setting of stochastic VIs. This makes {the two-sample clipped projection method} more favorable from a theoretical point of view. We also compare the performance of the two methods in the numerical experiments section.

\section{{Stochastic Clipped} Korpelevich Method}\label{sec-korp}
The stepsizes for the stochastic Korpelevich method for $\alpha$-symmetric operators are as given below
\begin{equation}
\label{stepsizes-korpelevich}
    \gamma_k = \beta_k \min \left\{ 1, \frac{1}{\|\Phi (h_{k}, \xi^1_{k})\|} \right\},
\end{equation}
where $\beta_k > 0$ for all $k \geq 0$ and $\xi_k^1$ is a random variable associated with the stochastic approximation $\Phi(h_k, \xi_k^1)$ of $F(h_k)$. We define the sigma-algebra $\mathcal{F}_{k}$ for the method, as follows:
\begin{equation}
\label{eq-clippedstepsize}
    \mathcal{F}_k = \{\xi_{0}^1, \xi_{0}^2, \ldots, \xi_{k}^1, \xi_{k}^2\}\qquad\hbox{for all }k\ge0,
\end{equation}
with $\mathcal{F}_{-1} = \emptyset$. Notice that to obtain $h_{k+1}$ from a point $u_k$, the stepsize $\gamma_k$ clips $\Phi(h_k, \xi_k^1)$, not the stochastic approximation $\Phi(u_k, \xi_k^2)$ of the operator at point $u_k$, i.e., {the update for $h_{k+1}$ in relation~\eqref{eq-korpelevich-stoch} is given by}

\begin{equation}
\begin{aligned}
\label{eq-clippedkm}
    u_{k} &=  P_{U}(h_k - \beta_k \min \left\{ 1, \frac{1}{\|\Phi (h_{k}, \xi^1_{k})\|} \right\}  \Phi(h_k, \xi_k^1)), \cr 
    h_{k+1} &= P_{U} \left(h_k - \beta_k \min \left\{ 1, \frac{1}{\|\Phi (h_{k}, \xi^1_{k})\|} \right\} \Phi(u_k, \xi_k^2)\right).
\end{aligned}
\end{equation}

Thus, sample $\xi_k^2$ is drawn after $\xi_k^1$, and 
$\Phi(h_k, \xi_k^1)$ is completely determined when $\mathcal{F}_{k-1}\cup\{\xi_k^1\}$ is given, thus the stepsize is determined as well. This property of the stochastic Korpelevich method with clipping stepsizes $\gamma_k$ is crucial for further convergence analysis of the method. In the sequel, we provide important results on the behavior of the iterates of {the stochastic clipped} Korpelevich method.

\subsection{Almost sure convergence}
In the forthcoming lemma, we provide some basic relations that hold almost surely for the iterates of the stochastic Korpelevich method with clipped stochastic stepsize.
\begin{lemma}
\label{Lemma-Korpelevich}
Let Assumptions~\ref{asum-alpha},~\ref{asum-set},~\ref{asum-sharp}, and~\ref{asum-samples}  hold. Also, let $\{h_k\}$  and $\{u_k\}$ be iterates generated by stochastic Korpelevich method~\plaineqref{eq-korpelevich-stoch} with stepsizes $\gamma_k$ defined in ~\plaineqref{stepsizes-korpelevich}
and with parameter $\beta_k$ such that $\sum_{k=0}^{\infty} \beta_k = \infty$ and $\sum_{k=0}^{\infty} \beta_k^2 < \infty$. Let $v_k = \|h_k - u^*\|^2 + \frac{1}{2}\|h_{k-1} - u_{k-1}\|^2 + 2 \g_k \mu \dist^p(u_k, U^*) $, then the following relation holds almost surely. 
    \begin{equation}
        \begin{aligned}
        \label{eq-lemma-korpelevich}
         \mathbb{E}[v_{k+1} \mid \mathcal{F}_{k-1}] &\leq v_k - \frac{1}{2}\|h_{k-1} - u_{k-1}\|^2  - 2 \mu \gamma_{k-1} \dist^p(u_{k-1}, U^*) \cr
         &+ 6 \beta_k^2 (  \sigma^2 + C_{e}(\beta_k, \alpha) \sigma^{2 \alpha}) \qquad\hbox{for all $k \geq 0$},
        \end{aligned}
    \end{equation}
where  $C_{e}(\beta_k, \alpha) = 
K_1$, when $\alpha \in (0,1)$, and 
 $C_{e}(\beta_k, \alpha) = \exp (L_1 \beta_k)$, when $\alpha = 1$. Moreover, the following relations hold almost surely,

\begin{equation}
\label{eq-lemma-korpelevich-2}
\sum_{k=0}^{\infty} \gamma_k \dist^p(u_k, U^*) \, < \infty, \quad \sum_{k=0}^{\infty} \|h_k - u_k\|^2 < \infty.
\end{equation}

Furthermore,
the sequence $\{ \|h_k - u^*\|\}$ is bounded almost surely for all $u^* \in U^*$.
\end{lemma}
The proof of Lemma~\ref{Lemma-Korpelevich} is in Appendix~\ref{sec-lem41}.

In the standard analysis of the Korpelevich method for SVIs with Lipschitz operators~\cite{kannan2019optimal, vankov2023last}, \emph{a.s.} convergence results were achieved for a deterministic stepsize sequence $\{\gamma_k\}$. In our case, similarly to projection method analysis, $\{\gamma_k\}$ is a sequence of random variables, which makes the analysis of the methods more difficult and involved. 
By the choice of stepsizes $\gamma_k$ as given in~\eqref{stepsizes-korpelevich} {and the iterated expectation rule}, the following relation holds true
\begin{align}
\label{eq-korpelevich-error}
    \mathbb{E}[\gamma_k \langle \Phi(u_k, \xi_k^2) - F(u_k), u_k - u^* \rangle \mid \mathcal{F}_{k-1}] = 0
    \qquad\hbox{for all }k\ge0.
\end{align}
 To prove~\eqref{eq-korpelevich-error} for the stochastic Korpelevich method, we note that the clipping stepsize is using $\|\Phi(h_k,\xi_k^1)\|$, which decouples from $\Phi(u_k, \xi_k^2)$ by properly using conditional expectation. 
 Specifically, we first take the expectation conditioned on $\mathcal{F}_{k-1} \cup \xi_k^1$ and observe that $\gamma_k$ is {completely determined  given $\mathcal{F}_{k-1} \cup \xi_k^1$. Then, we use Assumption~\ref{asum-samples} for the sample $\xi_k^2$, and the relation~\eqref{eq-korpelevich-error} follows by the law of iterated expectation.} Interestingly, we do not have to take another sample for the clipping in the stochastic Korpelevich method, as we have done in the stochastic projection method. Thus, to perform one iteration, we use two oracle calls in both methods.

Using Lemma~\ref{Lemma-Korpelevich}, we next present the almost sure convergence of the stochastic clipped Korpelevich method.
\begin{theorem}
\label{thm-korpelevich-as-converge}
Let Assumptions~\ref{asum-alpha},~\ref{asum-set},~\ref{asum-sharp}, and~\ref{asum-samples}  hold
and $\{h_k\}$,  $\{u_k\}$ be iterates generated by stochastic Korpelevich method~\plaineqref{eq-korpelevich-stoch} with stepsizes $\gamma_k$ defined in~\plaineqref{stepsizes-korpelevich}. Let parameter $\beta_k$ be such that
$\sum_{k=0}^{\infty} \beta_k = \infty$, and $\sum_{k=0}^{\infty} \beta_k^2 < \infty$.
Then, the iterates $h_k$ and $u_k$ converge almost surely to a point $\bar u$ such that $\bar u \in U^*$ almost surely.
\end{theorem}
To prove \emph{a.s.} convergence, we firstly show that $\sum_{k=0}^\infty \gamma_k = \infty$ \emph{a.s.}, by providing a sequence of lower bounds on $\gamma_k $, using \emph{a.s.} boundedness of $\|h_k - u^* \|$ of Lemma~\ref{Lemma-Korpelevich}, and proving that $\|F(h_k)\|$ is \emph{a.s.} bounded.  Similarly to the proof of \ref{thm-proj-as-converge}, we separate the series into \[\sum_{k=0}^{\infty} \gamma_k = \sum_{k=0}^{\infty}  (\gamma_k - S_k) + \sum_{k=0}^{\infty}  S_k, \]
where $\{S_k\}$ is a convergent martingale.
The full proof can be found in Appendix~\ref{sec-Thm42}. 

\subsection{Convergence rate} 
We start our analysis by taking the total expectation on both sides of equation~\plaineqref{eq-lemma-korpelevich} from Lemma~\ref{Lemma-Korpelevich}. For further analysis, similar to {the stochastic clipped projection method,} the challenge lies in the randomness of the stepsizes $\gamma_k$. To handle this,
firstly, we establish a lower bound for $\mathbb{E}[\gamma_k \dist^p(u_k, U^*)]$ by showing that the sequence $\{\mathbb{E}[|F(u_k)|]\}$ of expected norms  remains bounded, as shown in the next lemma. The proof of the lemma can be found in Appendix~\ref{sec-lemma-EF}.
\begin{lemma}
\label{lemma-korpelevich-EF-bound}
    Let Assumption~\ref{asum-alpha}, with $\a \in (0, 1/2]$, and Assumptions~\ref{asum-set}, ~\ref{asum-sharp}, ~\ref{asum-samples} hold. Let $\{u_k\}$, $\{h_k\}$ be iterates generated by stochastic Korpelevich method~\plaineqref{eq-korpelevich-stoch} with stepsizes $\gamma_k$ defined in~\plaineqref{stepsizes-korpelevich} and the parameter $\beta_k$ such that $\sum_{k=0}^{\infty} \beta_k = \infty$ and $\sum_{k=0}^{\infty} \beta_k^2 < \infty$. Then, $\mathbb{E}[\|F(h_k)\|]$ is bounded by some constant $C_F > 0$ for all $k \geq 0$.
\end{lemma}
Similarly to the analysis presented in Section~\ref{section-proj}, we bound $F(h_k)$ by using a triangle inequality and the property of $\alpha$-symmetric operators, and by taking the total expectation, we obtain
\begin{align}
\label{section-4-F-bound}
\mathbb{E}[\|F(u_k)\|] &\leq K_0 \mathbb{E}[ \|u_k - v^*\|]   + K_2 \mathbb{E}[\|u_k - v^* \|^{\a / (1 - \a)})] + \|F(v^*)\| + K_1 \|F(v^*)\|^{\alpha}.
\end{align}
We can show that the preceding bound has a finite expectation only for $0<\alpha \leq 1/2$, which motivates the restriction on $\alpha$ in Lemma~\ref{lemma-korpelevich-EF-bound}.
Equipped with the boundedness of the sequence $\{\mathbb{E}[\|F(h_k)\|\}$ of expected norms of the operator  along the iterates $\{h_k\}$, we present the next convergence rate theorem.
\begin{theorem}
\label{thm-korpelevich-rates}
Let Assumption~\ref{asum-alpha}, with $\a \in (0, 1/2]$, and Assumptions~\ref{asum-set},~\ref{asum-sharp}, and~\ref{asum-samples} hold.
Let $\{u_k\}$, $\{h_k\}$ be {the iterate sequences} generated by stochastic clipped Korpelevich method~\plaineqref{eq-korpelevich-stoch} with stepsizes $\gamma_k$ defined in~\plaineqref{stepsizes-korpelevich}. Let $D_k = \mathbb{E}[\dist^2(h_k, U^*)]$ and $C_F$ be an upperbound on $\mathbb{E}[\|F(h_k)\|]$ then the following results holds:

\textbf{Case $p=2$}. 
Let $\beta_k= \dfrac{2}{a(\frac{2d}{a} + k)}$, with $a = \mu \min \left\{1, \frac{1}{2 (C_F+ \sigma)}\right\}$, $
d =\max\{4\mu,2\sqrt{3}(K_0 + K_1 + K_2)\}$
where $K_0,K_1,$ and $K_2$ are from Proposition~\ref{prop-a}(a). 
Then, the following relation holds
\begin{align}\label{eq-thm-proj-rates-1}
 D_{k+1} \leq  \dfrac{8  d^2 D_0}{a^2 k^2} + \frac{12(\sigma^2 +  K_1 \sigma^{2 \alpha})}{a^2 k}
 \qquad \hbox{for all $k \geq 1$}.
\end{align}

\textbf{Case $p \geq 2$}. Let $\beta_k = \frac{b}{(k+1)^q}$, where $1/2< q < 1$ 
and $0<b \leq \min \left\{\frac{1}{4 \mu}, \frac{1}{2\sqrt{3}(K_0 + K_1 + K_2)}\right\}$.
Then, the following inequality holds for all $k \geq 1$,
\begin{align}\label{eq-thm-proj-rates-2}
\bar{D}_k
&\le \dfrac{ 2^{2(p-2) / p}(1 - q)^{2/p} \left(D_0  + 6b^2 (\sigma^2 + K_1 \sigma^{2 \alpha})  (2 \sigma^2 +1 ) / (2 q -1)\right)^{2/p}}{ (ab)^{2/p} \,\left((k+1)^{1 - q} - 2^{1 - q}\right)^{2/p}},
\end{align}
where $\bar{D}_k = \mathbb{E}[\dist^2 (\bar{u}_k, U^*)]$, $\bar{u}_k = (\sum_{t=0}^{k} \beta_t)^{-1} \sum_{t=0}^k \beta_t u_t$, and  $a=\mu \min \left\{1, \frac{1}{2 (C_F+ \sigma)}\right\}$.
\end{theorem}
Similarly to the proof of Theorem~\ref{thm-proj-rates}, to derive the convergence rate in terms of $\mathbb{E}[\dist^2(h_k, U^*)]$ we need to relate the progress at each iteration, measured by $\mathbb{E}[\gamma_k \dist^p(h_k, U^*)]$ to $\mathbb{E}[\dist^2(h_k, U^*)]$. Using the boundedness of $\mathbb{E}[\|F(h_k)\|]$ and applying Jensen's inequality, which holds for $p \geq 2$ we obtain the final rates. 
The proof of Theorem~\ref{thm-korpelevich-rates} is provided in Appendix~\ref{sec-thm44}.

For the simplicity of convergence rate comparison, assume $2 (C_F +\sigma) \geq 1$ and $ K_0 + K_1 + K_2 \geq \frac{2\mu}{\sqrt{3}}$. Then by denoting $K = K_0 + K_2 + K_3$, from Theorem~\ref{thm-korpelevich-rates}, we obtain $\mathcal{O}(\frac{(C_F + \sigma)^2 K^2 D_0}{\mu^2 k^2} + \frac{(\sigma^2+ K_1 \sigma^{2 \alpha} )(C_F + \sigma)^2}{\mu^2 k})$ last iterate convergence for $p=2$, and $\mathcal{O}\left(\frac{\sigma^{4/p} (D_0 + \sigma^2 + K_1\sigma^{2\alpha})^{2/p} (C_F +\sigma)^{2/p}}{\mu^{2/p} k^{2(1-q)/p}}\right)$ average (or best) iterate convergence rate for  $p>2$ with $q\in(1/2,1)$.  In both cases $p=2$ and $p>2$ the convergence rate of clipped stochastic projection method in Theorem~\ref{thm-proj-rates} and the rate of stochastic clipped Korpelevich method in Theorem~\ref{thm-korpelevich-rates} have the same dependency in $k$. For $p = 2$ the rate from Theorem~\ref{thm-korpelevich-rates} matches the rate $\mathcal{O}(\frac{1}{k})$ obtained in Proposition 5~\citep{kannan2019optimal} for stochastic Korpelevich method under stronger assumption on strong pseudo monotonicity and Lipschitz continuity of the operator. Interestingly, for $p=2$, the parameter $\mu$ appears in the rate as $\tfrac{1}{\mu^2}$ in both \citet{kannan2019optimal} and in our Theorem~\ref{thm-korpelevich-rates}, which is known to be the optimal dependence on $\mu$~\citep{beznosikov2022smooth}. For $p > 2$, the obtained rate is new in stochastic case and generalizes the results in deterministic setting for Lipschitz continuous operators~\citep{DBLP:conf/iclr/WeiLZL21}  and $\alpha$-symmetric operators~\citep{vankov2024icml}. 

\section{Numerical Experiments}
\label{Numerical}
We study the performance of the stochastic clipped projection and Korpelevich methods, for different values of parameters $\a>0$ and $p>0$. Despite the absence of analysis, we also implement the  stochastic clipped Popov method with $\gamma_k = \beta_k \min \{1, \frac{1}{\|F(h_k)\|}, \frac{1}{(\|u_k - h_{k-1}\| + 1)^{\a / (1 - \alpha)}}\}$:
\begin{equation*}
u_{k+1} =  P_{U}(u_k - \g_k \Phi(h_k, \xi_k)), \quad h_{k+1} = P_{U}(u_{k+1} - \g_{k+1}  \Phi(h_{k}, \xi_k)),
\end{equation*}
where $u_0, h_0 \in U$ are arbitrary deterministic initial.
We consider an unconstrained minmax game:
\[ \min_{u_1} \max_{u_2}  \frac{1}{p} \|u_1\|^p + \langle u_1, u_2 \rangle - \frac{1}{p} \|u_2\|^p ,\]
with  $p>1$, and $u_1 \in \mathbb{R}^{d}, u_2 \in \mathbb{R}^{d}$. Then, the corresponding operator $F: \mathbb{R}^{2d} \rightarrow \mathbb{R}^{2d}$ is defined by
\begin{align}
    \label{eq-example-quasi-sharp}
    F(u) =  \begin{bmatrix} \|u_1\|^{p - 2} u_1 + u_2 \\  \|u_2\|^{p - 2} u_2 -  u_1 \end{bmatrix}.
\end{align}

We assume that we have an access only to a noise evaluation of the corresponding operator and aim to solve unconstrained SVI$(\mathbb{R}^{2d}, F)$ with the following stochastic operator $\Phi(u, \xi) = F(u) + \xi$,
 where $\xi$ is a random vector with independent zero-mean Gaussian entries and with variance $\sigma^2 = 1$. 
 Then,  $F(u) = \mathbb{E}[\Phi(u, \xi)]$ is an $\a$-symmetric and $p$-quasi sharp operator due to~\cite{vankov2024icml}. 
We set these parameters to be $\{(\alpha \approx 0.33 , p = 2.5), (\alpha \approx 0.5, p = 3.0), (\alpha \approx 0.8, p = 6.0)\}$. 
We also compare our results with the projection method that uses the same sample clipping, meaning stepsizes $\gamma_k$ clip $\|\Phi(u_k, \xi_k)\|$ instead of a different sample $\|\Phi(u_k, \xi_k^2)\|$.

\begin{figure*}[hbt!]
\centering
\subfigure[$(\alpha \approx 0.33, p = 1.5)$]{
\includegraphics[width=.23\textwidth]{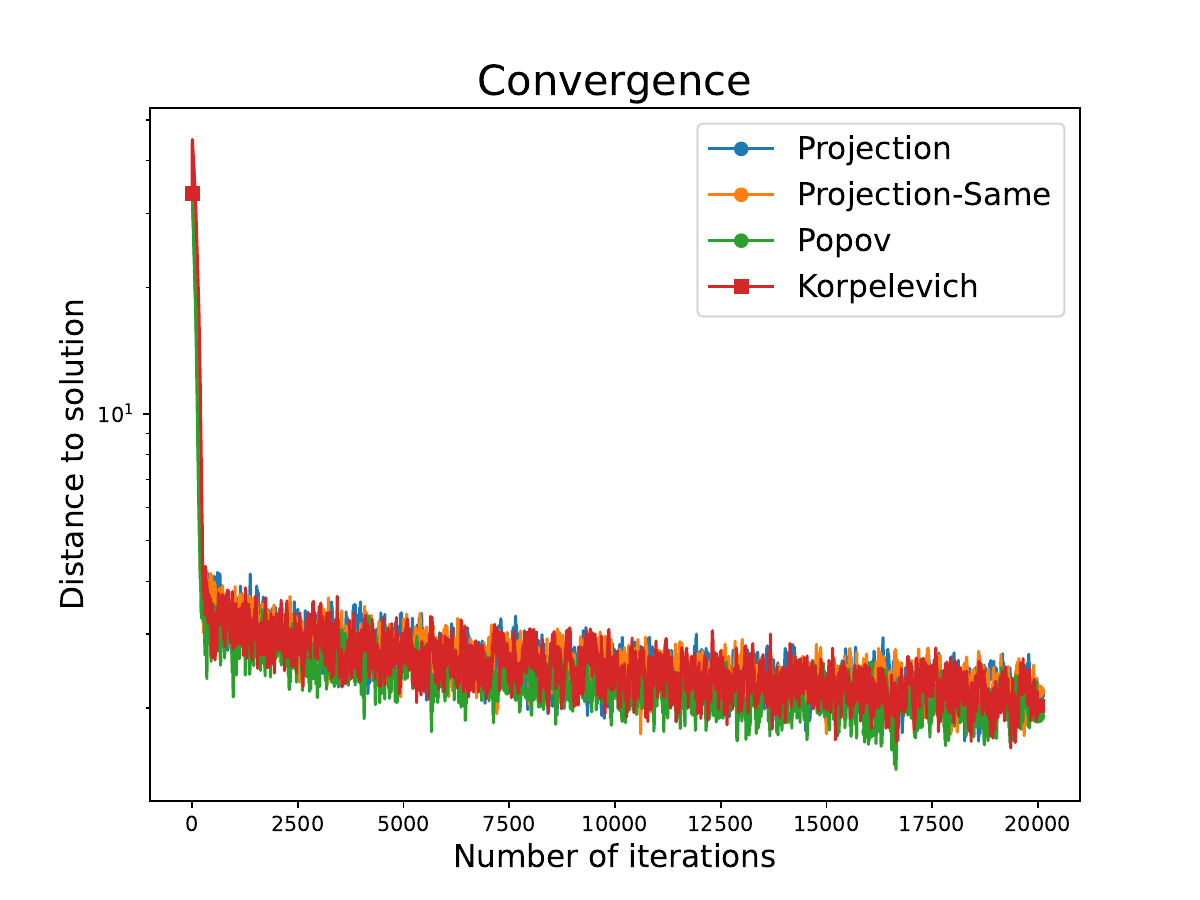}
}
\subfigure[$(\alpha \approx 0.33, p = 2.5)$]{
\includegraphics[width=.23\textwidth]{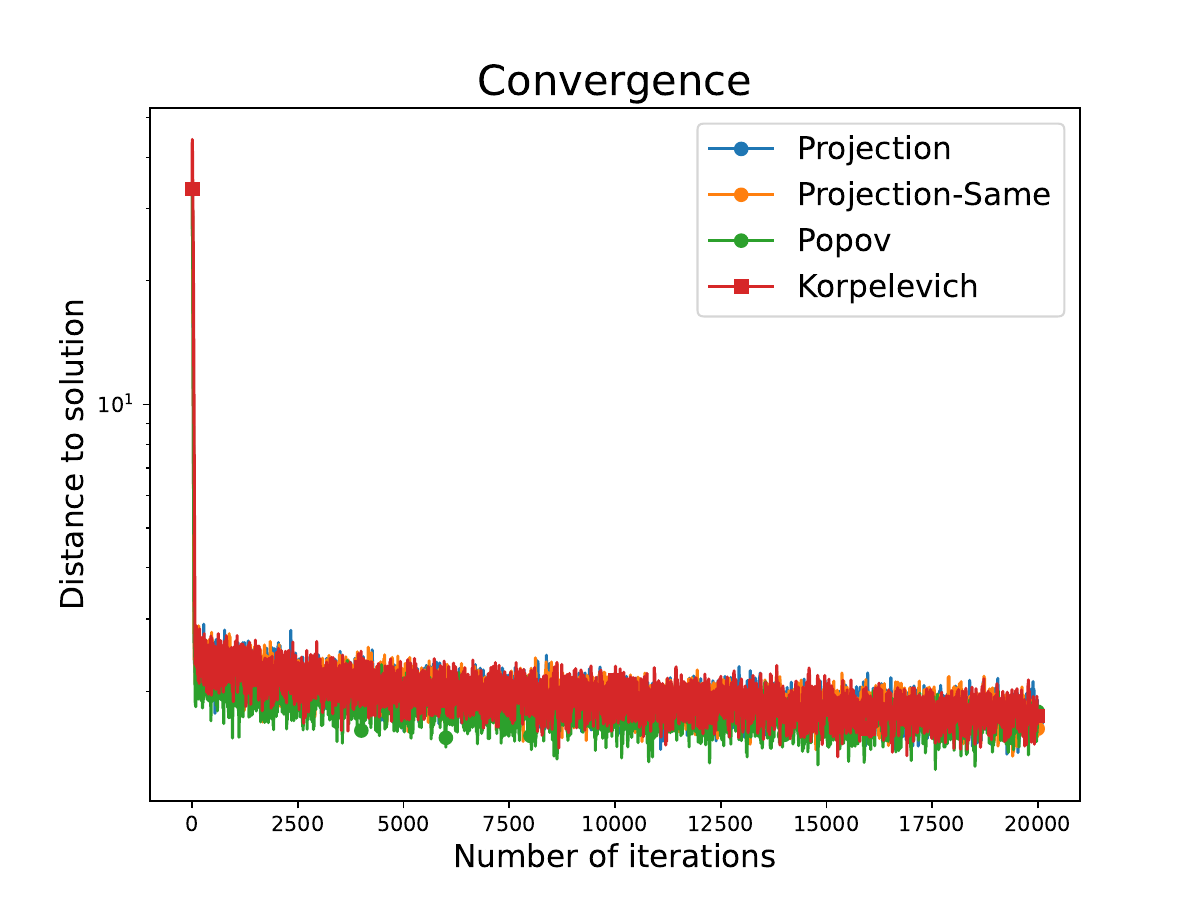}
}
\subfigure[$(\alpha \approx 0.8, p = 4.0)$]{
\includegraphics[width=.23\textwidth]{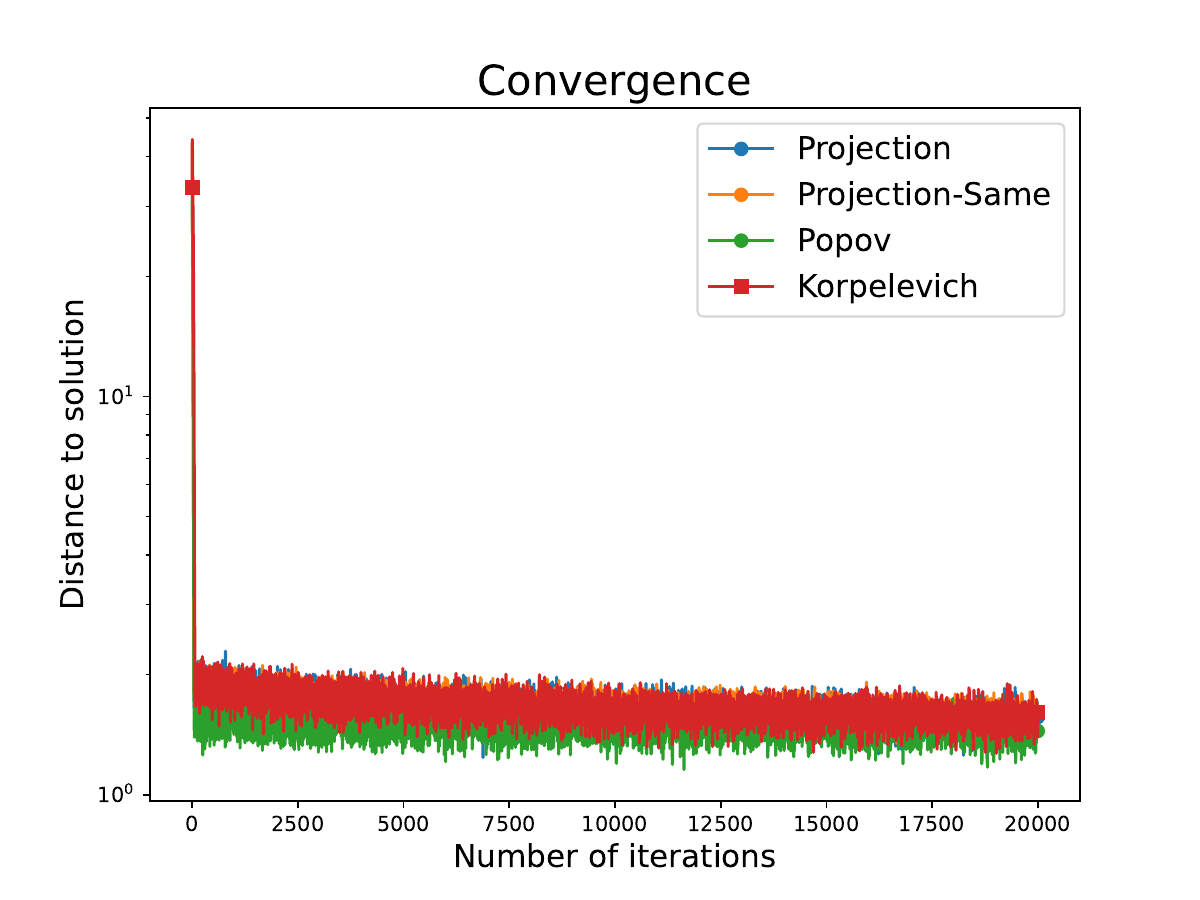}
}
\subfigure[$(\alpha \approx 0.8, p = 6.0)$]{
\includegraphics[width=.23\textwidth]{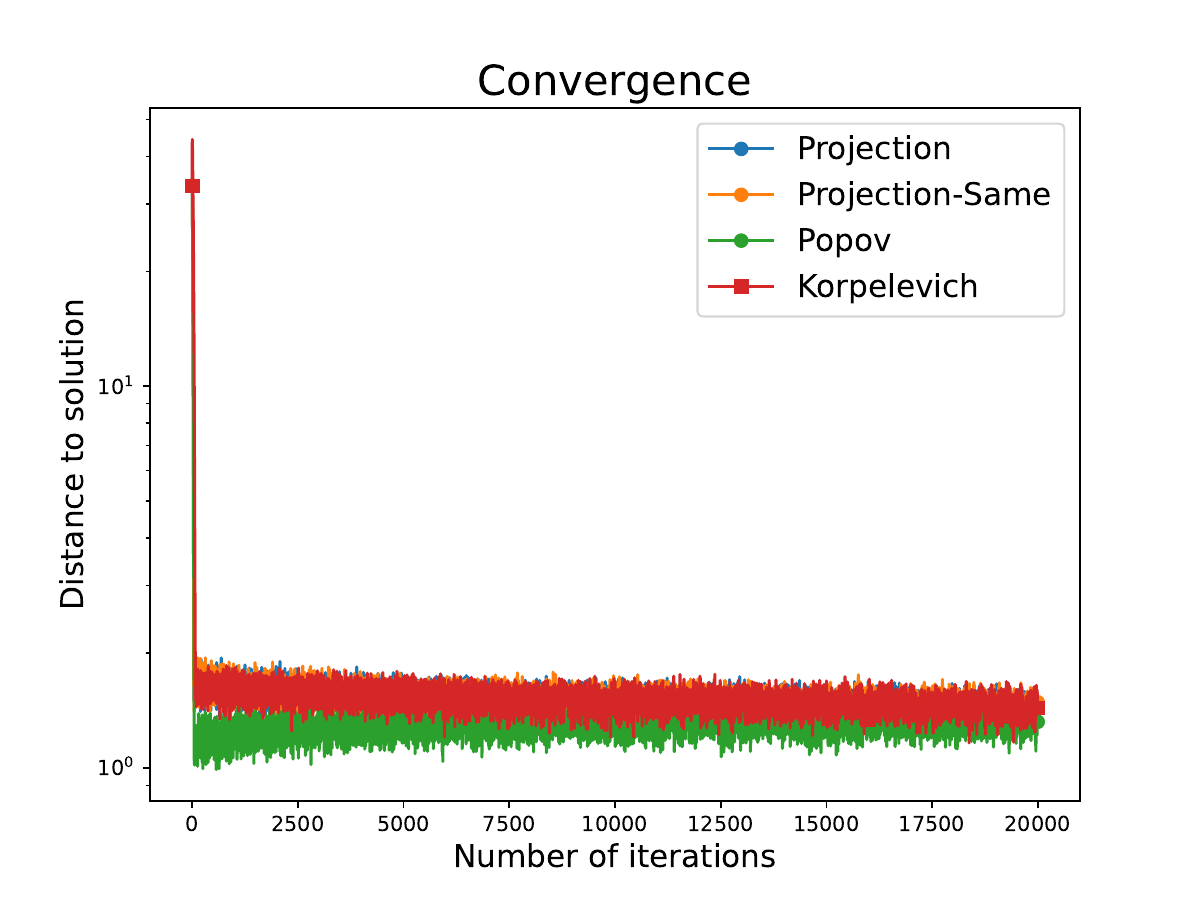}
}
\caption{Comparison of the clipped stochastic projection, same-sample projection, Korpelevich, and Popov methods with $\beta_k = 100/(100 + k^{1/2 + \epsilon})$. 
}
\label{fig:theory}
\end{figure*}

\begin{figure*}[hbt!]
\centering
\subfigure[$(\alpha \approx 0.33, p = 1.5)$]{
\includegraphics[width=.23\textwidth]{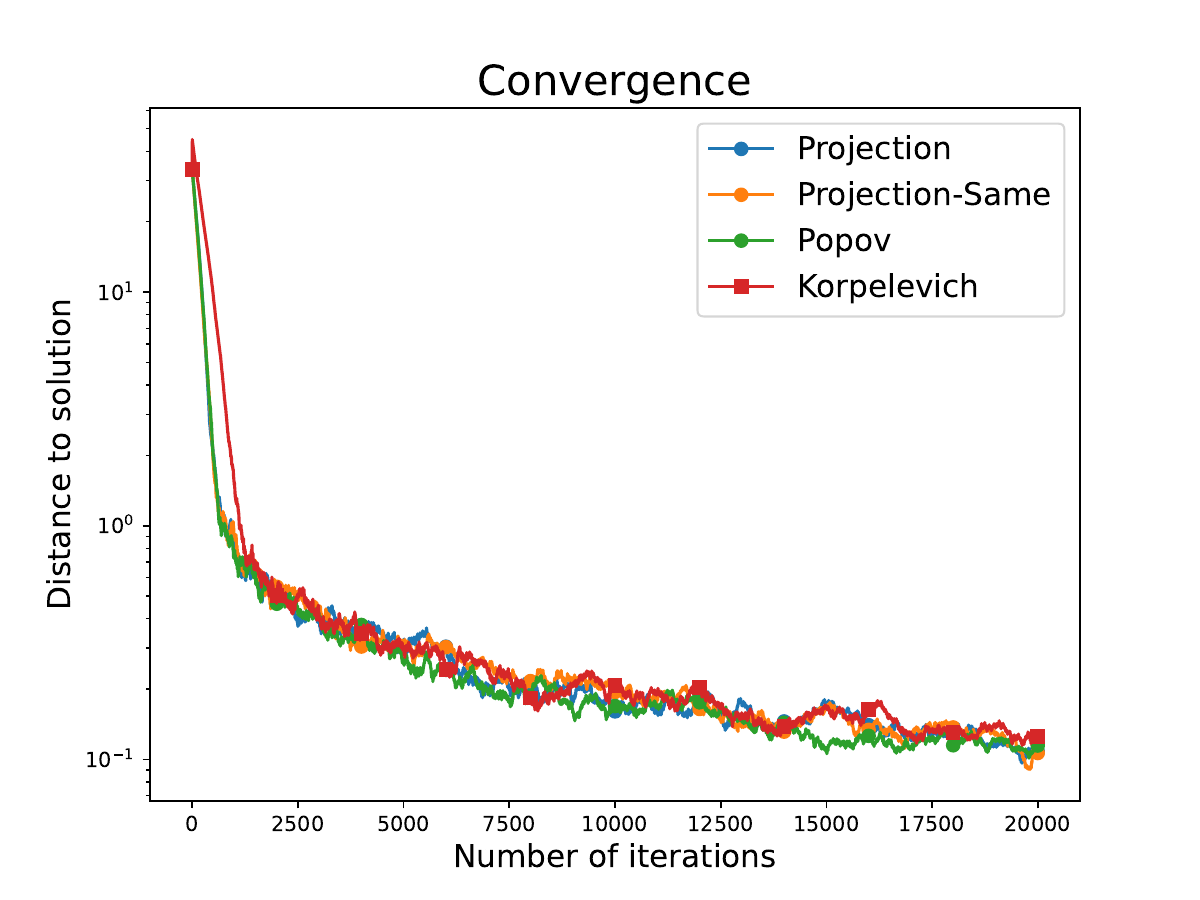}
}
\subfigure[$(\alpha \approx 0.33, p = 2.5)$]{
\includegraphics[width=.23\textwidth]{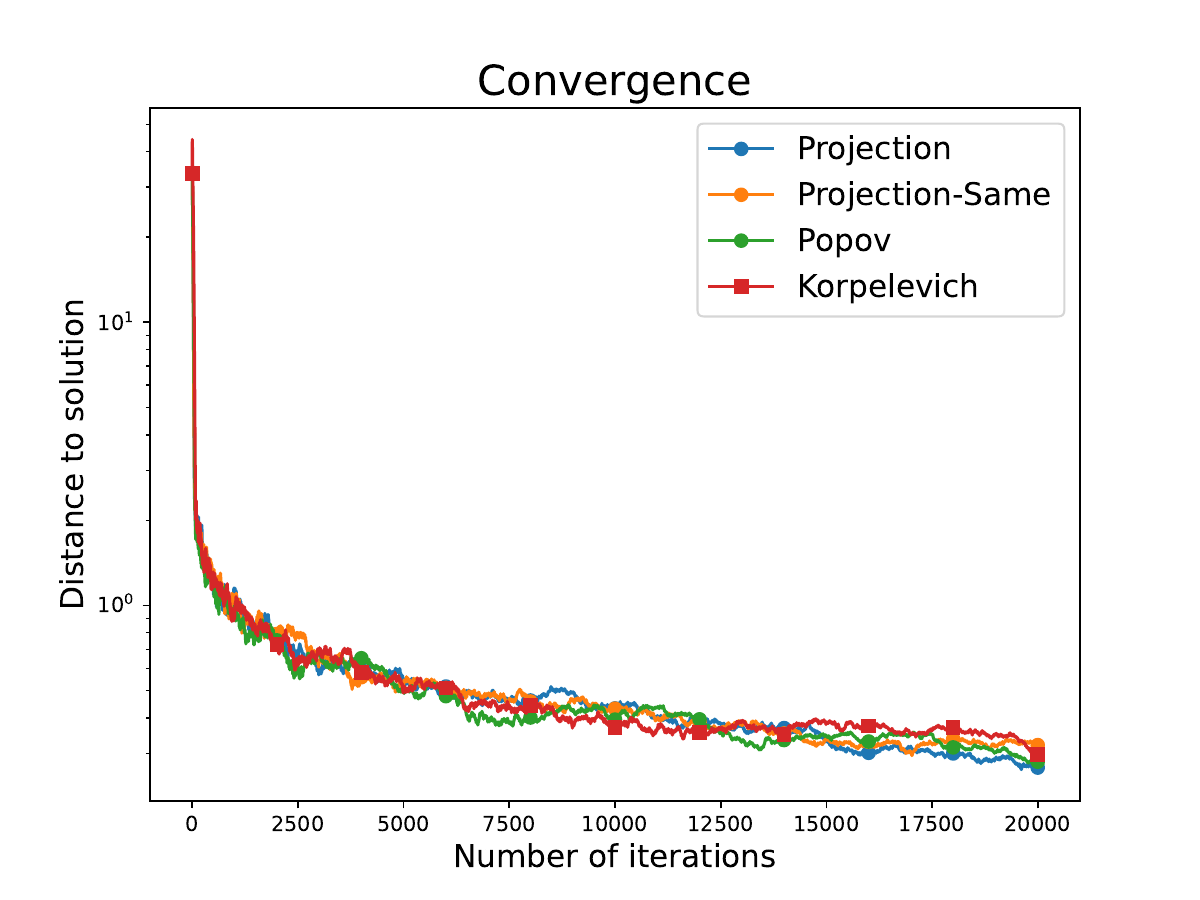}
}
\subfigure[$(\alpha \approx 0.8, p = 4.0)$]{
\includegraphics[width=.23\textwidth]{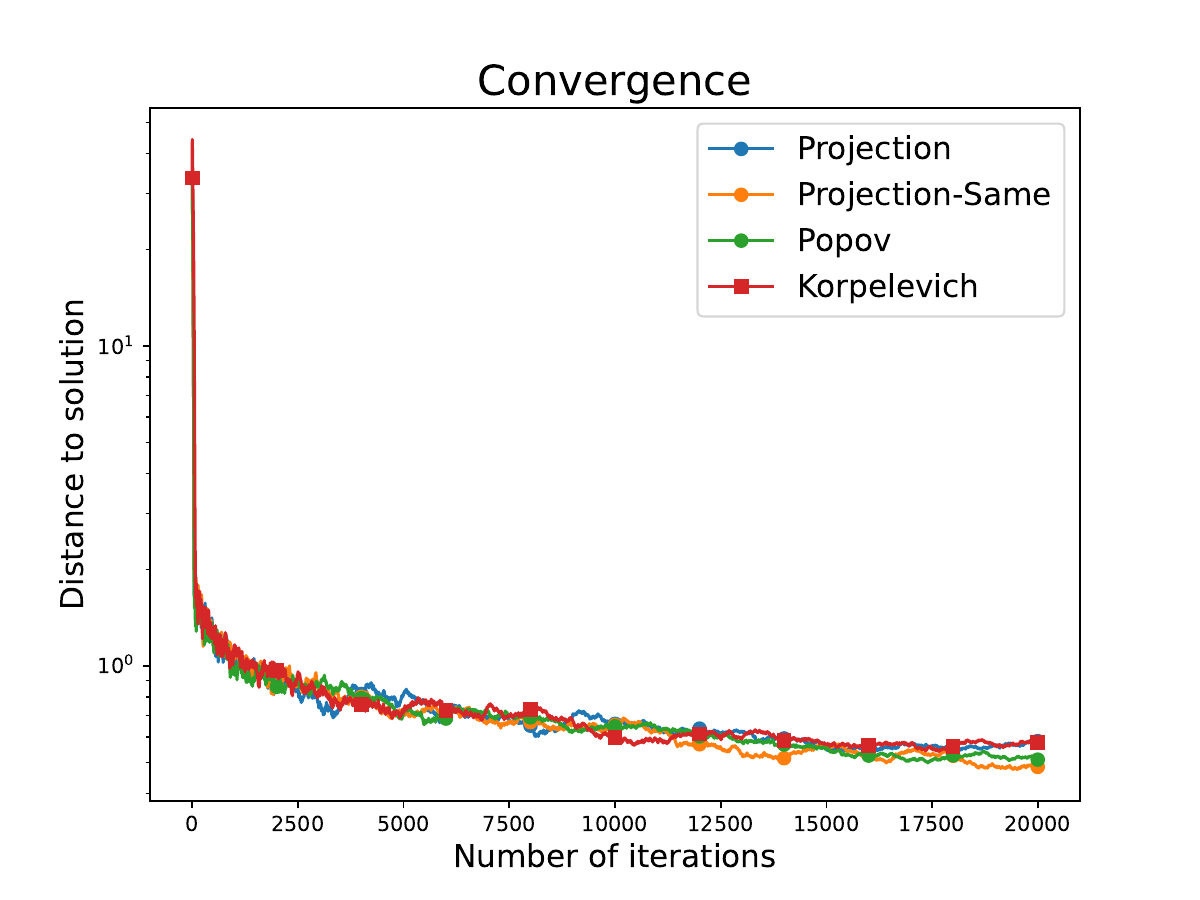}
}
\subfigure[$(\alpha \approx 0.8, p = 6.0)$]{
\includegraphics[width=.23\textwidth]{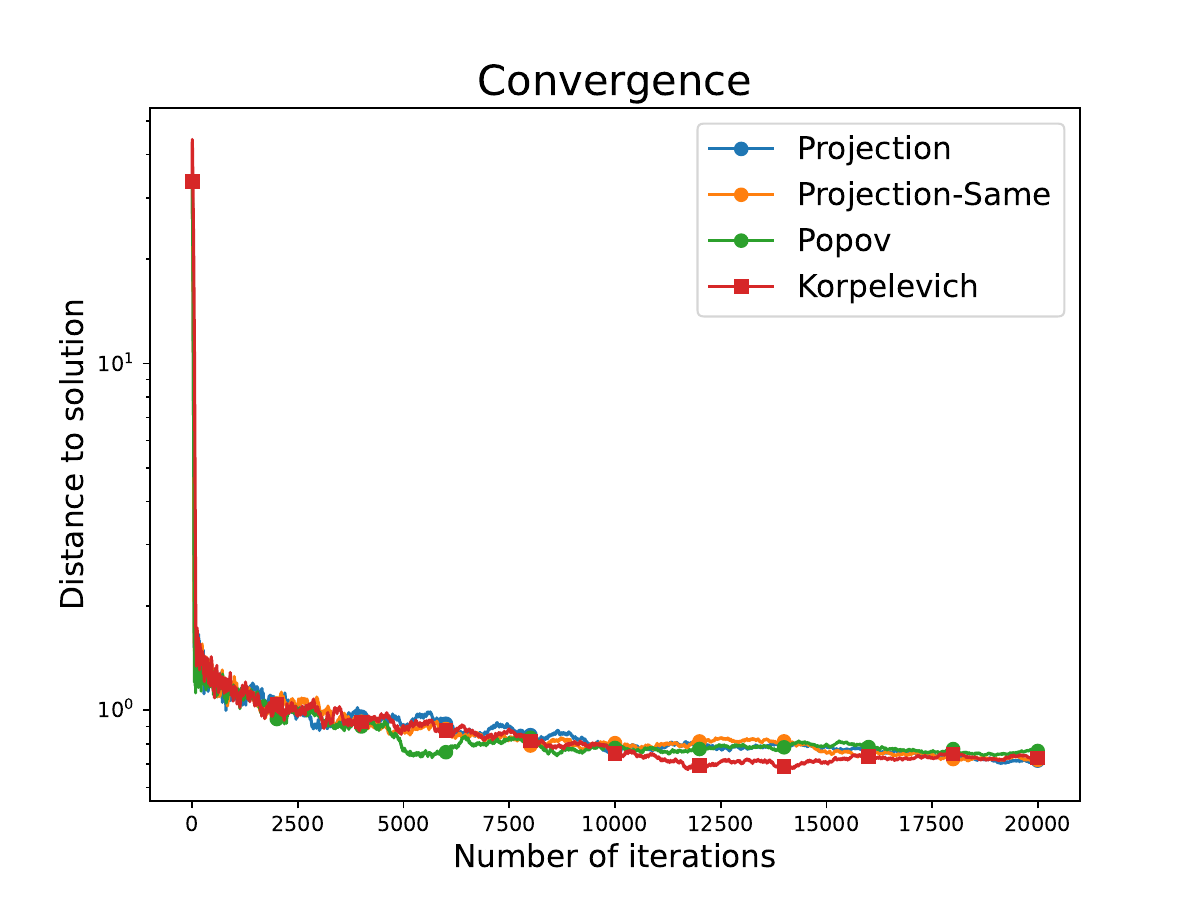}
}
\caption{Comparison of the clipped stochastic projection, same-sample projection, Korpelevich, and Popov methods with $\beta_k = 100/(100 + k^{1 - \epsilon})$. 
}
\label{fig:fair}
\end{figure*}

In Figure ~\ref{fig:theory}, we plot an average distance to solution from the current iterate over twenty runs to the solution set as a function of the number of iterations. In particular, the stepsizes for clipped stochastic projection and Korpelevich methods are chosen according to Theorems \ref{thm-proj-rates} and \ref{thm-korpelevich-rates}, respectively, with $\beta_k = \frac{100}{100 + k^{q}}$ for $q = 1/2 + \epsilon$ with $\epsilon> 0$. Note that, according to  Theorems~\ref{thm-proj-rates} and~\ref{thm-korpelevich-rates}, the parameter $q$ should be greater than $1/2$; meanwhile, the rates in these theorems are better for smaller choices of $q$. We also set $\beta_k =\frac{100}{100 + k^{q}}$ for  stochastic clipped  Popov method and the stochastic clipped  projection method using the same sample $\Phi(u_k,\xi_k)$ for clipping.  

Based on this experiment, we made three important observations. Firstly, the  stochastic clipped projection method and the same-sample stochastic clipped  projection method show similar results, despite the fact that the latter has a biased error.
We speculate that this is because the biased error is relatively small, and in practice dominated by the term $\gamma_k \dist^{p}(u_k, U^*)$.
Secondly, although in the stochastic Lipschitz SVI setting the Korpelevich method outperforms the projection method, we do not observe this advantage in the generalized smooth SVI setting. This aligns with our theoretical results, since both methods have the same complexity and require two oracle calls per iteration. Moreover, even in the standard Lipschitz continuous strongly monotone case, both methods achieve the same order $O(1/k)$ rate in the leading stochastic term, with the Korpelevich method enjoying a smaller non-leading term because it can use a larger stepsize of $\tfrac{1}{L}$~\cite{beznosikov2022smooth} compared to $\tfrac{\mu}{L^2}$ for the projection method~\cite{loizou2021stochastic}. However, in the generalized smooth case, where stepsizes are clipped, the Korpelevich method no longer benefits from larger stepsizes.

Next, we investigate the performance of the methods for larger values of $q$. In Figure~\ref{fig:fair}, we set $q=1 - \epsilon$, $\beta_k = \frac{100}{100 + k^{1 - \epsilon}}$, and run all four methods for the same problem parameter setting. We observe that for all considered $\a$, despite the theory, a larger choice of $q$ improved the performance of all methods in the $\sigma$-neighborhood.

For the same setting, in Figures~\ref{fig:theory-1} and~\ref{fig:fair-1}, now we plot the distance to the solution from  the average iterate $\bar{u_k} = (\sum_{i=0}^k \beta_i)^{-1} \sum_{i=0}^k \beta_i u_i$. In terms of average iterates, we observed that smaller values of $q$ are preferable. Interestingly, the clipped projection methods outperform the clipped Korpelevich method, even though both enjoy the same convergence guarantee of order $\mathcal{O}(1/k^{2(1-q)/p})$.
\begin{figure*}[hbt!]
\centering
\subfigure[$(\alpha \approx 0.33, p = 1.5)$]{
\includegraphics[width=.23\textwidth]{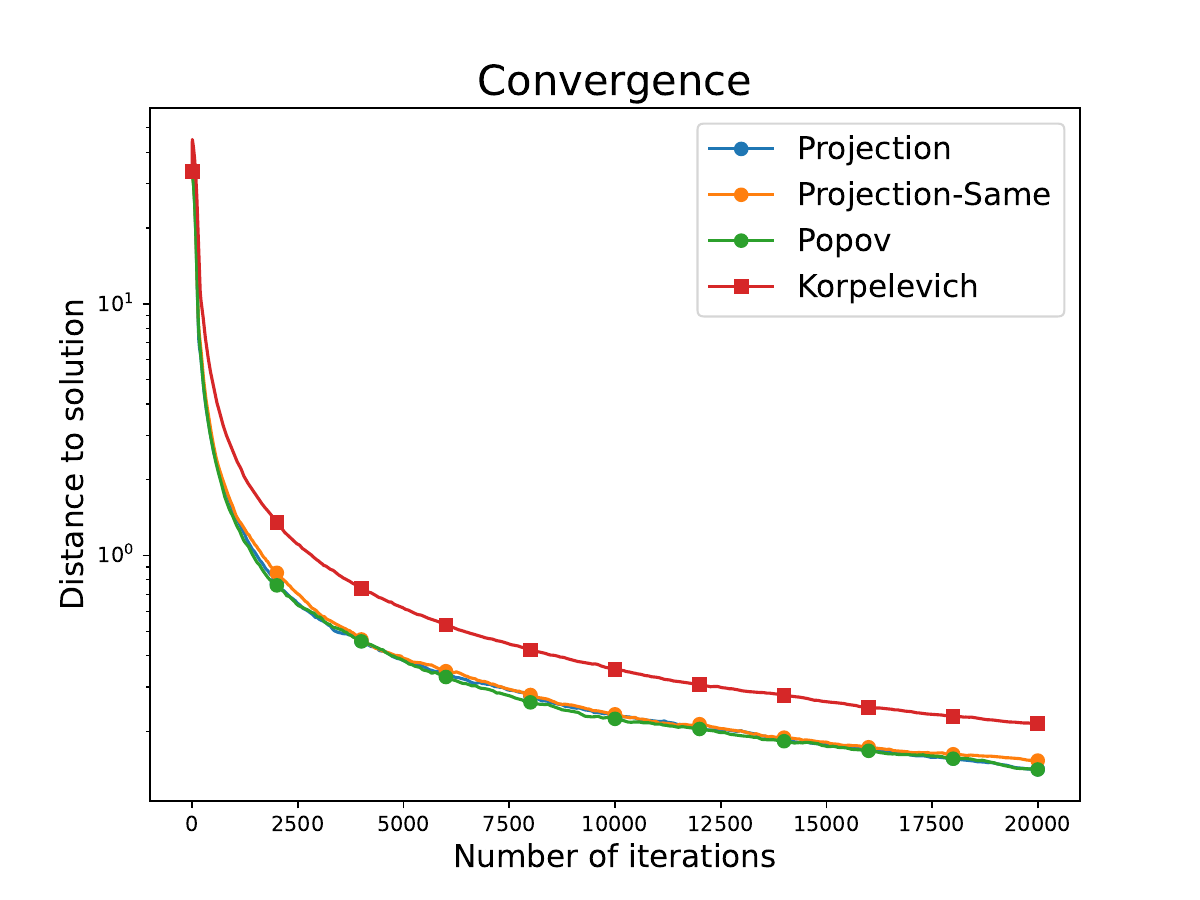}
}
\subfigure[$(\alpha \approx 0.33, p = 2.5)$]{
\includegraphics[width=.23\textwidth]{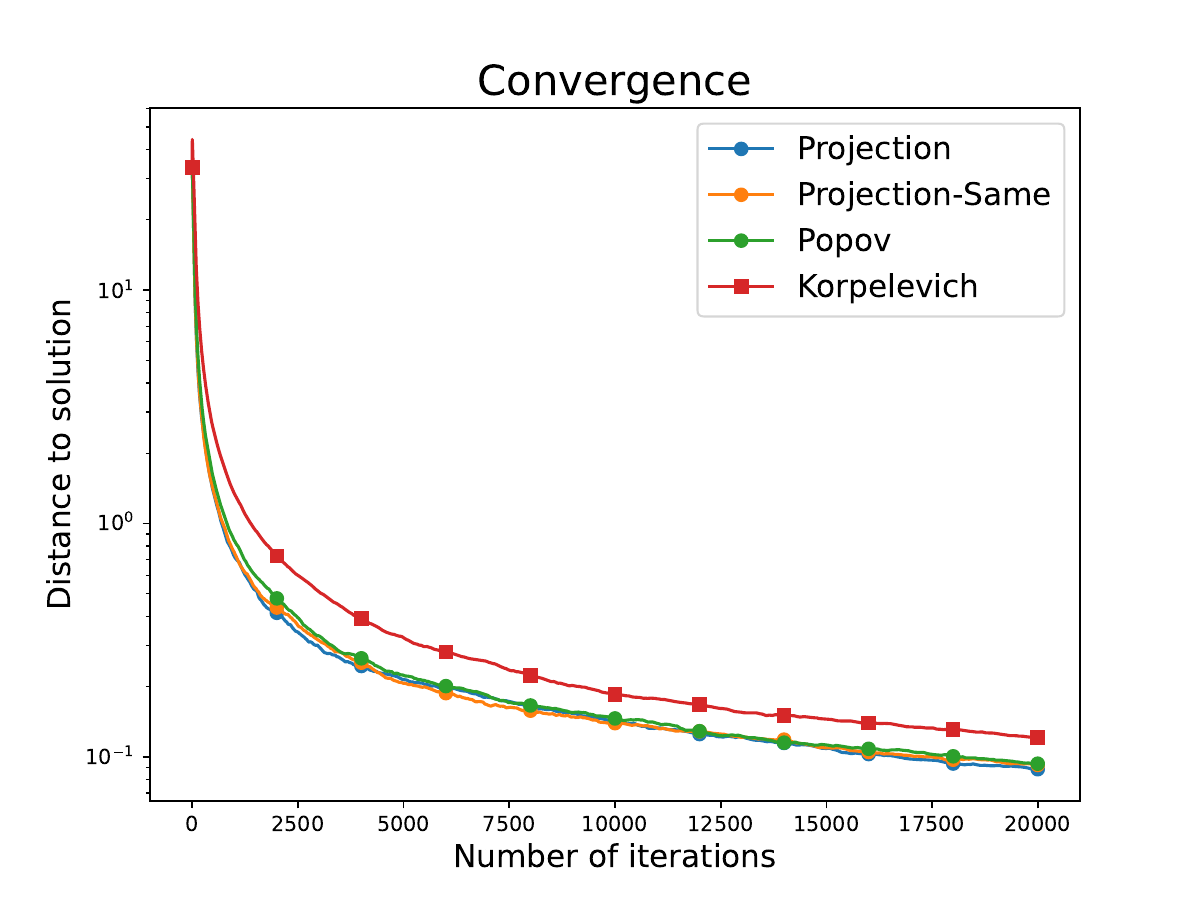}
}
\subfigure[$(\alpha \approx 0.8, p = 4.0)$]{
\includegraphics[width=.23\textwidth]{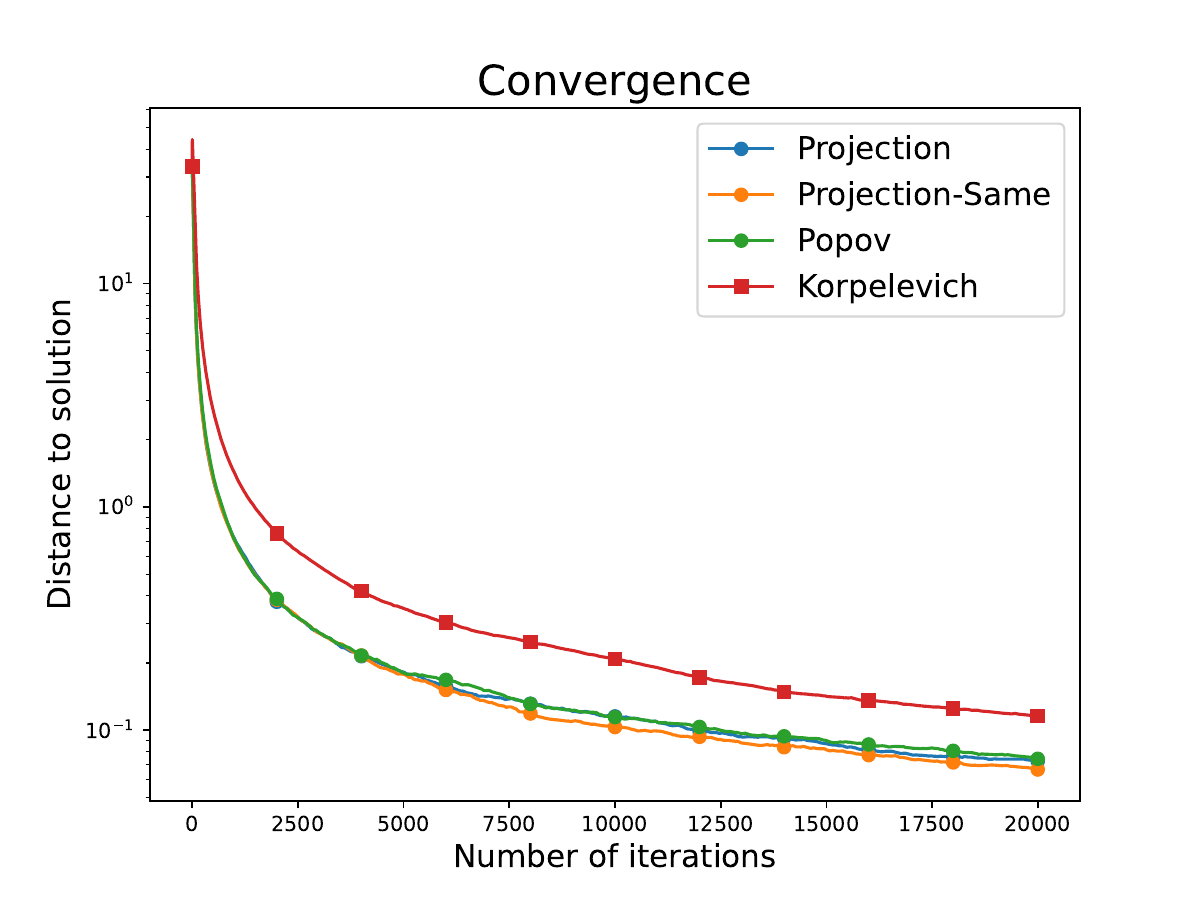}
}
\subfigure[$(\alpha \approx 0.8, p = 6.0)$]{
\includegraphics[width=.23\textwidth]{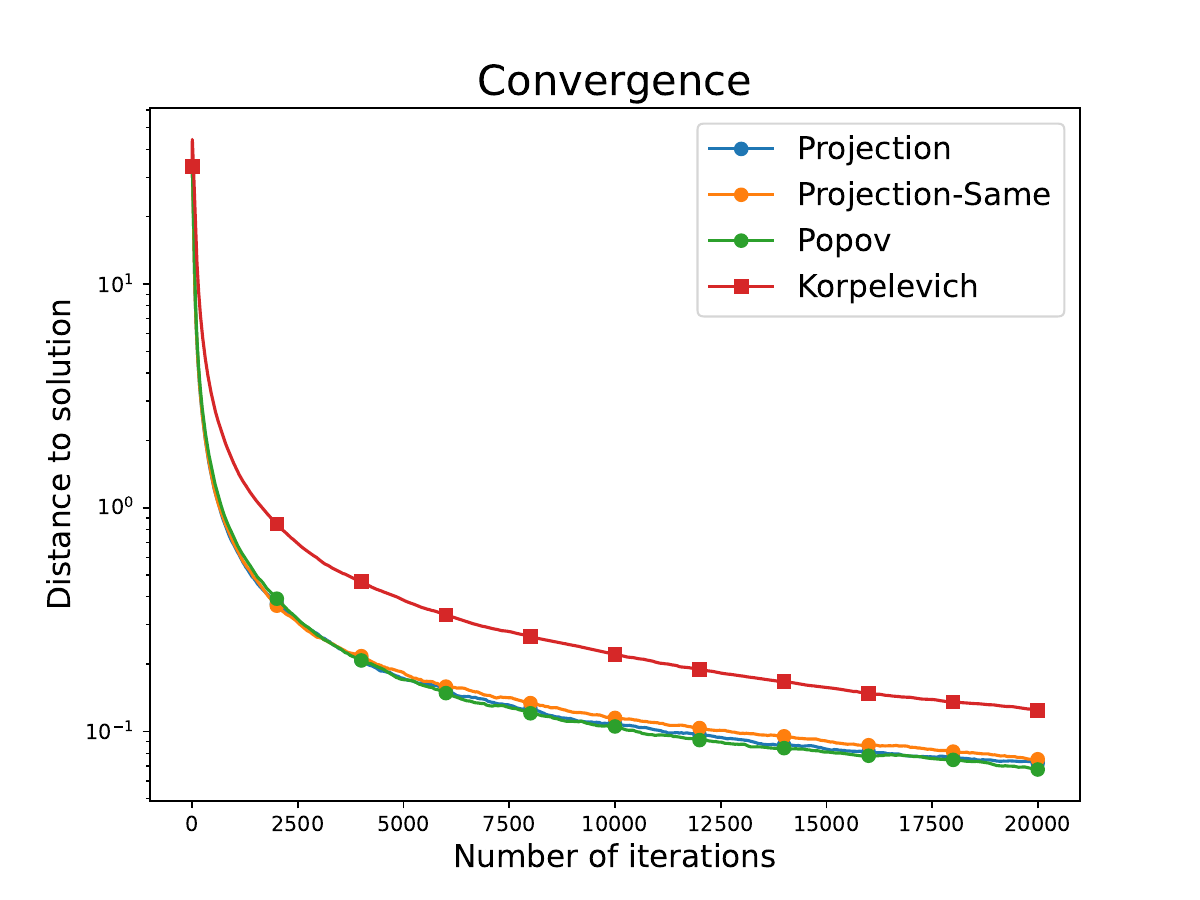}
}
\caption{Comparison of the clipped stochastic projection, same-sample projection, Korpelevich, and Popov methods with $\beta_k = 100/(100 + k^{1/2 + \epsilon})$ for averaged iterates. 
}
\label{fig:theory-1}
\end{figure*}

\begin{figure*}[hbt!]
\centering
\subfigure[$(\alpha \approx 0.33, p = 1.5)$]{
\includegraphics[width=.23\textwidth]{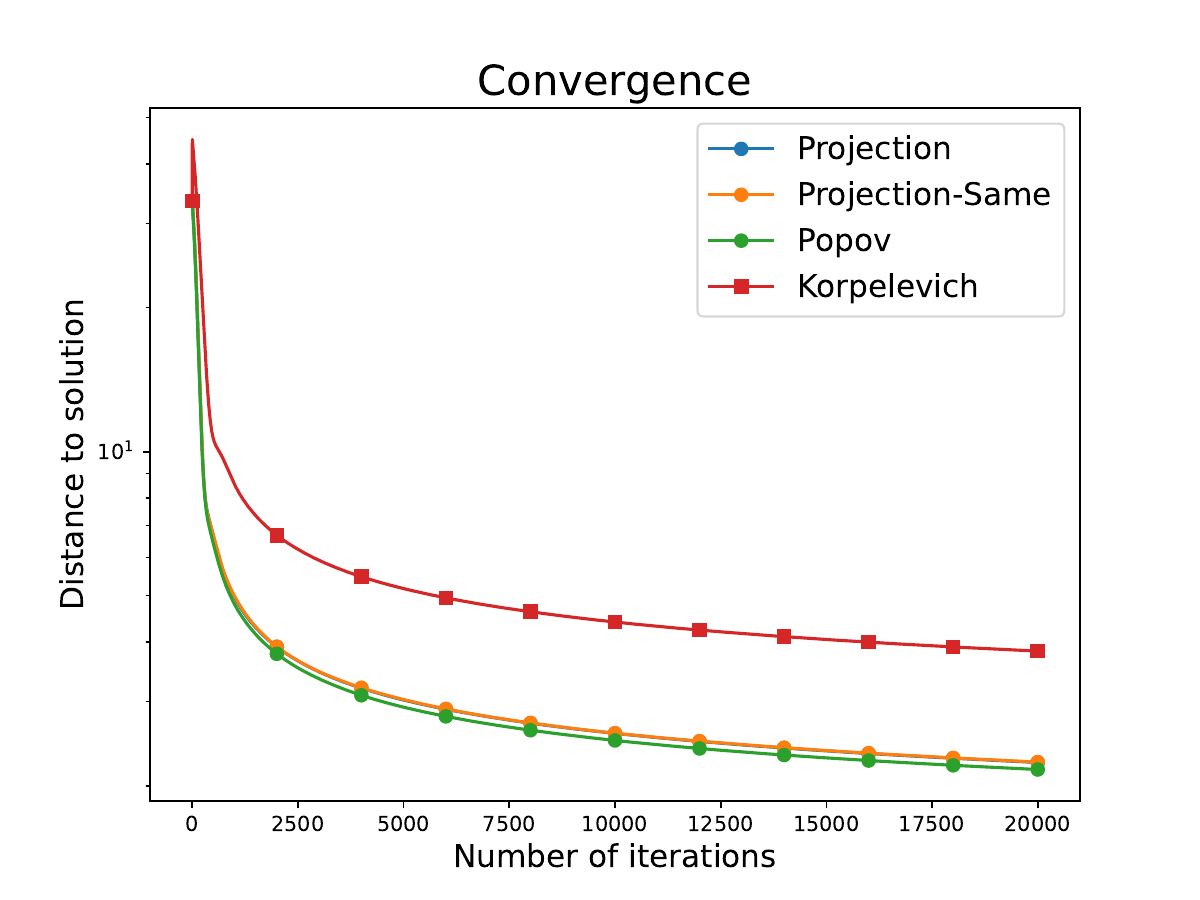}
}
\subfigure[$(\alpha \approx 0.33, p = 2.5)$]{
\includegraphics[width=.23\textwidth]{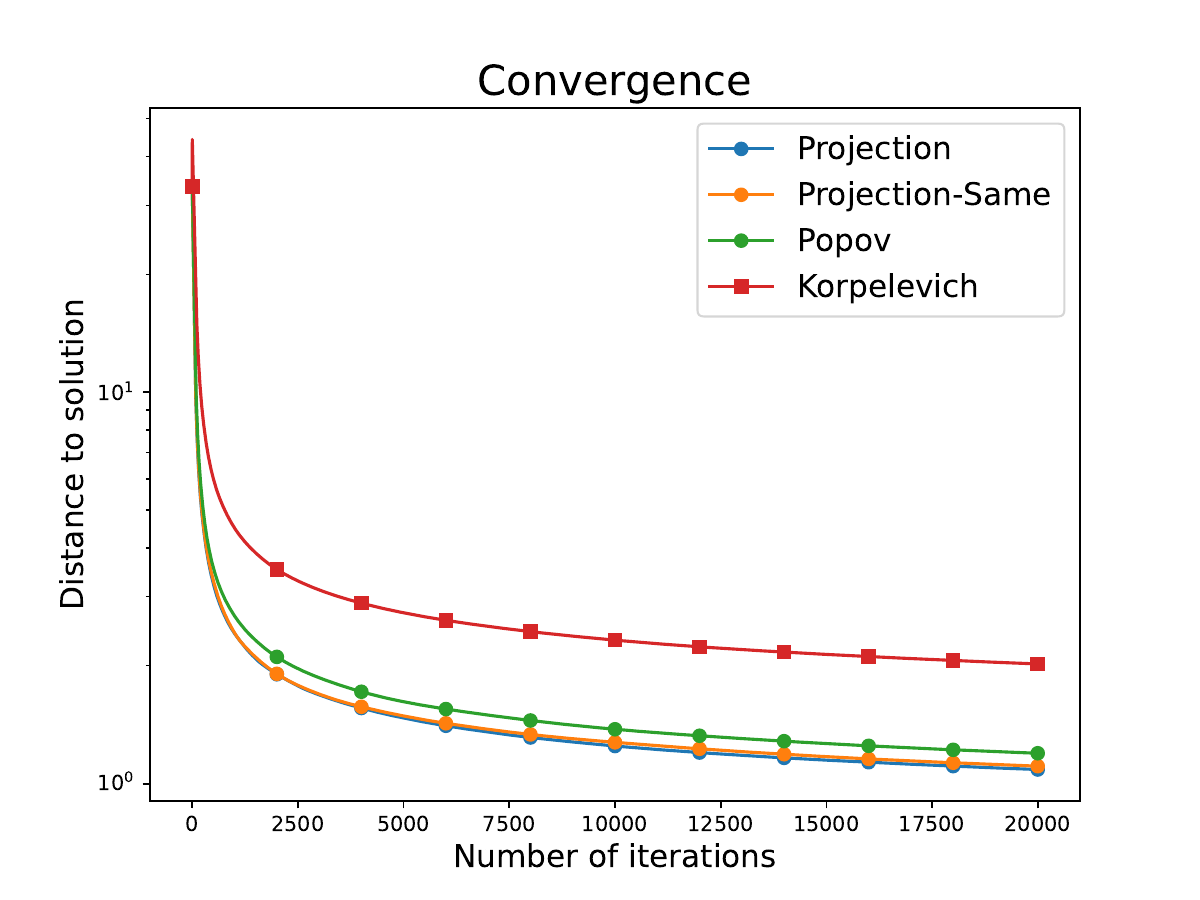}
}
\subfigure[$(\alpha \approx 0.8, p = 4.0)$]{
\includegraphics[width=.23\textwidth]{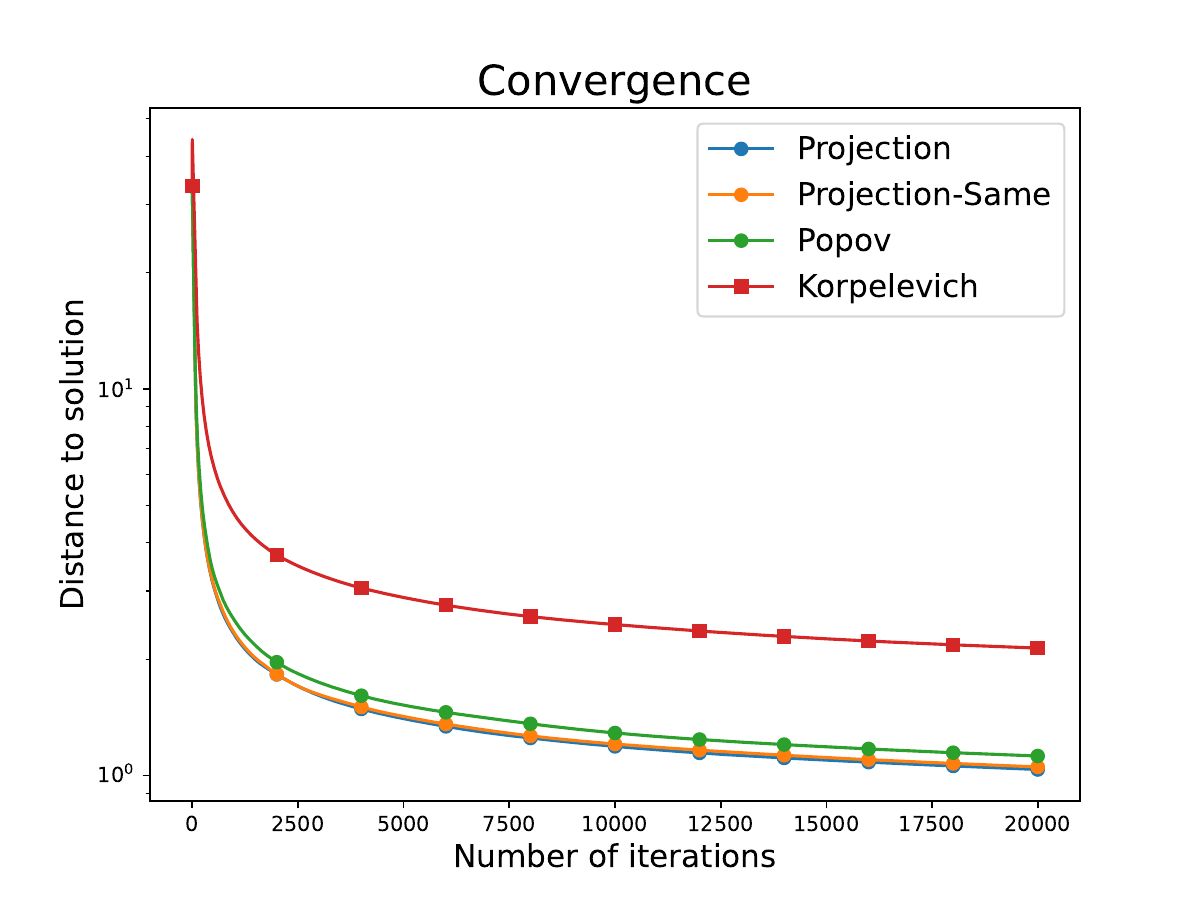}
}
\subfigure[$(\alpha \approx 0.8, p = 6.0)$]{
\includegraphics[width=.23\textwidth]{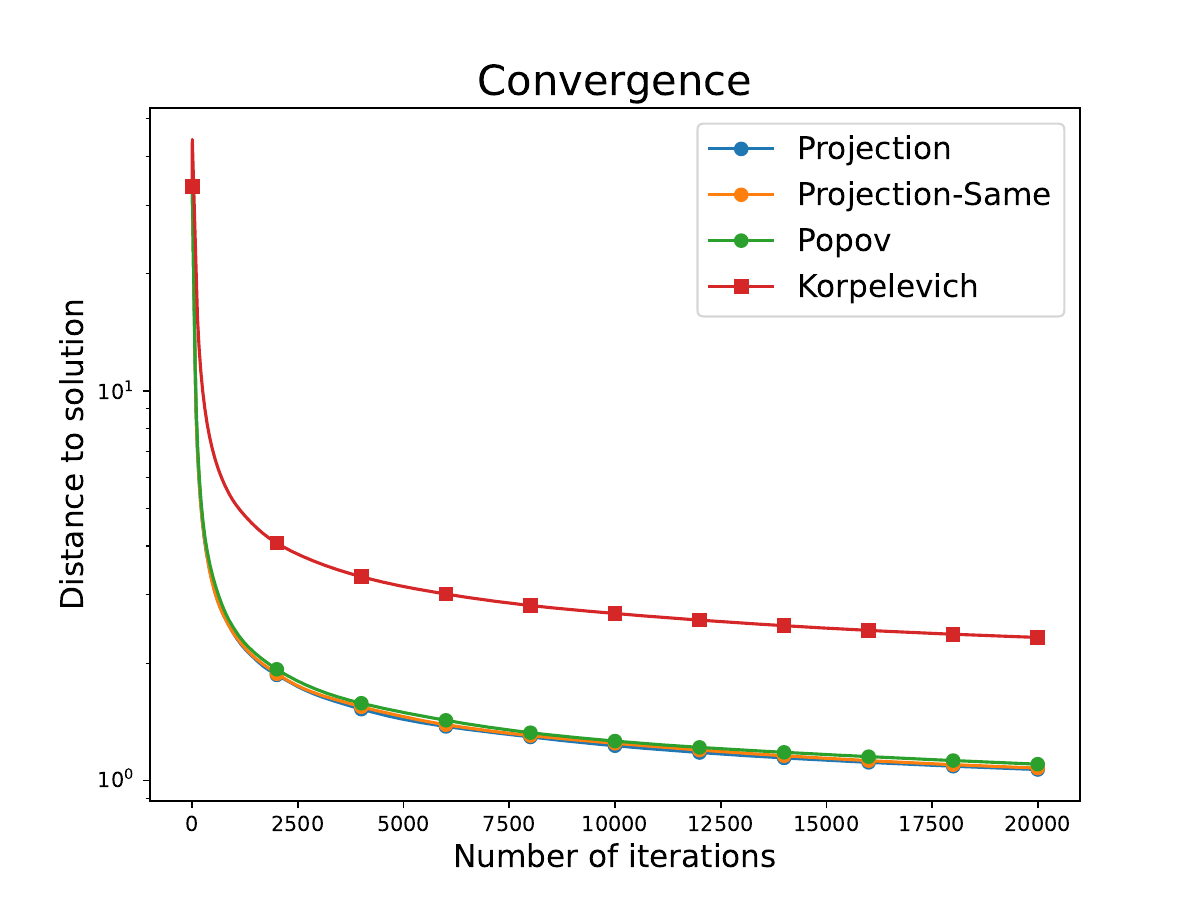}
}
\caption{Comparison of the clipped stochastic projection, same-sample projection, Korpelevich, and Popov methods with $\beta_k = 100/(100 + k^{1 - \epsilon})$ for averaged iterates. 
}
\label{fig:fair-1}
\end{figure*}

\section{Conclusion}
\label{Conclusion}
This paper studied the SVI problem under generalized smooth and structured non-monotone assumptions. Specifically, we consider $\a$-symmetric and $p$-quasi-sharp operators, a class of generalized smooth and structured non-monotone operators for SVIs. For this wide class of operators, we proved the first-known almost sure convergence of stochastic clipped projection and Korpelevich methods for all parameters $p$. We also provided $\mathcal{O}(1/k)$ convergence rate for both considered methods when the operator is $p$-quasi sharp with $p=2$. For $p >2$ we provided $\mathcal{O}(k^{-2(1 - q)/p})$ average (or best) iterate convergence rate for both methods, where $q$ is a stepsize parameter $1/2 < q < 1$. Despite the generality of our results, there are still open questions that remain. In particular, it would be interesting to know if it is possible to show in-expectation convergence rates for $\a$-smooth SVI for $\a > 1/2$. Another attractive direction of further research in generalized smooth SVIs is in the relaxation of $p$-quasi sharpness assumption to Minty ($\mu = 0$) or weak Minty conditions ($\mu < 0$). We also believe that our technique for proving almost sure convergence and in-expectation rates can be used for the analysis of other methods whose stepsizes are random variables, for example, stochastic clipped Popov method or first-order methods with adaptive stepsizes.

\section*{Acknowledgments}
We thank the anonymous reviewers for their valuable comments. This work is supported in part by NSF
grants CIF-2134256, SCH-2205080 and CIF-2007688.

\bibliography{generalized-svi}
\bibliographystyle{tmlr}

\appendix

\section{Technical Lemmas}
In our analysis, we use the properties of the projection operator $P_U(\cdot)$ given in the following lemma.
\begin{lemma}\label{lem-proj} (Theorem 1.5.5 and Lemma 12.1.13 in~\cite{facchinei2003finite}) Given a nonempty convex closed set $U\subset\mathbb{R}^\bd,$ the projection operator $P_U(\cdot)$ has the following properties:
\begin{equation}
    \label{eq-proj1}
    \langle v - P_{U}(v), u - P_{U}(v) \rangle \leq 0  \quad \hbox{for all } u \in U, v \in \mathbb{R}^\bd,
\end{equation}
\begin{equation}
    \label{eq-proj3} 
    \|u - P_{U}(v)\|^2 \leq \|u - v\|^2 - \|v - P_{U}(v)\|^2 \quad \hbox{for all } u \in U, v \in \mathbb{R}^\bd,
\end{equation}
\begin{equation}
    \label{eq-proj5}
    \|P_U(u)-P_U(v)\|\le \|u-v\| \quad \hbox{for all } u , v \in \mathbb{R}^\bd.
  \end{equation}  
\end{lemma}

In the forthcoming analysis, we use Lemma~11 \cite{polyak1987introduction}, which is stated below.
\begin{lemma} \label{lemma-polyak11} [Lemma~11 \cite{polyak1987introduction}]
Let $\{v_k\}, \{z_k\}, \{a_k\},$ and  $\{b_k\}$ be nonnegative random scalar sequences such that almost surely for all $k\ge0$,
\begin{equation}
\begin{aligned}
\label{eq-polyak-0}
\mathbb{E}[v_{k+1}\mid {\cal F}_k] \leq &(1 + a_k)v_k -z_k + b_k,
\end{aligned}
\end{equation}
where 
${\cal F}_k = \{v_0, \ldots, v_k, z_0, \ldots, z_k, a_0, \ldots, a_k,b_0, \ldots,b_k\}$, and
\emph{a.s.} $\sum_{k=0}^{\infty} a_k < \infty$, $\sum_{k=0}^{\infty} b_k < \infty$. Then, almost surely, $\lim_{k\to\infty} v_k =v $  for some nonnegative random variable $v$ and $\sum_{k=0}^{\infty} z_k < \infty$.
\end{lemma}

As a direct consequence of Lemma~\ref{lemma-polyak11}, when
the sequences $\{v_k\}, \{z_k\}, \{a_k\}, \{b_k\}$ are deterministic, we obtain the following result.
\begin{lemma} \label{lemma-polyak11-det} 
Let $\{\bar v_k\}, \{\bar z_k\}, \{\bar a_k\}, \{\bar b_k\}$ be nonnegative scalar sequences such that for all $k\ge0$,
\begin{equation}
\begin{aligned}
\label{eq-polyak-1}
\bar v_{k+1}\leq &(1 + \bar a_k)\bar v_k -\bar z_k + \bar b_k,
\end{aligned}
\end{equation}
where $\sum_{k=0}^{\infty} \bar a_k < \infty$ and $\sum_{k=0}^{\infty} \bar b_k < \infty$. Then, 
$\lim_{k\to\infty} \bar v_k = \bar v$ 
for some scalar $\bar v\ge0$ and $\sum_{k=0}^{\infty} \bar z_k < \infty$.
\end{lemma}

\begin{lemma} \label{lemma-chebyshev} 
Let $X$ be a non-negative random variable such that $\mathbb{E}[X^{\rho}]$ is defined for some $\rho \geq 1$, and $\mathbb{E}[X^{\rho}] \not = 0$, then for every $a > 0$ it holds
\begin{equation}
\begin{aligned}
\label{eq-chebyshev}
\mathbb{P}(X > a (\mathbb{E}[X^{\rho}] )^{1/\rho} )\leq \frac{1}{a^{\rho}}.
\end{aligned}
\end{equation}
\end{lemma}
\begin{proof}
    Let $Y = X^{\rho}$. By the conditions of the lemma, the expectation $\mathbb{E}[Y] = \mathbb{E}[X^{\rho}]$ is well defined. Then, by Markov's inequality:
    \begin{align*}
        \mathbb{P}(X > a (\mathbb{E}[X^{\rho}] )^{1/\rho}) &= \mathbb{P}(Y > a^{\rho}\mathbb{E}[X^{\rho}] )\cr
        &\leq \frac{\mathbb{E}[X^{\rho}] }{a^{\rho} \mathbb{E}[X^{\rho}] }.
    \end{align*}
\end{proof}

\begin{lemma}
\label{inequality-hoed-2}
Let $a_1, a_2$ be nonnegative scalar and $p>0$. Then the following inequality holds:
\[\left( a_1 + a_2\right)^p \leq 2^{p-1} (a_1^p + a_2^p).  \]
\end{lemma}
\begin{proof}
Let $a = (a_1, a_2), b = (1,1)$, then by H\"older inequality:
\begin{align*}
a_1 + a_2 &= \|a b\| \cr
&\leq \|a\|_{p} \|b\|_{p/(p-1)} \cr
&\leq (a_1^p + a_2^p)^{1/p} (1 + 1)^{(p-1)/p}.
\end{align*}
Raising the inequality in the power $p$ we get the desired relation.
\end{proof}

\subsection{Auxiliary Results}
In our analysis we make use of Lemma~3 and Lemma~7 from~\cite{stich2019unified}, as well as the sequences provided in the proofs in~\cite{stich2019unified}. 

\begin{lemma}
    \label{Lemma7-stich}
    Let $\{r_k\}$ and $\{s_k\}$ be nonnegative scalar sequences that satisfy the following relation
    \[r_{k+1}\le (1-a \alpha_k) r_k -b \alpha_k s_k + c \g_k^2\qquad\hbox{for all } k\ge0,\]
    where $a>0$, $b>0$, $c\ge0$, and 
    \[\g_k = \frac{2}{a\left( \frac{2d}{a} +k\right)} \qquad\hbox{for all }k\ge 0,\]
    where $d\ge a$.
    Then, for any given $K\ge0$, the following relation holds:
\[\frac{b}{W_K} \sum^{K}_{k=0} w_k s_k + a r_{K+1} \leq  \frac{8d^2}{a K^2}\,r_0 + \frac{2c}{aK}, \]
where $w_k=2d/a +k$,  $0\le k\le K$, and $W_K=\sum_{k=0}^K w_k$.
\end{lemma}

\begin{lemma} \label{lemma-kq-series-rates}
For $1 > q \geq 1/2$ and $K \geq 1$, we have
\begin{align}
\sum^{K}_{t = 0} \frac{1}{(t+ 1)^q} &\geq \frac{1}{1 - q} ((K+1)^{1 - q} - 2^{1 - q}). 
\end{align}
For $q =1/2$ and $K \geq 1$,
\begin{align}
\sum^{K}_{t = 0} \frac{1}{(t+ 1)^{2q}} &\leq \log (K+1).
\end{align}
For $q > 1/2$  and $K\ge1$,
\begin{align}
\sum^{K}_{t = 0} \frac{1}{(t+ 1)^{2q}} &\leq  \frac{1}{2q -1}.
\end{align}
\end{lemma}
\begin{proof}
 Let $1>q \geq 1/2$ and $K \geq 1$. Then, it holds
\begin{align}
    \sum^{K}_{t = 0} \frac{1}{(t + 1)^q} \geq \int_{s=1}^{K} \frac{d s}{(s + 1)^q} = \frac{1}{1-q} ((K+1)^{1 - q} - 2^{1-q}).
\end{align}
When $q = 1/2$ and $K\ge 1$, then
\begin{align}
    \sum^{K}_{t = 0} \frac{1}{t + 1}  \leq \int_{s=0}^{K} \frac{d s}{s + 1}=  \log (K+1).
\end{align}
When $q > 1/2$ and $K\ge 1$, we have that 
\begin{align}
    \sum^{K}_{t = 0} \frac{1}{(t + 1)^{2q}}  \leq \int_{s=0}^{K} \frac{d s}{(s + 1)^{2q}}=\frac{1}{2q-1}-\frac{1}{(2q-1)(K+1)^{2q-1}}  <\frac{1}{2q -1}.
\end{align}
\end{proof}

\section{Projection Method Analysis}
\label{Appendix-Projection}
\subsection{Proof of Lemma~\ref{lemma-proj-as-bound}}~\label{sec-lem31}
\begin{proof}
Let $k\ge 0$ be arbitrary but fixed. From the definition of $u_{k+1}$ in~\plaineqref{eq-projection-stoch}, we have $\|u_{k+1} - y\|^2 = \|P_{U}(u_k - \gamma_k \Phi(u_k, \xi_k)) - y\|^2$ for all $y \in U$. Using the non-expansiveness property of projection operator~\plaineqref{eq-proj5} we obtain for all $y \in U$ and $k\ge0$,
\begin{align}\label{eq-grad-1}
\|u_{k+1}-y\|^2 
&\le \|u_k - \g_k \Phi(u_k, \xi_k) - y\|^2\cr
&= \|u_k -y\|^2 - 2\g_k \langle \Phi(u_k, \xi_k), u_{k}-y\rangle + \g_k^2 \|\Phi(u_k, \xi_k)\|^2 \cr
&= \|u_k -y\|^2 + \g_k^2 \|\Phi(u_k, \xi_k)\|^2 \cr
&- 2\g_k \langle F(u_k), u_{k}-y\rangle  + 2\g_k \langle e_k, u_{k}-y\rangle,
\end{align}
where $e_k = F(u_k) - \Phi(u_k, \xi_k)$. By the definition of the stepsizes~\plaineqref{eq-stepsizes-projection}, $\gamma_k = \beta_k \min\{1, \frac{1}{\|\Phi(u_k, \xi_k^2)\|}\}$, then the term $\gamma_k^2 \|\Phi(u_k, \xi_k)\|^2$ can be upper bounded as follows
\begin{align}\label{eq-grad-2-0}
\gamma_k^2 \|\Phi(u_k, \xi_k)\|^2  &= \gamma_k^2 \|\Phi(u_k, \xi_k) - F(u_k) + F(u_k) - \Phi(u_k, \xi_k^2) + \Phi(u_k, \xi_k^2)\|^2 \cr
&\leq  \beta_k^2 \min\left\{1, \frac{1}{\|\Phi(u_k, \xi_k^2)\|^2}\right\} \,3(\|e_k\|^2 +  \|e_k^2\|^2 + \|\Phi(u_k, \xi_k^2)\|^2) \cr
&\leq 3 \beta_k^2 \|e_k\|^2 + 3 \beta_k^2 \|e_k^2\|^2 + 3 \beta_k^2,
\end{align}
where $e_k = \Phi(u_k, \xi_k) - F(u_k), e_k^2 = \Phi(u_k, \xi_k^2) - F(u_k)$. Thus,
\begin{align}\label{eq-grad-2}
\|u_{k+1}-y\|^2  &\le \|u_k -y\|^2  - 2\g_k \langle F(u_k), u_{k}-y\rangle \cr
&+ 2\g_k \langle e_k, u_{k}-y\rangle + 3 \beta_k^2 (\|e_k\|^2 + \|e_k^2\|^2 + 1). 
\end{align}

Plugging in $y = u^* \in U^*$, where $u^*$ is an arbitrary solution, and using $p$-quasi sharpness we get:
\begin{align}\label{eq-grad-3}
\|u_{k+1}-u^*\|^2 &\leq \|u_k -u^*\|^2 - 2 \mu \g_k \dist^p(u_k, U^*)  \cr
&+ 2\g_k \langle e_k, u_{k}-u^*\rangle  + 3 \beta_k^2 (\|e_k\|^2 + \|e_k^2\|^2 + 1).
\end{align}

Using stochastic properties of $\xi_k$ and $\xi_k^2$ imposed by Assumption~\ref{asum-samples}, and the conditional independence of $\xi_k$ and $\xi_k^2$, we have:
\[\mathbb{E}[\gamma_k \langle e_k, u_k - u^* \rangle| \mathcal{F}_{k-1}] = \mathbb{E}[\gamma_k \mid \mathcal{F}_{k-1}]  \langle \mathbb{E}[e_k \mid \mathcal{F}_{k-1}], u_k - u^* \rangle  =  0.\]
\[\mathbb{E} [\|e_k\|^2 \mid \mathcal{F}_{k-1}] \leq \sigma^2 , \quad  \mathbb{E}[\|e_k^2\|^2 \mid \mathcal{F}_{k-1}]\leq \sigma^2.\]
Thus, by taking the conditional expectation on  $\mathcal{F}_{k-1} \cup \xi_{k}^2 = \{\xi_0, \xi_0^2, \ldots, \xi_{k-1}, \xi_{k-1}^2, \xi_k^2\}$ in relation~\plaineqref{eq-grad-3} we obtain for all $u^* \in U^*$ and for all $k \geq 0$:

\begin{align}\label{eq-grad-3-1}
\mathbb{E}[\|u_{k+1}-u^*\|^2 \mid \mathcal{F}_{k-1} \cup \xi_{k}^2 ] &\leq \|u_k -u^*\|^2 + 3\beta_k^2(2 \sigma^2 + 1) - 2 \mu \beta_k \min \left\{1, \frac{1}{\|\Phi(u_k, \xi_k^2)\|} \right\} \dist^p(u_k, U^*) .
\end{align}

The equation~\plaineqref{eq-grad-3-1} satisfies the condition of Lemma~\ref{lemma-polyak11} with 
\begin{align} v_{k} = \|u_{k}  - u^*\|^2, \quad  a_k = 0, 
\quad z_k =  2  \mu  \gamma_k \mid \, \dist^p(u_{k}, U^*), \quad  b_k = 3\beta_k^2(2 \sigma^2 + 1) .
\end{align}

By Lemma~\ref{lemma-polyak11}, it follows that the sequence $\{v_k\}$ converges {\it a.s.}\ to a non-negative scalar for any $u^*\in U^*$, and  almost surely we have
\begin{equation}
\label{eq-projection-stoch-RS-lemma-result}
\sum_{k=0}^{\infty} \gamma_k  \, \dist^p(u_k, U^*) < \infty.
\end{equation}
Since the sequence $\{\|u_k - u^*\|^2\}$ 
converges {\it a.s.} for all $u^*\in U^*$, it follows that the sequence 
$\{\|u_k - u^*\|\}$ is bounded {\it a.s.} for all $u^* \in U^*$.
\end{proof}

\subsection{Proof of Theorem~\ref{thm-proj-as-converge}}\label{sec-thm32}

\begin{lemma}
\label{lemma-stepsizes-nonsummable-projection}
    Let $\gamma_k $ are given by \eqref{eq-stepsizes-projection} then the series of $\{ \gamma_k \}$ is non-summable almost surely,
    \begin{equation}
        \sum_{k=0}^{\infty} \gamma_k  = \infty \quad \rm{a.s.}
    \end{equation}
\end{lemma}

\begin{proof}
We will show that $\sum_{k=0}^{\infty} \beta_k  \min \left\{1, \frac{1}{\|\Phi(u_k, \xi_k^2)\|} \right\} = \infty$ almost surely by the sequences of lower bound on this series. Consider the following event:
\[A_k = \{ \|e_k^2\| \leq 2 \sigma \},\]
where $e_k^2 = \Phi(u_k, \xi_k^2) - F(u_k)$ is a stochastic error from the sample for the clipping stepsize $\gamma_k$. Define $x_k = \min \left\{1, \frac{1}{\|\Phi(u_k, \xi_k^2)\|} \right\}$, then, 
\begin{align}\label{eq-lemma-steps-projection-1}
x_k  = x_k  \mathbb{I}(A_k) + x_k \mathbb{I}(\overline{A}_k) \geq x_k  \mathbb{I}(A_k),
\end{align}
where the random variable $\mathbb{I}(A_k) = $ is the indicator function of the event $A_k$ taking value 1 when the event occurs, and taking value 0 otherwise. 

By the definition of $x_k$, the triangle inequality and definition of $\mathbb{I}(A_k)$, we have
\begin{align}
\label{eq-lemma-steps-projection-2}
x_k \mathbb{I}(A_k)  &=  \min \left\{1, \frac{1}{\|\Phi(u_k, \xi_k^2)\|} \right\} \mathbb{I}(A_k)  \cr
&\geq  \min \left\{1, \frac{1}{\|F(u_k)\|+ \|e_k^2\|}\right\}  \mathbb{I}(A_k) \cr
&\geq  \min \left\{1, \frac{1}{\|F(u_k)\| + 2 \sigma}\right\} \mathbb{I}(A_k) .\cr
\end{align}
By combining the resulting relation with \eqref{eq-lemma-steps-projection-2}, using the definition of $\g_k$, and adding and subtracting $\mathbb{E}[\mathbb{I}(A_k) \mid \mathcal{F}_{k-1}]$,  we have the following lower bound
\begin{align}
\label{eq-lemma-steps-projection-3}
\sum_{k=0}^{\infty} \beta_k x_k  
&\geq \sum_{k=0}^{\infty} \beta_k  \min \left\{1, \frac{1}{\|F(u_k)\| + 2 \sigma}\right\} (\mathbb{I}(A_k) - \mathbb{E}[\mathbb{I}(A_k) \mid \mathcal{F}_{k-1}]) \cr
&+ \sum_{k=0}^{\infty} \beta_k  \min \left\{1, \frac{1}{\|F(u_k)\| + 2 \sigma}\right\} \mathbb{E}[\mathbb{I}(A_k) \mid \mathcal{F}_{k-1}]  .
\end{align}

To bound $p_k := \mathbb{E}[\mathbb{I}(A_k) | \mathcal{F}_{k-1}] = \mathbb{P}(A_k \mid \mathcal{F}_{k-1})$ we  provide an upperbound on $\mathbb{P}(\overline{A}_k \mid \mathcal{F}_{k-1})$ using Markov's inequality and Assumption~\ref{asum-samples}:
\begin{align}\label{eq-lemma-steps-as-3-0}
\mathbb{P}(\overline{A}_k \mid \mathcal{F}_{k-1}) = \mathbb{P}( \|e_k^1\| > 2 \mathbb{E}[\|e_k^1\| \mid \mathcal{F}_{k-1}] \}) \leq \frac{\mathbb{E}[\|e_k^1\| \mid \mathcal{F}_{k-1}]}{2 \mathbb{E}[\|e_k^1\| \mid \mathcal{F}_{k-1}])} = \frac{1}{2}.
\end{align}
This implies $\mathbb{E}[\mathbb{I}(A_k) | \mathcal{F}_{k-1}] \geq \frac{1}{2}$. Define $S_n = \sum_{k=0}^{n} \beta_k (\mathbb{I}(A_k) - \mathbb{E}[\mathbb{I}(A_k) | \mathcal{F}_{k-1}] )$, by construction, $S_n$ is a martingale:
\[ \mathbb{E}[S_{n+1} \mid S_0, \ldots, S_n] = S_n + \mathbb{E}[\beta_{n+1} (\mathbb{I}(A_{n+1}) - \mathbb{E}[\mathbb{I}(A_{n+1}) | \mathcal{F}_{n}] ) \mid S_0, \ldots, S_n] = S_n. \]
We want to show that $\lim_{n \rightarrow \infty} S_n \rightarrow S < \infty$ almost surely. We provide an upper bound for $\mathbb{E}[S_n^2]$:
\begin{align}
\mathbb{E}[S_n^2] = \sum_{k=0}^{n} \beta_k^2 \mathbb{E}[(\mathbb{I}(A_k) - p_k )^2] + 2\sum_{0 \leq k < i \leq n} \beta_k^2 \mathbb{E}[ (\mathbb{I}(A_k) - p_k ) (\mathbb{I}(A_i) - p_i )]
\end{align}
By the law of total expectation, and noting that $ \mathbb{E}[\mathbb{I}(A_k) - p_k \mid \mathcal{F}_{k-1}] = 0 $ for any $k$, we find that
for all $0 \leq k < i \leq n$,
\begin{align}
\mathbb{E}[ (\mathbb{I}(A_k) - p_k ) (\mathbb{I}(A_i) - p_i )] = \mathbb{E}[(\mathbb{I}(A_k) - p_k )  \mathbb{E}[(\mathbb{I}(A_i) - p_i )\mid \mathcal{F}_{i-1}]]  = 0 .
\end{align}
implying that, for all $n \geq 0$
\begin{align}
\mathbb{E}[S_n^2] = \sum_{k=0}^{n} \beta_k^2 \mathbb{E}[(\mathbb{I}(A_k) - p_k )^2] 
\end{align}

Since
$ \mathbb{E}[(\mathbb{I}(A_k) - p_k )^2 \mid \mathcal{F}_{k-1}]  = Var(I(A_k)  \mid \mathcal{F}_{k-1})$ and the random variable $\mathbb{I}(A_k) $ is a Bernoulli given $ \mathcal{F}_{k-1}$ with mean $p_k$, then 
\[ 
\mathbb{E}[(\mathbb{I}(A_k) - p_k)^2 \mid \mathcal{F}_{k-1}] = \rm{Var}(\mathbb{I}(A_k)  \mid \mathcal{F}_{k-1}) \leq \frac{1}{4}
\]
By taking the total expectation we get $\mathbb{E}[(\mathbb{I}(A_k) - p_k)^2 ] \leq \frac{1}{4}$, and combining the previous two relations, we obtain

\[\mathbb{E}[S_n^2] \leq \frac{1}{4} \sum_{k=0}^{n} \beta_k^2 \leq \infty  \quad \text{a.s.}
\]
From Theorem 4.4.6. in \cite{durrett2019probability} it follows that $S_n$ converges to $S < \infty$ almost surely (and in $L_2$).

To further lower bound $x_k \mathbb{I}(A_k)$ we show \emph{a.s.} boundedness of $\|F(u_k)\|$  for all $k \geq 0$, using property of $\a$-symmetric operators. To estimate $\|F(u_k)\|$, we add and subtract $F(v^*)$, where $v^* \in U^*$ is an arbitrary but fixed solution, and get 
\[\|F(u_k)\| = \|F(u_k) - F(v^*) + F(v^*)\| \leq \|F(u_k) - F(v^*)\| + \|F(v^*)\|.\]
Define the following event:
\[A = \{\omega \in \Omega: \; \exists \;  C(\omega) \in \mathbb{R} \text{ s.t.} \|u_k(\omega) - v^*\| < C(\omega) \; \forall \; k \geq 0 \}.\]
Based on Lemma~\ref{lemma-proj-as-bound}, the sequence $\{\|u_k - v^*\|\}$ is bounded {\it a.s.}, and thus $\mathbb{P}(A) = 1$. Let $\omega \in A$, now we can estimate $\|F(u_k(\omega))\|$ using the $\alpha$-symmetric assumption on the operator. \\
\textbf{Case $\alpha \in (0,1)$}.\\
\begin{equation}
\begin{aligned}
\label{eq-lemma-steps-projection-5}
\|F(u_k(\omega)) - F(v^*)\| \leq \|u_k(\omega) - v^*\| (K_0 + K_1 \|F(v^*)\|^{\alpha} + K_2 \|u_k(\omega) - v^* \|^{\a / (1 - \a)}).
\end{aligned}
\end{equation}
Since $\omega \in A$, it follows that 
$\|u_k(\omega) - v^*\|   \leq C(\omega)$ for all $k \geq 0$.
Using this fact and~\plaineqref{eq-lemma-steps-projection-5} we obtain that for all $k \ge 0$,
\begin{align}\label{eq-lemma-steps-projection-6}
\|F(u_k(\omega)) \| \leq C(\omega)(K_0 + K_1 \|F(v^*)\|^{\alpha} + K_2 C(\omega)^{\a / (1 - \a)}) + \|F(v^*)\|.
\end{align}
Therefore, the sequence \{$\|F(u_k(\omega))\|$\} is upper bounded by $C_1(\omega) = C(\omega) (K_0 + K_1 \|F(v^*)\|^{\alpha} + K_2 C(\omega)^{\a / (1 - \a)}) + \|F(v^*)\|$. \\
\textbf{Case $\alpha = 1$}. \\
For $\a=1$ by Proposition~\ref{prop-a} we have
\begin{align}\label{eq-lemma-steps-projection-7}
\|F(u_k(\omega)) -  F(v^*)\| &\leq \|u_k(\omega) - v^*\| (L_0+ L_1 \|F(v^*)\|) \exp (L_1 \|u_k(\omega) - v^*\|) .
\end{align}
Therefore, for all $k \ge 0$,
\begin{align}\label{eq-lemma-steps-projection-8}
\|F(u_k(\omega))\| &\leq \|F(u_k(\omega)) - F(v^*)\| + \|F(v^*)\| \cr
&\leq \|u_k(\omega) - v^*\| (L_0+ L_1 \|F(v^*)\|) \exp (L_1 \|u_k(\omega) - v^*\|)+  \|F(v^*)\| . 
\end{align}
Since $ \omega \in A$, we have
$\|u_k(\omega) - v^*\| \leq  C(\omega)$ for all $k \geq 0$, 
which when used in~\plaineqref{eq-lemma-steps-projection-8}, implies that for all $k \ge 0$,
\begin{align}\label{eq-lemma-steps-projection-9}
\|F(u_k(\omega))\| &\leq  \|u_k(\omega) - v^*\| (L_0+ L_1 \|F(v^*)\|) \exp (L_1 \|u_k(\omega) - v^*\|) + \|F(v^*)\| \cr
&\leq C(\omega) (L_0+ L_1 \|F(v^*)\|) \exp (L_1 C(\omega) ) +  \|F(v^*)\| .
\end{align}
Hence, the sequence
$\{\|F(u_k(\omega))\|\}$ is upper bounded by $\overline{C}_1(\omega)$, where $\overline{C}_1(\omega) = C(\omega) (L_0+ L_1 \|F(v^*)\|) \exp (L_1 C(\omega) ) + \|F(v^*)\| $. 
Now, for both cases $\a \in (0,1)$ and $\alpha=1$ in~\plaineqref{eq-lemma-steps-projection-6} and~\plaineqref{eq-lemma-steps-projection-9}, respectively, we have that  $\|F(u_k(\omega))\|$ is upper bounded by $ \max \{ C_1(\omega), \overline{C}_1(\omega)\}$. Thus 
\[\mathbb{P}(F(u_k) \text{ is bounded}) = 1.\]

Then almost surely we have (i) $F(u_k)$ is bounded (ii) $\sum^{\infty}_k \beta_k (\mathbb{I}(A_k) - \mathbb{E}[ \mathbb{I}(A_k) \mid \mathcal{F}_{k-1}])$ converges to $S < \infty$, (iii) $\mathbb{E}[\mathbb{I}(A_k) \mid \mathcal{F}_{k-1}] \geq \frac{1}{2}$. Consider $\omega \in \Omega$ such that (i), (ii), and (iii) hold, then in a view of~\eqref{eq-lemma-steps-projection-3} we have

\begin{align}
\label{eq-lemma-steps-projection-10}
\sum_{k=0}^{\infty} \beta_k x_k(\omega) &\geq \min \left\{1, \frac{1}{\overline{C}_1(\omega) + 2 \sigma}\right\} S(\omega) + \frac{1}{2} \min \left\{1, \frac{1}{\overline{C}_1(\omega) + 2 \sigma}\right\}  \sum_{k=0}^{\infty} \beta_k  = \infty
\end{align}
where the last equality comes from $\sum_{k=0}^{\infty} \beta_k = \infty$, which concludes the proof.

\end{proof}

\textbf{Proof of Theorem~\ref{thm-proj-as-converge}}.
\begin{proof}
By Lemma~\ref{lemma-proj-as-bound}, we have
\begin{equation}
\label{eq-thm-grad-as-13}
\sum_{k=0}^{\infty}  \gamma_k  \, \dist^p(u_k, U^*) < \infty \quad \rm{a.s.}
\end{equation}
Due to Lemma~\ref{lemma-stepsizes-nonsummable-projection}, the series $\sum_{k=0}^{\infty} \gamma_k = \infty$ almost surely,  then it follows that 
\begin{equation}\label{eq-thm-grad-as-14}
\liminf_{k\to\infty}\dist^p(u_k, U^*) =0 \qquad a.s.
\end{equation}
Since $\|u_k - u^*\|$ converges {\it a.s.} for any given $u^* \in U^*$, the sequence $\{u_k\}$ is bounded {\it a.s.} and has accumulation points {\it a.s.} Let $\{k_{i}\}$ be an index sequence, such that 
\begin{equation}\label{eq-thm-grad-as-15}
\lim_{i \rightarrow \infty}  \dist^p(u_{k_i}, U^*) = \liminf_{k\to\infty}\dist^p(u_k, U^*) =0 \quad \it{ a. s.}
\end{equation}
We assume that the sequence $\{u_{k_i}\}$ is convergent with a limit point $\bar{u}$; otherwise, we choose a convergent subsequence. Therefore,
\begin{equation}\label{eq-thm-grad-as-16}
\lim_{i \rightarrow \infty}  \|u_{k_i} - \bar u \| = 0 \quad \it{ a. s.}
\end{equation}
Then,  by~\plaineqref{eq-thm-grad-as-14},  $\dist (\bar{u}, U^*) = 0$, thus $\bar u \in U^*$ {\it a.s.} since $U^*$ is closed. Since the sequence $\{\|u_k - u^*\|\}$ converges {\it a. s. } for all $u^* \in U^*$, by~\plaineqref{eq-thm-grad-as-16} we have
\begin{equation}\label{eq-thm-grad-as-17}
\lim_{k \rightarrow \infty}  \|u_{k} - \bar u \| = 0 \quad \it{a. s. }
\end{equation}
\end{proof}

\subsection{Proof of Lemma~\ref{lemma-EF-bound}}\label{sec-lemma-EF}
\begin{proof}
By taking the total expectation in~\eqref{eq-grad-3-1} in Lemma~\ref{lemma-proj-as-bound} and using the definition of the stepsize $\g_k$, 
we obtain for any solution $u^*\in U^*$ and all $k\ge0$,
\begin{align}\label{eq-lemma-grad-expect-1}
\mathbb{E}[\|u_{k+1}-u^*\|^2] &\leq \mathbb{E}[\|u_k -u^*\|^2] - 2 \mu \mathbb{E}[\g_k \dist^p(u_k, U^*)]  + 3\beta_k^2(2 \sigma^2 + 1).
\end{align}

The equation~\plaineqref{eq-lemma-grad-expect-1} satisfies the conditions of Lemma~\ref{lemma-polyak11-det} with 
\begin{align} \bar v_{k} = \mathbb{E} [\|u_{k}  - u^*\|^2], \quad  \bar a_k = 0, \quad  \bar z_k =  2  \mu \mathbb{E}[ \gamma_k \, \dist^p(u_{k}, U^*)], \quad  \bar b_k = 3\beta_k^2(2 \sigma^2 + 1).
\end{align}

Thus, by Lemma~\ref{lemma-polyak11-det}, it follows that the sequence $\{\mathbb{E} [\|u_{k}  - u^*\|^2]\}$ converges to a non-negative scalar for any $u^*\in U^*$.
Therefore, the sequence $\{\mathbb{E} [\|u_{k}  - u^*\|^2]\}$ is bounded  for all $u^* \in U^*$.  Next, using the property of $\a$-symmetric operators, we show that $\{\mathbb{E}[\|F(u_k)\|]\}$ is bounded. Let $v^* \in U^*$ be an arbitrary, but fixed solution.  Then,  by the $\alpha$-symmetric property of $F$, we have that 
\begin{equation}
\begin{aligned}
\label{eq-lemma-grad-expect-2}
\|F(u_k)\| &\leq \|F(u_k) - F(v^*)\| + \|F(v^*)\| \cr
&\leq \|u_k - v^*\| (K_0 + K_1 \|F(v^*)\|^{\alpha} + K_2 \|u_k - v^* \|^{\a / (1 - \a)}) + \|F(v^*)\| .
\end{aligned}
\end{equation}
Taking expectation, we obtain
\begin{equation}
\begin{aligned}
\label{eq-lemma-grad-expect-3}
\mathbb{E}[\|F(u_k)\|] \leq (K_0 + K_1 \|F(v^*)\|^{\alpha}) \mathbb{E}[\|u_k - v^*\|]  + K_2  \mathbb{E}[\|u_k - v^* \|^{1 + \a / (1 - \a)}] + \|F(v^*)\| .
\end{aligned}
\end{equation}
Notice, that $\mathbb{E}[\|u_k - v^* \|^{1 + \a / (1 - \a)})] = \mathbb{E}[(\|u_k - v^* \|^2)^{1/2(1 - \a)})]$, and for $\alpha \leq 1/2$, the quantity $1/2(1 - \a) \leq 1$. Thus, we can apply Jensen inequality for concave function
\[\mathbb{E}[(\|u_k - v^* \|^2)^{1/2(1 - \a)})] \leq \mathbb{E}[\|u_k - v^* \|^2]^{1/2(1 - \a)}.\]
Therefore, using these results and Jensen inequality for the first term in equation~\plaineqref{eq-lemma-grad-expect-3}, we obtain
\begin{equation}
\begin{aligned}
\label{eq-lemma-grad-expect-4}
\mathbb{E}[\|F(u_k)\|] \leq (K_0 + K_1 \|F(v^*)\|^{\alpha}) \mathbb{E}[\|u_k - v^*\|^2]^{1/2}  + K_2  \mathbb{E}[\|u_k - v^* \|^2]^{1/2(1 - \a)} + \|F(v^*)\| .
\end{aligned}
\end{equation}

Since $\mathbb{E}[\|u_k - v^* \|^2]$ is bounded, $\mathbb{E}[\|F(u_k)\|]$ is bounded by some constant $C_F > 0$ for all $k \geq 0$.
\end{proof}

\subsection{Proof of Theorem~\ref{thm-proj-rates}}
\label{sec-thm-proj-rates}
\begin{proof}
Letting $y =  P_{U^*}(u_k)$ in equation~\plaineqref{eq-grad-2} in Lemma~\ref{lemma-proj-as-bound} and using $p$-quasi sharpness we obtain
\begin{align}\label{eq-thm-grad-rates-1}
\|u_{k+1}-P_{U^*}(u_k)\|^2 
&\le \|u_k -P_{U^*}(u_k)\|^2 
- 2 \mu \g_k \dist^p(u_k, U^*) \cr
& \quad + 2\g_k \langle e_k, u_{k}-P_{U^*}(u_k)\rangle + 3 \beta_k^2 (\|e_k\|^2 + \|e_k^2\|^2 + 1). 
\end{align}
By the definition of the distance function, we have
\[\dist^2(u_{k+1}, U^*) \leq \|u_{k+1}-P_{U^*}(u_k)\|^2.\]
Thus,
\begin{align}\label{eq-thm-grad-rates-2}
\dist^2(u_{k+1}, U^*) 
&\le \dist^2(u_k, U^*) - 2 \mu  \g_k \dist^p(u_k, U^*) \cr
& \quad + 2\g_k \langle e_k, u_{k}-P_{U^*}(u_k)\rangle + 3 \beta_k^2 (\|e_k\|^2 + \|e_k^2\|^2 + 1). 
\end{align}
By Assumption~\ref{asum-samples} and the law of total expectation, and independence of samples $\xi_k$ and $\xi_k^2$, it follows that 
\begin{align}
\mathbb{E}[\g_k \langle e_k, u_{k}-P_{U^*}(u_k)\rangle] &= \mathbb{E}[\mathbb{E}[\g_k \langle e_k, u_{k}-P_{U^*}(u_k)\rangle \mid \mathcal{F}_{k-1}]] \cr
&=\mathbb{E}[\mathbb{E}[\g_k \mid \mathcal{F}_{k-1}]  \langle \mathbb{E}[e_k \mid |\mathcal{F}_{k-1}], u_{k}-P_{U^*}(u_k)\rangle ] \cr
&=0.
\end{align}
Also, we have $\mathbb{E}[\mathbb{E}[\|e_k^1\| \mid \mathcal{F}_{k-1}] \leq \sigma^2$ and $\mathbb{E}[\mathbb{E}[\|e_k^2\| \mid \mathcal{F}_{k-1}] \leq \sigma^2$. Thus, by taking the total expectation in~\eqref{eq-thm-grad-rates-2}, we obtain
\begin{align}\label{eq-thm-grad-rates-3}
\mathbb{E}[\dist^2(u_{k+1}, U^*)] 
&\le \mathbb{E}[\dist^2(u_k, U^*)] - 2 \mu \mathbb{E}[\g_k \dist^p(u_k, U^*)] + 3\beta_k^2 (2\sigma^2  +1). 
\end{align}
We aim to upper bound $2\mu \mathbb{E}[\g_k \dist^p(u_k, U^*)]$. To do so consider an event $A_k$, defined as follows:
\[A_k = \{\|F(u_k)\| + \|e_k\| \leq 2 (\mathbb{E}[\|F(u_k)\|] + \mathbb{E}[\|e_k\|]) \}.\]
Then, by the law of total expectation, we obtain
\begin{align}\label{eq-thm-grad-rates-4}
\mathbb{E}[\g_k \dist^p(u_k, U^*)] = \mathbb{E}[\g_k \dist^p(u_k, U^*)| A_k] \mathbb{P}(A_k) + \mathbb{E}[\g_k \dist^p(u_k, U^*)| \overline{A}_k] \mathbb{P}(\overline{A}_k),
\end{align}
where $\overline{A}$ denotes the complement of an event $A$.
We want to provide a lower bound on $\mathbb{P}(A_k)$. To do so, we upperbound $\mathbb{P}(\overline{A}_k)$ using Markov's inequality, as follows:
\begin{align}\label{eq-thm-grad-rates-5}
\mathbb{P}(\overline{A}_k) &= \mathbb{P}\left(\{\|F(u_k)\| + \|e_k\| > 2 (\mathbb{E}[\|F(u_k)\|] + \mathbb{E}[\|e_k\|]) \}\right) \cr
&\leq \frac{\mathbb{E}[\|F(u_k)\|] + \mathbb{E}[\|e_k\|]}{2(\mathbb{E}[\|F(u_k)\|] + \mathbb{E}[\|e_k\|])} \cr&= \frac{1}{2}.
\end{align}
Thus,
\begin{align}\label{eq-thm-grad-rates-6}
\mathbb{E}[\g_k \dist^p(u_k, U^*)] &= \mathbb{E}[\g_k \dist^p(u_k, U^*)| A_k] (1 - \mathbb{P}(\overline{A}_k)) + \mathbb{E}[\g_k \dist^p(u_k, U^*)| \overline{A}_k] \mathbb{P}(\overline{A}_k) \cr 
&\geq \frac{1}{2} \mathbb{E}[\g_k \dist^p(u_k, U^*)| A_k]+ \mathbb{E}[\g_k \dist^p(u_k, U^*)| \overline{A}_k] \mathbb{P}(\overline{A}_k) \cr
&\geq \frac{1}{2} \mathbb{E}[\g_k \dist^p(u_k, U^*)| A_k].
\end{align}
By the definition of the event $A_k$, we have
\begin{align}
\label{eq-thm-grad-rates-7}
\mathbb{E}[\g_k \dist^p(u_k, U^*)| A_k]  &= \beta_k \mathbb{E}\left[\min \left\{1, \frac{1}{\|\Phi(u_k, \xi_k)\|}\right\} \dist^p(u_k, U^*)| A_k\right] \cr
&\geq \beta_k \mathbb{E}\left[\min \left\{1, \frac{1}{\|F(u_k)\|+ \|e_k\|}\right\} \dist^p(u_k, U^*)| A_k]\right] \cr
&\geq \beta_k \min \left\{1, \frac{1}{2 (\mathbb{E}[\|F(u_k)\|]+ \mathbb{E}[\|e_k\|])}\right\}  \mathbb{E}[\dist^p(u_k, U^*)| A_k].
\end{align}
By Lemma~\ref{lemma-EF-bound}, $\mathbb{E}[\|F(u_k)\|] \leq C_F$ for all $k \geq 0$, and by Assumption~\ref{asum-samples} and Jensen inequality, we have $\mathbb{E}[\|e_k\|]\leq \mathbb{E}[\|e_k\|^{2}]^{1/2} \leq \sigma$. 
Thus, it follows that 
\begin{align}
\label{eq-thm-grad-rates-70}
\mathbb{E}[\g_k \dist^p(u_k, U^*)] 
\geq \frac{1}{2}\beta_k \min \left\{1, \frac{1}{2 (C_F+\sigma)}\right\}  \mathbb{E}[\dist^p(u_k, U^*)].\end{align}
Combining equations~\plaineqref{eq-thm-grad-rates-3} and~\plaineqref{eq-thm-grad-rates-70}, and using $a=\mu \min \left\{1, \frac{1}{2 (C_F+ \sigma)}\right\}$, we obtain
\begin{align}\label{eq-thm-grad-rates-8}
\mathbb{E}[\dist^2(u_{k+1}, U^*)] 
&\le \mathbb{E}[\dist^2(u_k, U^*)]  - a \beta_k \mathbb{E}[\dist^p(u_k, U^*)] + 3\beta_k^2 (2\sigma^2  +1). 
\end{align}
Now let $D_k = \mathbb{E}[\dist^2(u_k, U^*)]$, and consider the following two cases: 

\textbf{Case $p=2$}.
When $p=2$, equation~\plaineqref{eq-thm-grad-rates-8} satisfies the assumptions of Lemma~\ref{Lemma7-stich} with 
\begin{align}
r_{k} = D_k, 
\quad \alpha_k = \beta_k, \quad s_k = 0, \quad d=a, \quad c =3(2 \sigma^2  +1).
\end{align}
Then, by Lemma~\ref{Lemma7-stich}, 
we get the following convergence rate for all $k \geq 1$,
\begin{align}
 D_{k+1} \leq  \dfrac{8  D_0}{k^2} + \frac{6 (2 \sigma^2 + 1)}{ a^2 k}.
\end{align}

\textbf{Case $p> 2$}.
When $p \geq 2$, by applying telescoping sum to inequality~\plaineqref{eq-thm-grad-rates-8} and rearranging the terms we obtain
\begin{align}\label{eq-thm-grad-rates-9-new}
\mathbb{E}[ a \sum_{t=0}^k   \beta_k \dist^p(u_k, U^*)] 
&\le D_0  - D_{k+1} + 3 (2\sigma^2  +1) \sum_{t=0}^k \beta_k^2.  
\end{align}
Since $p \geq 2$, the function $\dist^p(\cdot, U^*)$ is convex, thus by defining $\bar{u}_k = (\sum_{t=0}^k \beta_k)^{-1} \sum_{t=0}^k \beta_k u_t$ and applying Jensen inequality be obtain
\[ (\sum_{t=0}^k \beta_k ) \mathbb{E}[\dist^p (\bar{u}_k, U^*)] \leq \mathbb{E}[ \sum_{t=0}^k   \beta_k \dist^p(u_k, U^*)] .\]
Since $p \geq 2$, by applying Jensen inequality one more time, we obtain
\[ (\bar{D}_k)^{p/2} = \left(\mathbb{E}[\dist^2(\bar{u}_k, U^*)]\right)^{p/2} \leq  \mathbb{E}\left[\left(\dist^2(\bar{u}_k, U^*)\right)^{p/2}\right] = \mathbb{E}[\dist^p(\bar{u}_k, U^*)]
.\]

Applying these estimates, we get
\begin{align}\label{eq-thm-grad-rates-11-new}
(\bar{D}_{k})^{p/2} \sum_{t=0}^{k} \beta_t \leq \sum_{t=0}^{k} \beta_t   D_t^{p/2}
&\le \frac{1}{a}\,\left(D_0 - D_{k+1}  + 3(2\sigma^2  +1)\sum_{t=0}^k \beta_t^2\right).
\end{align}

Since $\beta_k = \frac{b}{(k+1)^q}$, with $b > 0, 1 >q > 1/2$,
then $\{\beta_k\}$ satisfies the conditions of Lemma~\ref{lemma-EF-bound}. Also, by Lemma~\ref{lemma-kq-series-rates} the following inequalities hold: for all $k\ge1,$
\begin{align}
\label{eq-thm-grad-rates-12-new}
\sum^{k}_{t = 0} \beta_t \geq \frac{b}{1 - q} ((k+1)^{1 - q} - 2^{1 - q}),\qquad\quad
\sum^{k}_{t = 0} \beta_t^2 \leq \frac{b^2}{2 q -1}.
\end{align}
Combining equations~\plaineqref{eq-thm-grad-rates-11-new} and ~\plaineqref{eq-thm-grad-rates-12-new}, and omitting $D_{k+1}$, we obtain
\begin{align}\label{eq-thm-grad-rates-13-new}
(\bar{D}_{k})^{p/2} 
&\le \dfrac{(1 -q) \left(D_0  + 3b^2 (2 \sigma^2 +1 ) / (2 q -1)\right)}{ a b \, \left((k+1)^{1 - q} - 2^{1 - q}\right)}.
\end{align}
Raising both sides of the preceding inequality in power $2/p$, we obtain
\begin{align}\label{eq-thm-grad-rates-14-new}
\bar{D}_{k}
&\le \dfrac{(1 - q)^{2/p} 
\left(D_0  +3 b^2 (2 \sigma^2 +1 ) / (2 q -1)\right)^{2/p}}{ (ab)^{2/p} \, \left((k+1)^{1 - q} - 2^{1 - q}\right)^{2/p}}.
\end{align}

\end{proof}

\section{Korpelevich Method analysis}
\label{Appendix-Korpelevich}
\begin{lemma} \label{Lemma-korpelevich-basic} 
Let $U$ be a closed convex set. Then, for the iterate sequences
$\{u_k\}$ and $\{h_k\}$ generated by the stochastic Korpelevich method~\plaineqref{eq-korpelevich-stoch} and $y \in U$ and $k \ge 0$,
\begin{equation*}
\begin{aligned}
 \|h_{k+1}  - y\|^2 &\leq \|h_k - y\|^2  - \|h_k - u_k\|^2 - 2 \g_k \langle F(u_k), u_k - y \rangle  - 2 \g_k \langle e_k^2, u_k - y \rangle \cr
 & \ \ + 3 \g_k^2 \|F(h_k) - F(u_k)\|^2 + 3 \g_k^2 (\|e_k^2\|^2 +  \|e_k^1\|^2),
\end{aligned}
\end{equation*}
where  $e_{k}^1 = \Phi(h_k, \xi_k^1) - F(h_k)$, $e_{k}^2 = \Phi(u_k, \xi_k^2) - F(u_k)$ for all $k \ge 0$.
\end{lemma}

\begin{proof}
Let $k\ge 0$ be arbitrary but fixed. By the definition of $h_{k+1}$ in~\plaineqref{eq-korpelevich-stoch}, we have 
$\|h_{k+1} - y\| = \|P_{U} (h_{k} - \g_k \Phi(u_k, \xi^2_k)) - y\|$ for any $y \in U$. Using the projection inequality,
we obtain for any $y\in U$,
\begin{equation}
\begin{aligned}
\label{eq-lemma-korp-2}
 \|h_{k+1}  - y\|^2 & \leq \|h_{k} - \g_k \Phi(u_k, \xi^2_k) - y\|^2 - \|h_{k+1} -h_{k} + \g_k \Phi(u_k, \xi^2_k)\|^2\\
& = \|h_k - y\|^2 - \| h_{k+1} - h_k\|^2 + 2 \g_k \langle \Phi(u_k, \xi^2_k), y -  h_{k+1}\rangle.
\end{aligned}
\end{equation}
Next, we consider the term $\|h_{k+1} - h_k\|^2$, where we add and subtract $u_k$, thus
\begin{equation}
\begin{aligned}
\label{eq-lemma-korp-3}
 \|h_{k+1}  - h_k\|^2 & = \|h_{k+1} - u_k\|^2 + \|h_{k} - u_k\|^2 - 2 \langle h_{k+1} - u_k, h_k - u_k \rangle.
\end{aligned}
\end{equation}
Adding and subtracting $ 2 \g_k \langle \Phi(h_k, \xi^1_k), u_k-h_{k+1}\rangle$, and combining~\plaineqref{eq-lemma-korp-2} and~\plaineqref{eq-lemma-korp-3} we obtain
\begin{equation}
\begin{aligned}
\label{eq-lemma-korp-4}
 \|&h_{k+1}  - y\|^2  \leq \|h_k - y\|^2 - \| h_{k+1} - u_k \|^2 - \|h_k - u_k\|^2 +2 \langle h_{k+1} - u_k, h_k - u_k \rangle \cr
 & \ \ + 2 \g_k \langle \Phi(u_k, \xi^2_k), y - u_k + u_k-  h_{k+1}\rangle  + 2 \g_k \langle \Phi(h_k, \xi^1_k) - \Phi(h_k, \xi^1_k), u_k-  h_{k+1}\rangle\\
 &\leq \|h_k - y\|^2 - \| h_{k+1} - u_k \|^2 - \|h_k - u_k\|^2 +2 \langle h_{k+1} - u_k, h_k - \g_k \Phi(h_k, \xi^1_k) - u_k \rangle \cr
 & \ \ + 2 \g_k \langle \Phi(u_k, \xi^2_k), y - u_k \rangle + 2\g_k   \langle \Phi(h_k, \xi^1_k)- \Phi(u_k, \xi^2_k), h_{k+1} - u_k \rangle.
\end{aligned}
\end{equation}
Since $u_{k} =  P_{U}(h_k - \g_k \Phi(h_k, \xi_k^1))$ and $h_{k+1}\in U$, by the projection inequality in~\eqref{eq-proj1}, it follows that
\[2 \langle h_{k+1} - u_k, h_k - \g_k \Phi(h_k, \xi^1_k) - u_k \rangle \leq 0 .\]
Using Cauchy-Schwarz inequality and relation $2ab \leq a^2 + b^2$ for $a,b \in \mathbb{R}$, we obtain
\begin{equation*}
\begin{aligned}
2\g_k   \langle \Phi(h_k, \xi^1_k) - \Phi(u_k, \xi^2_k), h_{k+1} - u_k \rangle   
&\leq 2 \g_k \|\Phi(h_k, \xi^1_k) - \Phi(u_k, \xi^2_k) \| \|h_{k+1} - u_k\| \cr
&\leq \g_k^2 \|\Phi(h_k, \xi^1_k) - \Phi(u_k, \xi^2_k)\|^2 + \|h_{k+1} - u_k\|^2.
\end{aligned}
\end{equation*}
Using triangle inequality and relation $(\sum_{i=1}^m a_i)^2 \leq m \sum_{i=1}^m a_i^2$ we get
\begin{equation*}
\begin{aligned}
\|\Phi(h_k, \xi^1_k) - \Phi(u_k, \xi^2_k)\|^2 
& = \|\Phi(h_k, \xi^1_k) - F(h_k) + F(h_k) - F(u_k) + F(u_k) - \Phi(u_k, \xi^2_k)\|^2 \cr
&\leq 3 (\|e_k^1\|^2 + \|F(h_k) - F(u_k)\|^2 + \|e_k^2\|^2).
\end{aligned}
\end{equation*}
Combining the preceding three estimates with~\plaineqref{eq-lemma-korp-4}, we get the stated relation
\begin{equation*}
\begin{aligned}
 \|h_{k+1}  - y\|^2 &\leq \|h_k - y\|^2  - \|h_k - u_k\|^2 - 2 \g_k \langle F(u_k), u_k - y \rangle  - 2 \g_k \langle e_k^2, u_k - y \rangle \cr
 & \ \ + 3 \g_k^2 \|F(h_k) - F(u_k)\|^2 + 3 \g_k^2 (\|e_k^2\|^2 +  \|e_k^1\|^2).
\end{aligned}
\end{equation*}
\end{proof}

\subsection{Proof of Lemma~\ref{Lemma-Korpelevich}} \label{sec-lem41}
\begin{proof}
By Lemma~\ref{Lemma-korpelevich-basic} we have for all $k \geq 0$ and for all $y \in U$,
\begin{equation}
\begin{aligned}
\label{eq-lm2-korp-1}
 \|h_{k+1}  - y\|^2 &\leq \|h_k - y\|^2  - \|h_k - u_k\|^2 - 2 \g_k \langle F(u_k), u_k - y \rangle  - 2 \g_k \langle e_k^2, u_k - y \rangle \cr
 & \ \ + 3 \g_k^2 \|F(h_k) - F(u_k)\|^2 + 3 \g_k^2 (\|e_k^2\|^2 +  \|e_k^1\|^2),
\end{aligned}
\end{equation}
with  $e_{k}^1 = \Phi(h_k, \xi_k^1) -F(h_k)$ and $e_{k}^2 = \Phi(u_k, \xi_k^2) - F(u_k) $ for all $k\ge0$.
We want to estimate the term $\|F(h_k) -  F(u_{k}) \|^2$ on the RHS of the inequality using the fact that $F(\cdot)$ is an $\a$-symmetric operator for two cases\textbf{ (a) $\alpha \in (0,1)$} and \textbf{(b) $\alpha  = 1$}. 

\textbf{Case $\a \in (0,1)$}. 
Using the alternative characterization of $\a$-symmetric operators from Proposition~\ref{prop-a}(a) (as given in ~\plaineqref{alpha-property}), when $\a \in (0,1)$, the next inequality holds for any $k\ge0$,
\begin{equation}
\begin{aligned}
\label{eq-lm2-korp-2}
\|F(h_k) - F(u_k)\| \leq \|h_k - u_k\| (K_0 + K_1 \|F(h_k)\|^{\alpha} + K_2 \|h_k - u_k \|^{\a / (1 - \a)}).
\end{aligned}
\end{equation}
We want to separate $\|F(h_{k})\|$ into two parts: stochastic approximation of operator $\Phi(h_{k}, \xi_{k}^1)$ and error $e_{k}^1$. Recall that  $e_k^1 = F(h_{k}) - \Phi(h_{k}, \xi_{k}^1)$, then based on triangle inequality $\|F(h_{k})\| \leq \|\Phi(h_{k}, \xi_{k}^1)\| + \|e_{k}^1\|$, and since $\a \leq 1$ we obtain 
\begin{align}
\label{eq-lm2-korp-2-1}
\|F(h_{k})\|^{\a} \leq \|\Phi(h_{k}, \xi_{k}^1)\|^{\a} + \|e_{k}^1\|^{\a}.
\end{align}
Thus, combining this fact with~\plaineqref{eq-lm2-korp-2} we get the following estimation
\begin{equation}
\begin{aligned}
\label{eq-lm2-korp-3}
\|F(h_k) - F(u_k)\| \leq \|h_k - u_k\| (K_0 + K_1 \|\Phi(h_k, \xi_k^1)\|^{\alpha} + K_1\|e_k^1\|^{\alpha} + K_2 \|h_k - u_k \|^{\a / (1 - \a)}).
\end{aligned}
\end{equation}

By the projection property~\plaineqref{eq-proj3} and the stepsize choice~\plaineqref{stepsizes-korpelevich}, we have
\begin{align}
\label{eq-lm2-korp-4}
\| h_{k} - u_k\| \leq \g_k \| \Phi(h_{k}, \xi_k^1)\| = \beta_k \min\{1, \frac{1}{\|\Phi(h_k, \xi_k^1)\|}\} \|\Phi(h_k, \xi_k^1)\| \leq \beta_k \leq 1.
\end{align}
Then, $ K_2 \|h_k - u_k \|^{\a / (1 - \a)} \leq K_2$, and
\begin{equation}
\begin{aligned}
\label{eq-lm2-korp-5}
\gamma_k \|F(h_k) &- F(u_k)\| \leq \gamma_k (K_0 + K_1 \|\Phi(h_k, \xi_k^1)\|^{\alpha} + K_1\|e_k^1\|^{\alpha} + K_2 ) \|h_k - u_k\| \cr
&\leq \beta_k ( K_0 \min \{1, \frac{1}{\|\Phi(h_{k}, \xi_k^1)\|}\} +  K_1 \min \{1, \frac{1}{\|\Phi(h_{k}, \xi_k^1)\|}\} \|\Phi(h_k, \xi_k^1)\|^{\alpha} ) \|h_k - u_k\| \cr
& \ \ +\beta_k(K_1\|e_k^1\|^{\alpha} + K_2 \min \{1, \frac{1}{\|\Phi(h_{k}, \xi_k^1)\|}\} ) \|h_k - u_k\| \cr
&\leq \beta_k ( K_0 +  K_1  +  K_2  ) \|h_k - u_k\| + \beta_k K_1\|e_k^1\|^{\alpha} \|h_k - u_k\|.
\end{aligned}
\end{equation}
By inequality~\plaineqref{eq-lm2-korp-4}, we have $\|h_k - u_k\| \leq 1$, and using this estimate in equation~\plaineqref{eq-lm2-korp-5} we obtain
\begin{equation}
\begin{aligned}
\label{eq-lm2-korp-6}
\gamma_k \|F(h_k) - F(u_k)\|  &\leq \beta_k ( K_0 +  K_1  +  K_2  ) \|h_k - u_k\| + \beta_k K_1\|e_k^1\|^{\alpha}. \cr
\end{aligned}
\end{equation}

\textbf{Case $\a=1$}.
Based on the alternative characterization of $\a$-symmetric operators from Proposition~\ref{prop-a}(b) (as given in ~\plaineqref{stepsizes-korpelevich}), when $\a = 1$, the following inequality holds for any $k\ge0$,
\begin{equation}
\begin{aligned}
\label{eq-lm2-korp-7}
\|F(h_{k}) -  F(u_k)\| &\leq \|h_k - u_k\| (L_0+ L_1 \|F(h_k)\|) \exp (L_1 \|h_{k} - u_k\|).
\end{aligned}
\end{equation}
We upperbound $\|F(h_k)\|$ using equation~\plaineqref{eq-lm2-korp-2-1}  and get
\begin{equation}
\begin{aligned}
\label{eq-lm2-korp-8}
\|F(h_{k}) -  F(u_k)\| &\leq \|h_k - u_k\| (L_0+ L_1 \|\Phi(h_{k}, \xi_k^1)\| + L_1 \|e_k^1\|) \exp (L_1 \|h_{k} - u_k\|).
\end{aligned}
\end{equation}
Note that relation in~\eqref{eq-lm2-korp-4} holds irrespective of the value of $\a$. Thus, since $\|h_k-u_k\|\le 1$,
we have $\exp(L_1 \| h_{k} - u_k\|) \leq \exp(L_1 \beta_k)$, and  we obtain
\begin{equation}
\begin{aligned}
\label{eq-lm2-korp-9}
\gamma_k \|F(h_{k}) -  F(u_k)\| &\leq  \gamma_k (L_0+ L_1 \|\Phi(h_{k}, \xi_k^1)\| + L_1 \|e_k^1\|) \exp (L_1 \beta_k) \|h_k - u_k\| \cr
&=  \exp (L_1 \beta_k)  L_0  \beta_k \min\{1, \frac{1}{\|\Phi(h_{k}, \xi_k^1)\|} \}  \|h_k - u_k\| \cr
& \ \ + \exp (L_1 \beta_k) L_1  \beta_k\min\{1, \frac{1}{\|\Phi(h_{k}, \xi_k^1)\|} \} \| \Phi(h_{k}, \xi_k^1)\| \|h_k - u_k\| \cr
& \ \ +  \exp (L_1 \beta_k) L_1 \beta_k\min\{1, \frac{1}{\|\Phi(h_{k}, \xi_k^1)\|} \} \|e_k^1\| \|h_k - u_k\| \cr
&\leq \exp (L_1 \beta_k) \beta_k ( L_0 + L_1 + L_1 \|e_k^1\|) \|h_k - u_k\| .
\end{aligned}
\end{equation}
By inequality~\plaineqref{eq-lm2-korp-4}, we have $\|h_k - u_k\| \leq 1$. Using this estimate in~\plaineqref{eq-lm2-korp-9}, we further obtain
\begin{equation}
\begin{aligned}
\label{eq-lm2-korp-10}
\gamma_k \|F(h_k) - F(u_k)\|  &\leq \exp (L_1 \beta_k) \beta_k ( L_0 + L_1 ) \|h_k - u_k\| + \exp (L_1 \beta_k) \beta_k \|e_k^1\|. \cr
\end{aligned}
\end{equation}

Now, we are done with the cases of $\a$ values. Let 
\begin{align}
\label{eq-lm2-korp-c-a-def}
C_{a}(\beta_k, \alpha) = \begin{cases}
(K_0 + K_1 + K_2), 
\quad \text{when } \alpha \in (0,1), \\
\exp (L_1 \beta_k) ( L_0 + L_1),
\quad \text{when } \alpha = 1.
\end{cases}
\end{align}
Also, define 
\begin{align}
\label{eq-lm2-korp-c-e-def}
C_{e}(\beta_k, \alpha) = \begin{cases}
K_1, 
\quad \text{when } \alpha \in (0,1), \\
\exp (L_1 \beta_k),
\quad \text{when } \alpha = 1.
\end{cases}
\end{align}
Then, by inequality $(\sum_{i=1}^m a_i)^2\le m\sum_{i=1}^m a_i^2$, for both cases we have
\begin{align}
\label{eq-lm2-korp-11}
\g_k^2 \|F(h_{k}) -  F(u_k)\|^2 \leq 2 \beta_k^2 C_a(\beta_k, \alpha)^2 \|h_k - u_k\|^2 + 2 \beta_k^2 C_e(\beta_k, \alpha)^2 \|e_k^1\|^{2 \alpha}.
\end{align}
Combining preceding inequality with~\plaineqref{eq-lm2-korp-1} we obtain that for any $k \ge 0$,
\begin{equation}
\begin{aligned}
\label{eq-lm2-korp-12}
 \|h_{k+1}  - y\|^2 &\leq \|h_k - y\|^2  - (1 - 6 \beta_k^2 C_a(\beta_k, \alpha)^2) \|h_k - u_k\|^2 - 2 \g_k \langle F(u_k), u_k - y \rangle \cr
 & \ \ - 2 \g_k \langle e_k^2, u_k - y \rangle  + 6  \beta_k^2 C_e(\beta_k, \alpha)^2 \|e_k^1\|^{2 \alpha} + 3 \g_k^2 (\|e_k^2\|^2 +  \|e_k^1\|^2) .\cr
\end{aligned}
\end{equation}

Next, we plug $y=u^*$, where $u^* \in U^*$ is an arbitrary solution and use $p$-quasi sharpness of the operator $F$ to obtain
\begin{equation}
\begin{aligned}
\label{eq-lm2-korp-13}
 \|h_{k+1}  - u^*\|^2 
 &\leq \|h_k - u^*\|^2  - (1 - 6 \beta_k^2 C_a(\beta_k, \alpha)^2) \|h_k - u_k\|^2 - 2 \g_k \mu \dist^p(u_k, U^*) \cr
 & \ \ - 2 \g_k \langle e_k^2, u_k - u^* \rangle  + 6  \beta_k^2 C_e(\beta_k, \alpha)^2 \|e_k^1\|^{2 \alpha} + 3 \g_k^2 (\|e_k^2\|^2 +  \|e_k^1\|^2). \cr
\end{aligned}
\end{equation}
By the stepsize choice $\gamma_k \leq \beta_k$, thus
\begin{equation}
\begin{aligned}
\label{eq-lm2-korp-14}
 \|h_{k+1}  - u^*\|^2 &\leq \|h_k - u^*\|^2  - (1 - 6 \beta_k^2 C_a(\beta_k, \alpha)^2) \|h_k - u_k\|^2 - 2 \g_k \mu \dist^p(u_k, U^*) \cr
 & \ \ - 2 \g_k \langle e_k^2, u_k - u^* \rangle  + 6  \beta_k^2 C_e(\beta_k, \alpha)^2 \|e_k^1\|^{2 \alpha} + 3 \beta_k^2 (\|e_k^2\|^2 +  \|e_k^1\|^2).
\end{aligned}
\end{equation}

Since $\sum_{k=0}^{\infty} \beta_k^2 < \infty$, it follows that $\beta_k \to 0$. By definitions of $C_a(\beta_k, \alpha)$ and $C_e(\beta_k, \alpha)$ in~\plaineqref{eq-lm2-korp-c-a-def} and~\plaineqref{eq-lm2-korp-c-e-def}, respectively, there exists $N \ge 0$ such that the stepsizes satisfy $1 - 6 \beta_k^2 C_a(\beta_k, \alpha)^2 \ge \frac{1}{2} $ and $C_e(\beta_k, \alpha)^2 \leq \max \{K_1, \exp(L_1 \beta_k )\} \leq \max \{K_1, 2\} $. Thus, the following inequality holds for any $k \ge N$,
\begin{equation}
\begin{aligned}
\label{eq-lm2-korp-15}
 \|h_{k+1}  - u^*\|^2 &\leq \|h_k - u^*\|^2  - \frac{1}{2}\|h_k - u_k\|^2 - 2 \g_k \mu \dist^p(u_k, U^*) \cr
 & \ \ - 2 \g_k \langle e_k^2, u_k - u^* \rangle  + 3 \beta_k^2 (\|e_k^2\|^2 +  \|e_k^1\|^2 + 2 \max \{K_1, 2\}  \|e_k^1\|^{2 \alpha}). \cr
\end{aligned}
\end{equation}
By rearranging the terms, defining $v_k = \|h_k - u^*\|^2 + \frac{1}{2}\|h_{k-1} - u_{k-1}\|^2 + 2 \g_k \mu \dist^p(u_k, U^*) $ and adding, and substituting $\frac{1}{2}\|h_{k-1} - u_{k-1}\|^2 + 2 \g_{k-1} \mu \dist^p(u_{k-1}, U^*) $  on the RHS  we obtain
\begin{equation}
\begin{aligned}
\label{eq-lm2-korp-15-1}
v_{k+1} &\leq v_k  - \frac{1}{2}\|h_{k-1} - u_{k-1}\|^2 - 2 \g_{k-1} \mu \dist^p(u_{k-1}, U^*) \cr
 & \ \ - 2 \g_k \langle e_k^2, u_k - u^* \rangle  + 3 \beta_k^2 (\|e_k^2\|^2 +  \|e_k^1\|^2 + 2 \max \{K_1, 2\}  \|e_k^1\|^{2 \alpha}). \cr
\end{aligned}
\end{equation}

Recalling that $e_k^1=\Phi(h_k,\xi_k^1)-F(h_k), \; e_k^2=\Phi(u_k,\xi_k^2)-F(u_k)$ and  
using the stochastic properties of $\xi_k^1, \xi_k^2$ imposed by Assumption~\ref{asum-samples} and method's updates, we have 
\[\mathbb{E}[\gamma_k \langle e_k^2, u_k - u^* \rangle \mid \mathcal{F}_{k-1}] = \mathbb{E}[\gamma_k  \langle \mathbb{E}[e_k^2 \mid \mathcal{F}_{k-1} \cup \{\xi_k^1\}], u_k - u^* \rangle \mid \mathcal{F}_{k-1}] = 0
,\]
since stepsizes $\gamma_k$ is measurable in $\mathcal{F}_{k-1} \cup \{\xi_{k}^1\}$. Also, it holds that for all $k \geq 0$,
 \[\mathbb{E}[\mathbb{E}[ \|e_k^2\|^2 \mid \mathcal{F}_{k-1} \cup \{\xi_k^1\}]| \mathcal{F}_{k-1}] \leq \sigma^2, \quad \text{and} \quad \mathbb{E}[ \|e_k^1\|^2 \mid \mathcal{F}_{k-1}] \leq \sigma^2 .\]
Moreover, since $\a \leq 1$, the conditional expectation $\mathbb{E}[\|e_k^1\|^{2\a}|\mathcal{F}_{k-1}]$ is finite, and by Jensen inequality, it follows that for all $k \ge 0$,
\[\mathbb{E}[\|e_k^1\|^{2\a}|\mathcal{F}_{k-1}] \le \sigma^{2 \alpha}.\]

Therefore, by taking the conditional expectation on $\mathcal{F}_{k-1}$
in relation~\plaineqref{eq-lm2-korp-15-1}, we obtain 
for all $u^*\in U^*$ and for all $k\ge N$,
\begin{equation}
\begin{aligned}
\label{eq-lm2-korp-16}
 \mathbb{E}[ v_{k+1} \mid \mathcal{F}_{k-1}] &\leq v_k  - \frac{1}{2}\|h_k - u_k\|^2 - 2 \mu \g_{k-1}  \dist^p(u_k, U^*) \cr
 & \ \ + 6 \beta_k^2 (  \sigma^2 + \max \{K_1, 2\} \sigma^{2 \alpha}). 
\end{aligned}
\end{equation}
By Lemma~\ref{lemma-polyak11}, it follows that the sequence $\{\|h_{k}- u^*\|^2\}$ converges {\it a.s.}\ to a non-negative scalar for any $u^*\in U^*$, and  almost surely we have
\begin{equation}
\label{eq-lm2-korp-17}
\sum_{k=0}^{\infty}  \gamma_k \dist^p(u_k, U^*)  \,  < \infty, \quad \sum_{k=0}^{\infty} \|h_k - u_k\|^2 < \infty.
\end{equation}
Since the sequence $\{\|h_k - u^*\|^2\}$ 
converges {\it a.s.} for all $u^*\in U^*$, it follows that the sequence $\{\|h_k - u^*\|\}$ is bounded {\it a.s.} for all $u^* \in U^*$.

\end{proof}

\subsection{Proof of Theoreom~\ref{thm-korpelevich-as-converge}}\label{sec-Thm42}

\begin{lemma}
\label{lemma-stepsizes-nonsummable-korpelevich}
    The stepsizes $\gamma_k $ are given by \eqref{stepsizes-korpelevich} are nonsummable almost surely,
    \begin{equation}
        \sum_{k=0}^{\infty} \gamma_k = \infty \quad \rm{a.s.}
    \end{equation}
\end{lemma}
\begin{proof}
We will show that $\sum_{k=0}^{\infty} \beta_k  \min \left\{1, \frac{1}{\|\Phi(h_k, \xi_k^1)\|} \right\} = \infty$ almost surely by the sequences of lower bound on this series. Consider the following event:
\[A_k = \{ \|e_k^1\| \leq 2 \sigma \},\]
where $e_k^1 = \Phi(h_k, \xi_k^1) - F(h_k)$ is a stochastic error from the sample for the clipping stepsize $\gamma_k$. Define $x_k = \min \left\{1, \frac{1}{\|\Phi(h_k, \xi_k^1)\|} \right\}$, then, 
\begin{align}\label{eq-thm-korpelevich-as-3}
x_k  = x_k  \mathbb{I}(A_k) + x_k \mathbb{I}(\overline{A}_k) \geq x_k  \mathbb{I}(A_k),
\end{align}
where the random variable $\mathbb{I}(A_k)$ is the indicator function of the event $A_k$ taking value 1 when the event occurs, and taking value 0 otherwise.

By the definition of $x_k$, the triangle inequality and definition of $\mathbb{I}(A_k)$, we have
\begin{align}
\label{eq-thm-korpelevich-as-6}
x_k \mathbb{I}(A_k)  &=  \min \left\{1, \frac{1}{\|\Phi(h_k, \xi_k^1)\|} \right\} \mathbb{I}(A_k) \cr
&\geq  \min \left\{1, \frac{1}{\|F(h_k)\|+ \|e_k^1\|}\right\} \mathbb{I}(A_k) \cr
&\geq  \min \left\{1, \frac{1}{\|F(h_k)\| + 2 \sigma}\right\}  \mathbb{I}(A_k) .\cr
\end{align}
Adding and subtracting $\mathbb{E}[\mathbb{I}(A_k) | \mathcal{F}_{k-1}]$ and combining \eqref{eq-thm-korpelevich-as-3}, \eqref{eq-thm-korpelevich-as-6} we have the following lower bound
\begin{align}
\label{eq-thm-korpelevich-as-6-1}
\sum_{k=0}^{\infty} \beta_k x_k \geq &
 \sum_{k=0}^{\infty} \beta_k  \min \left\{1, \frac{1}{\|F(h_k)\| + 2 \sigma}\right\}  (\mathbb{I}(A_k) - \mathbb{E}[\mathbb{I}(A_k) | \mathcal{F}_{k-1}] )\cr
 & + \sum_{k=0}^{\infty} \beta_k  \min \left\{1, \frac{1}{\|F(h_k)\| + 2 \sigma}\right\} \mathbb{E}[\mathbb{I}(A_k) | \mathcal{F}_{k-1}]. 
\end{align}

To bound $p_k := \mathbb{E}[\mathbb{I}(A_k) | \mathcal{F}_{k-1}] = \mathbb{P}(A_k \mid \mathcal{F}_{k-1})$ we  provide an upperbound on $\mathbb{P}(\overline{A}_k \mid \mathcal{F}_{k-1})$ using Markov's inequality and Assumption~\ref{asum-samples}:
\begin{align}\label{eq-thm-korpelevich-as-4}
\mathbb{P}(\overline{A}_k \mid \mathcal{F}_{k-1}) = \mathbb{P}( \|e_k^1\| > 2 \mathbb{E}[\|e_k^1\| \mid \mathcal{F}_{k-1}] \}) \leq \frac{\mathbb{E}[\|e_k^1\| \mid \mathcal{F}_{k-1}]}{2 \mathbb{E}[\|e_k^1\| \mid \mathcal{F}_{k-1}])} = \frac{1}{2}.
\end{align}
This implies $\mathbb{E}[\mathbb{I}(A_k) | \mathcal{F}_{k-1}] \geq \frac{1}{2}$. Define $S_n = \sum_{k=0}^{n} \beta_k (\mathbb{I}(A_k) - \mathbb{E}[\mathbb{I}(A_k) | \mathcal{F}_{k-1}] )$, by construction, $\{S_n\}$ is a martingale: 
\[ \mathbb{E}[S_{n+1} \mid S_0, \ldots, S_n] = S_n + \mathbb{E}[\beta_{n+1} (\mathbb{I}(A_{n+1}) - \mathbb{E}[\mathbb{I}(A_{n+1}) | \mathcal{F}_{n}] ) \mid S_0, \ldots, S_n] = S_n. \]
We want to show that $\lim_{n \rightarrow \infty} S_n \rightarrow S < \infty$ almost surely. We provide an upper bound for $\mathbb{E}[S_n^2]$:
\begin{align}
\mathbb{E}[S_n^2] = \sum_{k=0}^{n} \beta_k^2 \mathbb{E}[(\mathbb{I}(A_k) - p_k )^2] + 2\sum_{0 \leq k < i \leq n} \beta_k^2 \mathbb{E}[ (\mathbb{I}(A_k) - p_k ) (\mathbb{I}(A_i) - p_i )]
\end{align}
By the law of total expectation,and noting that $ \mathbb{E}[\mathbb{I}(A_k) - p_k \mid \mathcal{F}_{k-1}] = 0 $ for any $k$, we find that
for all $0 \leq k < i \leq n$,
\begin{align}
\mathbb{E}[ (\mathbb{I}(A_k) - p_k ) (\mathbb{I}(A_i) - p_i )] = \mathbb{E}[(\mathbb{I}(A_k) - p_k )  \mathbb{E}[(\mathbb{I}(A_i) - p_i )\mid \mathcal{F}_{i-1}]]  = 0,
\end{align}
{implying that, for all $n\ge0$,}
\begin{align}
\mathbb{E}[S_n^2] = \sum_{k=0}^{n} \beta_k^2 \mathbb{E}[(\mathbb{I}(A_k) - p_k )^2] 
\end{align}
{Since $ \mathbb{E}[(\mathbb{I}(A_k) - p_k )^2 \mid \mathcal{F}_{k-1}]  = {\rm Var}(I(A_k)  \mid \mathcal{F}_{k-1})$ and the random variable} $\mathbb{I}(A_k) $ is a Bernoulli given $ \mathcal{F}_{k-1}$ with mean $p_k$, its variance cannot exceed 1/4, i.e., 
\[ 
\mathbb{E}[(\mathbb{I}(A_k) - p_k)^2 \mid \mathcal{F}_{k-1}] = \rm{Var}(\mathbb{I}(A_k)  \mid \mathcal{F}_{k-1}) \leq \frac{1}{4}.
\]
By taking the total expectation we get $\mathbb{E}[(\mathbb{I}(A_k) - p_k)^2 ] \leq \frac{1}{4}$, and combining the preceding two relations, we obtain
\[\mathbb{E}[S_n^2] \leq \frac{1}{4} \sum_{k=0}^{n} \beta_k^2 \leq \infty.
\]
From Theorem 4.4.6. in \cite{durrett2019probability} it follows that $S_n$ converges to $S < \infty$ almost surely.

To further lower bound $x_k \mathbb{I}(A_k)$ we show \emph{a.s.} boundedness of $\|F(h_k)\|$  for all $k \geq 0$, using property of $\a$-symmetric operators. To estimate $\|F(h_k)\|$, we add and subtract $F(v^*)$, where $v^* \in U^*$ is an arbitrary but fixed solution, and get 
\[\|F(h_{k})\| = \|F(h_{k}) - F(v^*) + F(v^*)\| \leq \|F(h_{k}) - F(v^*)\| + \|F(v^*)\|.\]
Define the following event:
\[A = \{\omega \in \Omega: \; \exists \;  C(\omega) \in \mathbb{R} \text{ s.t.} \|h_k(\omega) - v^*\| < C(\omega) \; \forall \; k \geq 0 \}.\]
Based on Lemma~\ref{Lemma-Korpelevich}, the sequence $\{\|h_k - v^*\|\}$ is bounded {\it a.s.}, and thus $\mathbb{P}(A) = 1$. Let $\omega \in A$, now we can estimate $\|F(h_k(\omega))\|$ using the $\alpha$-symmetric assumption on the operator. \\
\textbf{Case $\alpha \in (0,1)$}.\\
\begin{equation}
\begin{aligned}
\label{eq-thm-korpelevich-as-7}
\|F(h_k(\omega)) - F(u^*)\| \leq \|h_k(\omega) - v^*\| (K_0 + K_1 \|F(v^*)\|^{\alpha} + K_2 \|h_k(\omega) - v^* \|^{\a / (1 - \a)}).
\end{aligned}
\end{equation}
Since $\omega \in A$, it follows that 
$\|h_{k}(\omega) - v^*\|   \leq C(\omega)$ for all $k \geq 0$.
Using this fact and~\plaineqref{eq-thm-korpelevich-as-7} we obtain that for all $k \ge 0$,
\begin{align}\label{eq-thm-korpelevich-as-8}
\|F(h_k(\omega)) \| \leq C(\omega)(K_0 + K_1 \|F(v^*)\|^{\alpha} + K_2 C(\omega)^{\a / (1 - \a)}) + \|F(v^*)\|.
\end{align}
Therefore, the sequence \{$\|F(h_k(\omega))\|$\} is upper bounded by $C_1(\omega) = C(\omega) (K_0 + K_1 \|F(v^*)\|^{\alpha} + K_2 C(\omega)^{\a / (1 - \a)}) + \|F(v^*)\|$. \\
\textbf{Case $\alpha = 1$}. \\
For $\a=1$ by Proposition~\ref{prop-a} we have
\begin{align}\label{eq-thm-korpelevich-as-9}
\|F(h_{k}(\omega)) -  F(v^*)\| &\leq \|h_k(\omega) - v^*\| (L_0+ L_1 \|F(v^*)\|) \exp (L_1 \|h_{k}(\omega) - v^*\|) .
\end{align}
Therefore, for all $k \ge 0$,
\begin{align}\label{eq-thm-korpelevich-as-10}
\|F(h_{k}(\omega))\| &\leq \|F(h_{k}(\omega)) - F(v^*)\| + \|F(v^*)\| \cr
&\leq \|h_k(\omega) - v^*\| (L_0+ L_1 \|F(v^*)\|) \exp (L_1 \|h_{k}(\omega) - v^*\|)+  \|F(v^*)\| . 
\end{align}
Since $ \omega \in A$, we have
$\|h_{k}(\omega) - v^*\| \leq  C(\omega)$ for all $k \geq 0$, 
which when used in~\plaineqref{eq-thm-korpelevich-as-10}, implies that for all $k \ge 0$,
\begin{align}\label{eq-thm-korpelevich-as-11}
\|F(h_{k}(\omega))\| &\leq  \|h_k(\omega) - v^*\| (L_0+ L_1 \|F(v^*)\|) \exp (L_1 \|h_{k}(\omega) - v^*\|) + \|F(v^*)\| \cr
&\leq C(\omega) (L_0+ L_1 \|F(v^*)\|) \exp (L_1 C(\omega) ) +  \|F(v^*)\| .
\end{align}
Hence, the sequence
$\{\|F(h_k(\omega))\|\}$ is upper bounded by $\overline{C}_1(\omega)$, where $\overline{C}_1(\omega) = C(\omega) (L_0+ L_1 \|F(v^*)\|) \exp (L_1 C(\omega) ) + \|F(v^*)\| $. 
Now, for both cases $\a \in (0,1)$ and $\alpha=1$ in~\plaineqref{eq-thm-korpelevich-as-8} and~\plaineqref{eq-thm-korpelevich-as-11}, respectively, we have that  $\|F(h_k(\omega))\|$ is upper bounded by $ \max \{ C_1(\omega), \overline{C}_1(\omega)\}$. Thus 
\[\mathbb{P}( \{F(h_k)\}\text{ is bounded}) = 1.\]

Thus, almost surely we have (i)  $\{F(h_k)\}$ is bounded, (ii) $\sum^{n}_{k=0} \beta_k (\mathbb{I}(A_k) - \mathbb{E}[ \mathbb{I}(A_k) \mid \mathcal{F}_{k-1}])$ converges to $S < \infty$ as $n\to\infty$, and (iii) $\mathbb{E}[\mathbb{I}(A_k) | \mathcal{F}_{k-1}] \geq \frac{1}{2}$. Now, consider $\omega \in \Omega$ such that (i), (ii), and (iii) hold, then in a view of~\eqref{eq-thm-korpelevich-as-6-1} we have
\begin{align}
\label{eq-thm-korpelevich-as-6-2}
\sum_{k=0}^{\infty} \beta_k x_k(\omega) \geq 
\min \left\{1, \frac{1}{\overline{C}_1(\omega) + 2 \sigma}\right\} S(\omega) +
\frac{1}{2} \min \left\{1, \frac{1}{\overline{C}_1(\omega) + 2 \sigma}\right\}  \sum_{k=0}^{\infty} \beta_k  = \infty,
\end{align}
where the last equality comes from $\sum_{k=0}^{\infty} \beta_k = \infty$, which concludes the proof.

\end{proof}

\textbf{Proof of Theorem~\ref{thm-korpelevich-as-converge}}
\begin{proof}
By Lemma~\ref{Lemma-Korpelevich}, we almost surely have
\begin{equation}
\label{eq-thm-korpelevich-as-1}
\sum_{k=0}^{\infty}  \gamma_k \dist^p(u_k, U^*)  \,  < \infty.
\end{equation}
By Lemma~\ref{lemma-stepsizes-nonsummable-korpelevich}, we have $\sum_{k=0}^{\infty} \gamma_k  = \infty$ almost surely, than from~\eqref{eq-thm-korpelevich-as-1} it follows that
\begin{equation}
\label{eq-thm-korpelevich-as-14}
    \liminf_{k \rightarrow \infty} \dist^p(u_k, U^*) = 0 \quad \text{\it a.s.}
\end{equation}
By Lemma~\ref{Lemma-Korpelevich}, the sequence $\{\|h_k - u^*\|\}$ 
converges {\it a.s.}\ 
for any given $u^*\in U^*$.
Thus, the sequence $\{h_k\}$ is bounded {\it a.s.} and, consequently,
it has accumulation points {\it a.s.}
In view of relation~\plaineqref{eq-lemma-korpelevich-2} in Lemma~\ref{Lemma-Korpelevich}, it follows that
\begin{align}
\label{eq-sums-squares}
\lim_{k\to\infty}\|h_k-u_k\|=0\qquad\hbox{a.s.}
\end{align}
Therefore, 
the sequences $\{u_k\}$ and $\{h_k\}$ have the same accumulation points a.s.

Now, let $\{k_i\mid i\ge 1\}$ be a (random) index sequence such that 
\begin{equation}\label{eq-korpelevich-an2}
\lim_{i\to\infty} \dist^p(u_{k_i}, U^*)=\liminf_{k\to\infty}\dist^p(u_k, U^*)=0 \quad {\it a.s.}\end{equation}
Without loss of generality, we may assume that $\{h_{k_i}\}$ is a convergent sequence (for otherwise, we will select a convergent subsequence), and let $\bar u$ be its (random) limit point, i.e.,
\begin{equation}\label{eq-korpelevich-an3}
\lim_{i\to\infty} \|h_{k_i}-\bar u\|=0
\qquad{\it a.s.}\end{equation}
By relation~\plaineqref{eq-lemma-korpelevich-2}, it follows that
$\lim_{k\to\infty}\| h_k-u_k\|=0$ a.s., which in view of the preceding relation implies that
\[\lim_{i\to\infty} \|u_{k_i}-\bar u\|=0
\qquad{\it a.s.}\]
By continuity of the distance function $\dist(\cdot,U^*)$, from relation~\plaineqref{eq-korpelevich-an2} we conclude that $\dist(\bar u,U^*)=0$ {\it a.s.}, which implies that $\bar u\in U^*$ almost surely since the set $U^*$ is closed.
 Since the sequence $\{\|h_k - u^*\|^2\}$ 
converges {\it a.s.}\ 
for any $u^*\in U^*$, it follows that 
$\lim_{k\to\infty}\|h_k - \bar u\|=0$ {\it a.s.} 
By relation~\plaineqref{eq-sums-squares} we conclude that 
$\lim_{k\to\infty}\|u_k - \bar u\|=0$ a.s.
\end{proof}

\subsection{Proof of Lemma~\ref{lemma-korpelevich-EF-bound}}\label{sec-lem43}
\begin{proof}
The choice of parameters $\beta_k$, ensures that $1 - 6\beta_k^2 (K_0 + K_1 + K_3)^2 \geq 1/2$. Then, by taking the expectation in~\plaineqref{eq-lemma-korpelevich} of Lemma~\ref{Lemma-Korpelevich} and using Assumption~\ref{asum-samples}, and definition of $ C_{e}(\beta_k, \alpha) = K_1$ for $\a \in (0,1)$, we obtain 
   \begin{equation}
        \begin{aligned}
        \label{eq-korp-exp-conv-1}
         \mathbb{E}[\|h_{k+1}  - u^*\|^2] &\leq \mathbb{E}[\|h_k - u^*\|^2]  - \frac{1}{2}\mathbb{E}[\|h_k - u_k\|^2] - 2 \mathbb{E}[\g_k \mu \dist^p(u_k, U^*)] \cr
         &\quad + 6 \beta_k^2 (  \sigma^2 + K_1 \sigma^{2 \alpha}). \cr
        \end{aligned}
    \end{equation}

The equation~\plaineqref{eq-korp-exp-conv-1} satisfies the condition of Lemma~\ref{lemma-polyak11-det} with 
\begin{align} \bar v_{k} = \mathbb{E} [\|u_{k}  - u^*\|^2], \quad  \bar a_k = 0, \quad  \bar b_k = 6 \beta_k^2 (  \sigma^2 + K_1 \sigma^{2 \alpha}), \cr
\bar z_k =  2  \mu \mathbb{E}[ \gamma_k \, \dist^p(u_{k}, U^*)] + \frac{1}{2}\mathbb{E}[\|h_k - u_k\|^2].
\end{align}

By Lemma~\ref{lemma-polyak11}, it follows that the sequence $\mathbb{E} [\|h_{k}  - u^*\|^2]$ converges to a non-negative scalar for any $u^*\in U^*$.
Since the sequence $\{\mathbb{E} [\|h_{k}  - u^*\|^2]\}$ 
converges for all $u^*\in U^*$, it follows that the sequence $\{\mathbb{E} [\|h_{k}  - u^*\|^2]\}$ is bounded  for all $u^* \in U^*$. Next, using property of $\a$-symmetric operators, we show that $\mathbb{E}[\|F(h_k)\|]$ is bounded for all $k \geq 0$. Let $v^* \in U^*$ be an arbitrary but fixed solution. Since $\a \leq 1/2$, it holds that
\begin{equation}
\begin{aligned}
\label{eq-korp-exp-conv-2}
\|F(h_k)\| &\leq \|F(h_k) - F(v^*)\| + \|F(v^*)\| \cr
&\leq \|h_k - v^*\| (K_0 + K_1 \|F(v^*)\|^{\alpha} + K_2 \|h_k - v^* \|^{\a / (1 - \a)}) + \|F(v^*)\| .
\end{aligned}
\end{equation}
Taking the expectation, we obtain
\begin{equation}
\begin{aligned}
\label{eq-korp-exp-conv-3}
\mathbb{E}[\|F(h_k)\|] \leq (K_0 + K_1 \|F(v^*)\|^{\alpha}) \mathbb{E}[\|h_k - v^*\|]  + K_2  \mathbb{E}[\|h_k - v^* \|^{1 + \a / (1 - \a)})] + \|F(v^*)\| .
\end{aligned}
\end{equation}
Notice that $\mathbb{E}[\|h_k - v^* \|^{1 + \a / (1 - \a)})] = \mathbb{E}[(\|h_k - v^* \|^2)^{1/2(1 - \a)})]$ and, for $\alpha \leq 1/2$, the quantity $1/2(1 - \a) \leq 1$. Thus, we can apply Jensen inequality for concave function
\[\mathbb{E}[(\|h_k - v^* \|^2)^{1/2(1 - \a)})] \leq \mathbb{E}[\|h_k - v^* \|^2]^{1/2(1 - \a)}.\]
Therefore, using the preceding relation and Jensen inequality for the first term on the RHS of equation~\plaineqref{eq-korp-exp-conv-3}, we obtain
\begin{equation}
\begin{aligned}
\label{eq-korp-exp-conv-4}
\mathbb{E}[\|F(h_k)\|] \leq (K_0 + K_1 \|F(v^*)\|^{\alpha}) \mathbb{E}[\|h_k - v^*\|^2]^{1/2}  + K_2  \mathbb{E}[\|h_k - v^* \|^2]^{1/2(1 - \a)} + \|F(v^*)\| .
\end{aligned}
\end{equation}
Since $\mathbb{E}[\|h_k - v^* \|^2]$ is bounded, it follows that $\mathbb{E}[\|F(h_k)\|]$ is bounded by some constant $C_F > 0$ for all $k \geq 0$.
\end{proof}

\subsection{Proof of Theorem~\ref{thm-korpelevich-rates}}\label{sec-thm44}

\begin{proof}
The choice of the parameters $\beta_k$ ensures that $1 - 6 \beta_k^2 (K_0 + K_1 + K_2)^2 \ge \frac{1}{2}$, then  by letting $u^* = P_{U^*}(h_k)$ in~\plaineqref{eq-lm2-korp-15} in the proof of Lemma~\ref{Lemma-Korpelevich}, with  $C_{e}(\beta_k, \alpha) = K_1$, we get
\begin{equation}
\begin{aligned}
\label{eq-thm-rates-korp-1}
 \|h_{k+1}  - P_{U^*}(h_k)\|^2 &\leq \dist^2(h_k, U^*)  - \frac{1}{2}\|h_k - u_k\|^2 - 2 \g_k \mu \dist^p(u_k, U^*) \cr
 &\quad - 2 \g_k \langle e_k^2, u_k - u^* \rangle  + 3 \beta_k^2 (\|e_k^2\|^2 +  \|e_k^1\|^2 + 2 K_1 \|e_k^1\|^{2 \alpha}). \cr
\end{aligned}
\end{equation}
By the definition of the distance function, we have
\[\dist^2(h_{k+1}, U^*) \leq \|h_{k+1}-P_{U^*}(h_k)\|^2.\]
Thus,
\begin{equation}
\begin{aligned}
\label{eq-thm-rates-korp-2}
 \dist^2(h_{k+1}, U^*) &\leq \dist^2(h_k, U^*)  - \frac{1}{2}\|h_k - u_k\|^2 - 2 \g_k \mu \dist^p(u_k, U^*) \cr
 &\quad - 2 \g_k \langle e_k^2, u_k - u^* \rangle  + 3 \beta_k^2 (\|e_k^2\|^2 +  \|e_k^1\|^2 + 2 K_1 \|e_k^1\|^{2 \alpha}). \cr
\end{aligned}
\end{equation}

Next, we estimate the term $\dist^p(u_k,U^*)$ in~\plaineqref{eq-thm-rates-korp-2}. By the triangle inequality, we have
\[\|h_k - u^*\| \leq \|u_k - h_k\| + \|u_k - u^*\| \qquad \hbox{for all }u^*\in U^*,\]
and by taking the minimum over $u^*\in U^*$ on both sides of the preceding relation, we obtain
\begin{equation}\label{eq-thm-rates-korp-3}
\dist(h_k,U^*)\le \|u_k - h_k\| +\dist(u_k,U^*).
\end{equation}

Applying Lemma~\ref{inequality-hoed-2} with $p > 0$ in equation~\plaineqref{eq-thm-rates-korp-3} yields 
\begin{equation}
\begin{aligned}
\label{eq-thm-rates-korp-4}
\dist^p(h_k,U^*) &\le (\|u_k - h_k\| +\dist(u_k,U^*))^p \cr
&\le 2^{p-1} \|u_k - h_k\|^p + 2^{p-1}   \, \dist^p(u_k,U^*).
\end{aligned}
\end{equation}
Using projection inequality~\plaineqref{eq-proj3}, and stepsizes choice~\plaineqref{stepsizes-korpelevich}, we obtain
\[ \|u_k - h_k\| \leq \|\gamma_{k} \Phi(h_{k}, \xi_k^1)\| \leq 1 . \]
Combining this result with equation~\plaineqref{eq-thm-rates-korp-4}, with $p \geq 2$, we get
\begin{equation}
\begin{aligned}
\label{eq-thm-rates-korp-5}
\dist^p(h_k,U^*) &\le 2^{p-1} \|u_k - h_k\|^{2 +(p-2)} + 2^{p-1} \dist^p(u_k,U^*) \cr
&\le 2^{p-1} \|u_k - h_k\|^2 + 2^{p-1}   \, \dist^p(u_k,U^*).
\end{aligned}
\end{equation}
By dividing the relation in~\eqref{eq-thm-rates-korp-5} with $2^{p-1}$
and by rearranging the terms, we obtain the following relation
\begin{equation}
\begin{aligned}
\label{eq-thm-rates-korp-6}
- \dist^p(u_k,U^*) &\leq \|u_k - h_k\|^2 -  2^{1-p}  \, \dist^p(h_k,U^*) .
\end{aligned}
\end{equation}
Combining the preceding inequality with \plaineqref{eq-thm-rates-korp-2}, we find that for any $k \geq 0$,
\begin{equation}
    \begin{aligned}
    \label{eq-thm-rates-korp-7}
     \dist^2(h_{k+1}, U^*) &\leq \dist^2(h_k, U^*) - 2^{2-p} \mu \gamma_k \dist^p(h_k, U^*) - \frac{1}{2} \|u_k - h_{k}\|^2 + 2 \mu \gamma_k \, \|u_k -  h_{k}\|^2  \cr
     &\quad - 2 \g_k \langle e_k^2, u_k - u^* \rangle  + 3 \beta_k^2 (\|e_k^2\|^2 +  \|e_k^1\|^2 + 2 K_1 \|e_k^1\|^{2 \alpha}).
    \end{aligned}
\end{equation}
By the choice of $\beta_k$, we have 
$\beta_k= \dfrac{2}{a(\frac{2d}{a} + k)}$, where $a = \mu \min \left\{1, \frac{1}{2 (C_F+ \sigma)}\right\}$ and $d\ge 4\mu$. Thus, for all $k\ge 0$,
\[\beta_k\le \frac{1}{d}\le \frac{1}{4\mu}\qquad\implies\qquad 2\mu\beta_k\le \frac{1}{2}.\]
By the definition of the stepsize $\gamma_k$, we always have $\gamma_k\le\beta_k$. 
Therefore, $2 \mu\g_k \leq  2 \mu\beta_k \leq \frac{1}{2}$ for all $k\ge0$,  thus implying that
\begin{align}
\label{eq-thm-rates-korp-8}
- \frac{1}{2} \|u_k - h_{k}\|^2 + 2 \mu \gamma_k \, \|u_k -  h_{k}\|^2  \leq 0.
\end{align}

Using the stochastic properties of $\xi_k$ imposed by Assumption~\ref{asum-samples}, we have for all $k \geq 0$,
\begin{align}
&\mathbb{E}[ \mathbb{E}[\gamma_k \mathbb{E}[ \langle e_k^2, u_k - u^* \rangle \mid \mathcal{F}_{k-1} \cup \{\xi_k^1\}]  \mid \mathcal{F}_{k-1}]] = 0, \cr &\mathbb{E}[\mathbb{E}[ \|e_k^2\|^2 \mid \mathcal{F}_{k-1} \cup \{\xi_k^1\}]] \leq \sigma^2,\qquad 
\mathbb{E}[\mathbb{E}[ \|e_k^1\|^2 \mid  \mathcal{F}_{k-1}]] \leq \sigma^2.
\end{align}

Moreover, since $\a \leq 1$ then the conditional expectation $\mathbb{E}[\|e_k^1\|^{2\a}|\mathcal{F}_{k-1}]$ is defined, and by Jensen inequality $\mathbb{E}[\|e_k^1\|^{2\a}|\mathcal{F}_{k-1}] \le \sigma^{2 \alpha}$ for all $k \ge 0$. Thus, by taking the total expectation  in relation~\plaineqref{eq-thm-rates-korp-7} and using an estimate from~\plaineqref{eq-thm-rates-korp-8}, we obtain 
for all $u^*\in U^*$ and for all $k\ge 0$,
\begin{equation}
    \begin{aligned}
    \label{eq-thm-rates-korp-9}
     \mathbb{E}[\dist^2(h_{k+1}, U^*)] &\leq \mathbb{E}[\dist^2(h_k, U^*)] - 2^{2-p} \mu \mathbb{E}[\gamma_k \dist^p(h_k, U^*)] + 6 \beta_k^2 (\sigma^2 + K_1 \sigma^{2 \alpha}).
    \end{aligned}
\end{equation}
The equation ~\plaineqref{eq-thm-rates-korp-9} is similar to equation ~\plaineqref{eq-thm-grad-rates-3} in the proof of Theorem~\ref{thm-proj-rates}, with the same stepsize structure. Thus, by following the same arguments from equations~\plaineqref{eq-thm-grad-rates-3} to equation~\plaineqref{eq-thm-grad-rates-8} in the proof of Theorem~\ref{thm-proj-rates}, we arrive at 
\begin{equation}
    \begin{aligned}
    \label{eq-thm-rates-korp-10}
     \mathbb{E}[\dist^2(h_{k+1}, U^*)] &\leq \mathbb{E}[\dist^2(h_k, U^*)] - 2^{2-p} \mu \beta_k \min \left\{1, \frac{1}{2 (C_F + \sigma)}\right\}\mathbb{E}[\dist^p(h_k, U^*)]  \cr
     &\quad + 6 \beta_k^2 (\sigma^2 + K_1 \sigma^{2 \alpha}),
    \end{aligned}
\end{equation}
where $C_F$ is an upperbound on $\mathbb{E}[\|F(h_k)\|]$  from the statement of Lemma~\ref{lemma-korpelevich-EF-bound}. Now let $D_k = \mathbb{E}[\dist^2(h_k, U^*)]$, and consider two cases $p = 2$ and $p > 2$.

\textbf{Case $p=2$}.

We note that by the definition of $a = \mu \min \left\{1, \frac{1}{2 (C_F+ \sigma)}\right\}$  and $d$, we have that $d\ge 4\mu$ and $\mu\ge a$, implying that $d\ge a$. Hence, for $p=2$, relation~\plaineqref{eq-thm-rates-korp-10} satisfies the conditions of Lemma~\ref{Lemma7-stich} with the following identification
\begin{align}
r_{k} = D_k, \quad a =   \mu \min \left\{1, \frac{1}{2 (C_F+ \sigma)}\right\},  \quad \alpha_k = \beta_k, \quad s_k = 0, \quad c = 6 (\sigma^2 + K_1 \sigma^{2 \alpha}).
\end{align}
Therefore, for the choice $\beta_k= \dfrac{2}{a(\frac{2d}{a} + k)}$, we get the following convergence rate for all $k \geq 1$,
\begin{align}
 D_{k+1} \leq  \dfrac{8  d^2 D_0}{a^2 k^2} + \frac{12(\sigma^2 + K_1 \sigma^{2 \alpha})}{a^2 k}.
\end{align}

\textbf{Case $p \geq 2$}.

When $p \geq 2$, by applying telescoping sum to inequality~\plaineqref{eq-thm-rates-korp-10} and rearranging the terms we obtain
\begin{align}\label{eq-thm-korp-rates-9-new}
\mathbb{E}[  2^{2 -p } a \sum_{t=0}^k   \beta_k \dist^p(h_k, U^*)] 
&\le D_0  - D_{k+1} + 6  (\sigma^2 + K_1 \sigma^{2 \alpha}) \sum_{t=0}^k \beta_k^2.  
\end{align}
Since $p \geq 2$, the function $\dist^p(\cdot, U^*)$ is convex, thus by defining $\bar{u}_k = (\sum_{t=0}^k \beta_k)^{-1} \sum_{t=0}^k \beta_k h_t$ and applying Jensen inequality be obtain
\[ (\sum_{t=0}^k \beta_k ) \mathbb{E}[\dist^p (\bar{h}_k, U^*)] \leq \mathbb{E}[ \sum_{t=0}^k   \beta_k \dist^p(h_k, U^*)] .\]
Since $p \geq 2$, by applying Jensen inequality one more time, we obtain
\[ (\bar{D}_k)^{p/2} = \left(\mathbb{E}[\dist^2(\bar{h}_k, U^*)]\right)^{p/2} \leq  \mathbb{E}\left[\left(\dist^2(\bar{h}_k, U^*)\right)^{p/2}\right] = \mathbb{E}[\dist^p(\bar{h}_k, U^*)]
.\]
\begin{align}\label{eq-thm-rates-korp-13}
(\bar{D}_k)^{p/2}  \sum_{t=0}^{k} \beta_t \leq \sum_{t=0}^{k} \beta_t   (D_t)^{p/2}
&\le \dfrac{D_0 - D_{k+1}  +  6 (\sigma^2 + K_1 \sigma^{2 \alpha}) 
 \sum_{t=0}^k \beta_t^2}{ 2^{2-p}a}.
\end{align}

Now, we use the choice for $\beta_k$, i.e., $\beta_k = \frac{b}{(k+1)^q}$, where $ 0< b < \frac{1}{2\sqrt{3}(K_0 + K_1 + K_2)}$ and $1/2 < q <1$. 
Then, the sequence $\{\beta_k\}$ satisfies the conditions of Lemma~\ref{lemma-korpelevich-EF-bound}. Furthermore, by Lemma~\ref{lemma-kq-series-rates}, we have that  for all $k\ge1$,
\begin{align}
\label{eq-thm-rates-korp-14}
\sum^{k}_{t = 0} \beta_t \geq \frac{b}{1 - q} ((k+1)^{1 - q} - 2^{1 - q}),  \qquad\qquad
\sum^{k}_{t = 0} \beta_t^2 \leq \frac{b^2}{2 q -1}.
\end{align}
Combining equations~\plaineqref{eq-thm-rates-korp-13} and ~\plaineqref{eq-thm-rates-korp-14}, and omitting $D_{k+1}$, we obtain for all $k\ge 1$,
\begin{align}\label{eq-thm-rates-korp-15}
(\bar{D}_k)^{p/2}
&\le \dfrac{ 2^{p- 2}(1 -q) \left(D_0  + 6b^2 (\sigma^2 + K_1 \sigma^{2 \alpha})  (2 \sigma^2 +1 ) / (2 q -1)\right)}{  ab \left((k+1)^{1 - q} - 2^{1 - q}\right)}.
\end{align}
Raising both sides of the preceding inequality in power $2/p$, we have that for all $k\ge1$,
\begin{align}\label{eq-thm-rates-korp-16}
\bar{D}_k
&\le \dfrac{ 2^{2(p-2) / p}(1 - q)^{2/p} \left(D_0  + 6b^2 (\sigma^2 + K_1 \sigma^{2 \alpha})  (2 \sigma^2 +1 ) / (2 q -1)\right))^{2/p}}{  (ab)^{2/p}\, \left((k+1)^{1 - q} - 2^{1 - q}\right)^{2/p}}.
\end{align}
\end{proof}


\end{document}